%% file: d-dim_Dirac_arxiv_v3.tex
\documentclass[reqno]{amsart}
\usepackage[tmargin=28truemm,lmargin=26truemm,rmargin=26truemm]{geometry}
\usepackage{gnuplot-lua-tikz}
\usepackage{xcolor}
\usepackage[skip=2pt]{caption}
\usepackage[subrefformat=parens,labelformat=parens,labelsep=period]{subcaption}

\usepackage{adjustbox}

\newcommand{\FitGraphs}[2][.48\linewidth]{%
	\adjustbox{width=#1}{#2}%
}

\usepackage{amsmath}
\usepackage{amssymb}
\usepackage{mathtools}
\mathtoolsset{showonlyrefs, showmanualtags}
\usepackage{aligned-overset}

\usepackage{bm}

\usepackage{amsthm}
\usepackage{mleftright}
\usepackage{enumitem}

\newlist{eqenumerate}{enumerate}{1}
\makeatletter
\setlist[eqenumerate,1]{
	label=\textup{(\theequation)},
	ref=\theequation,
	before=\let\c@eqenumeratei\c@equation
}
\makeatother

\newtheorem{theorem}{Theorem}[section]
\newtheorem{corollary}[theorem]{Corollary}
\newtheorem{lemma}[theorem]{Lemma}
\newtheorem{proposition}[theorem]{Proposition}

\numberwithin{equation}{section}

\newtheorem{quotetheorem}[theorem]{Theorem}

\theoremstyle{remark}

\newtheorem{remark}{Remark}[section]

\newtheorem{example}[theorem]{Example}

\makeatletter
\renewcommand*\env@matrix[1][\arraystretch]{%
	\edef\arraystretch{#1}%
	\hskip -\arraycolsep
	\let\@ifnextchar\new@ifnextchar
	\array{*\c@MaxMatrixCols c}}
\makeatother

\renewcommand{\S}{\mathbb{S}}
\DeclareMathOperator{\supp}{supp}
\newcommand{\Laplacian}{\Delta}

\DeclareMathOperator*{\esssup}{ess\,sup}

\newcommand{\dimN}{N}

\newcommand{\R}{\mathbb{R}}
\newcommand{\N}{\mathbb{N}}
\newcommand{\C}{\mathbb{C}}
\newcommand{\Z}{\mathbb{Z}}

\newcommand{\D}{D}

\newcommand{\positiveR}{\R_{>0}}

\newcommand{\Fourier}{\mathcal{F}}

\newcommand*{\variabledot}{\makebox[1ex]{\textbf{$\cdot$}}}

\newcommand{\sumk}{\sum_{k=0}^{\infty}}
\newcommand{\sumn}{\sum_{n}}

\newcommand{\Fwd}[1]{\Fourier_w^{(#1)}}
\newcommand{\Fwthetad}[2]{\Fourier_{w_{#1}}^{(#2)}}

\newcommand{\abs}[1]{\lvert#1\rvert} 
\newcommand{\Abs}[1]{\mleft \lvert #1 \mright \rvert}
\newcommand{\norm}[1]{\lVert#1\rVert} 
 
\newcommand{\set}[2]{\{ \, #1 : #2 \, \} } 
 
\newcommand{\measure}[1]{ \lvert#1\rvert }
\newcommand{\innerproduct}[2]{\langle #1, #2 \rangle}

\newcommand{\SchrodingerS}{S}
\newcommand{\DiracS}{\widetilde{S}}

\newcommand{\phim}[1]{\phi_{#1}}
\newcommand{\Diracphim}[1]{A_{#1}}

\newcommand{\Diracopm}[1]{ H_{#1} }

\newcommand{\const}{\bm{C}}

\newcommand{\Sconstd}[1]{\const_{\textup{S}}^{(#1)}}
\newcommand{\Sconstwpd}[3]{\const_{\textup{S}}^{(#3)}(#1, #2)}

\newcommand{\RSconstwpmd}[4]{\const^{(#4)}_{\textup{RS}, #3}(#1, #2)}

\newcommand{\Diracconstmd}[2]{\const_{\textup{D}, #1}^{(#2)}}
\newcommand{\Diracconstwpmd}[4]{\const_{\textup{D}, #3}^{(#4)}(#1, #2)}

\newcommand{\HPk}[1]{ \mathcal{H}_{#1} } 

\newcommand{\Pkd}[2]{P_{#1}^{(#2)}} 

\newcommand{\adjoint}[1]{{#1}^*}

\newcommand{\gammaj}[1]{\gamma_{#1}}

\newcommand{\sigmaj}[1]{\sigma_{#1}}

\newcommand{\matrixIN}[1]{I_{#1}}

\newcommand{\Lambdak}[1]{\Lambda_{#1, \mu}}
\newcommand{\DiracLambdak}[1]{\widetilde{\Lambda}_{#1, \mu}}

\newcommand{\pkn}[2]{p_{#1}^{#2}}

\newcommand{\fkn}[2]{f_{#1}^{#2}}

\newcommand{\lambdak}[1]{\lambda_{#1}}

\newcommand{\lambdakd}[2]{\lambda_{#1}}

\newcommand{\lambdasupd}[1]{\lambdakd{*}{#1}}

\newcommand{\Diraclambda}{\widetilde{\lambda}}
\newcommand{\Diraclambdak}[1]{\Diraclambda_{#1}}

\newcommand{\Diraclambdasup}{\Diraclambdak{*}}

\newcommand{\Wk}[1]{W_{#1, \mu}}

\newcommand{\BesselI}[1]{I_{#1}}
\newcommand{\BesselJ}[1]{J_{#1}}
\newcommand{\BesselK}[1]{K_{#1}}

\newcommand{\EllipticE}{E}

\newcommand{\fd}[1]{f_{#1}}
\newcommand{\gd}[1]{g}

\DeclareFontFamily{U}{matha}{\hyphenchar\font45}
\DeclareFontShape{U}{matha}{m}{n}{
	<5> <6> <7> <8> <9> <10> gen * matha
	<10.95> matha10 <12> <14.4> <17.28> <20.74> <24.88> matha12
}{}
\DeclareSymbolFont{matha}{U}{matha}{m}{n}
\DeclareFontFamily{U}{mathx}{\hyphenchar\font45}
\DeclareFontShape{U}{mathx}{m}{n}{
	<5> <6> <7> <8> <9> <10>
	<10.95> <12> <14.4> <17.28> <20.74> <24.88>
	mathx10
}{}
\DeclareSymbolFont{mathx}{U}{mathx}{m}{n}

\DeclareMathSymbol{\obot}         {2}{matha}{"6B}
\DeclareMathSymbol{\bigobot}       {1}{mathx}{"CB}

\newcommand{\Vk}[1]{V_{#1}}

\newcommand{\Vkm}[1]{V_{#1, \mu}}

\newcommand{\Vkmm}[1]{V_{#1, -\mu}}

\newcommand{\Qmn}{Q_{\nu}}

\newcommand{\Pn}{P_{\nu}}
\newcommand{\eknm}[2]{e_{#1, \mu}^{#2}}
\newcommand{\Eknm}[2]{E_{#1, \mu}^{#2}}
\newcommand{\fknm}[2]{f_{#1, \mu}^{#2}}

\usepackage[utf8]{inputenc}
\usepackage[T1]{fontenc}
\usepackage{lmodern}

\usepackage{mmacells}
\mmaDefineMathReplacement[≤]{<=}{\leq}
\mmaDefineMathReplacement[≥]{>=}{\geq}
\mmaDefineMathReplacement[≠]{!=}{\neq}
\mmaDefineMathReplacement[→]{->}{\to}[2]
\mmaDefineMathReplacement[⧴]{:>}{:\hspace{-.2em}\to}[2]
\mmaDefineMathReplacement{∉}{\notin}
\mmaDefineMathReplacement{∞}{\infty}
\mmaDefineMathReplacement{𝕕}{\mathbbm{d}}

\usepackage[numbers,sort]{natbib}
\providecommand*{\doi}[1]{doi:\href{https://doi.org/#1}{#1}}
\renewcommand*{\MR}[1]{\href{https://mathscinet.ams.org/mathscinet-getitem?mr=#1}{MR#1}}

\usepackage{hyperref}
\hypersetup{
	setpagesize=false,
	bookmarksnumbered=true,
	bookmarksopen=true,
	colorlinks=true,
	linkcolor=blue,
	citecolor=blue,
}

\title[Optimal constants of smoothing estimates for the Dirac equation]{Optimal constants of smoothing estimates for the Dirac equation in arbitrary dimensions}
\date{}

\author{Soichiro Suzuki}
\address[Soichiro Suzuki]{Department of Mathematics, Chuo University, 1-13-27, Kasuga, Bunkyo-ku, Tokyo, 112-8551, Japan}
\email{soichiro.suzuki.m18020a@gmail.com}
\thanks{The author was supported by Japan Society for the Promotion of Science (JSPS) KAKENHI Grant Number JP23KJ1939.}

\subjclass[2020]{33C55, 35B65, 35Q41, 42B10}

\keywords{Dirac equations; smoothing estimates; optimal constants; spherical harmonics.}

\begin{document}
	\begin{abstract}
		We give optimal constants of smoothing estimates for the $d$-dimensional free Dirac equation for any $d \geq 2$.
		Our main abstract theorem shows that the optimal constant $C$ of the smoothing estimate associated with a spatial weight $w$ and smoothing function $\psi$
		is given by $C = \sup_{k \in \mathbb{N}} \sup_{r > 0} \widetilde{\lambda}_k(r)$, where $\{ \widetilde{\lambda}_k \}$ is a certain sequence of functions defined via integral formulae involving $(w, \psi)$. 
		This is an analogue of a similar result for Schr\"{o}dinger equations given by \citeauthor*{BSS2015} (\citeyear{BSS2015}), and also extends previous results of \citeauthor*{Iko2022} (\citeyear{Iko2022}) and \citeauthor*{IkS2025} (\citeyear{IkS2025}) for $d=2, 3$ to arbitrary dimensions $d \geq 2$.
		In order to prove this, we establish a modified version of the spherical harmonics decomposition of $L^2(\mathbb{S}^{d-1})$, which is well suited to the Dirac operator and allows us to find optimal constants.
		Furthermore, using our abstract theorem, we give explicit values of optimal constants associated with typical examples of $(w, \psi)$.
		As it turns out, optimal constants for Dirac equations can be written explicitly in many cases, even when it is impossible for Schr\"{o}dinger equations.
		In particular, the classical result of \citeauthor{Sim1992} (\citeyear{Sim1992}) for Schr\"{o}dinger equations, which holds when $d \geq 3$ but fails when $d=2$, is true for Dirac equations whenever $d \geq 3$ and remains valid for the massless two-dimensional case.
	\end{abstract}
	\maketitle
	\section{Introduction}
	The Kato smoothing estimate is one of the fundamental results in the study of dispersive equations such as Schr\"odinger equations and Dirac equations, which was originally developed by \citet{Kat1966} in an abstract setting. 
	First, we consider the free Schr\"odinger equation on $\R^d$:
	\begin{equation} \label{eq:Schrodinger}
		\begin{cases}
			i \partial_t u(x, t) = - \Laplacian u(x, t) , & (x, t) \in \R^d \times \R , \\
			u(x, 0) = u_0(x) , & x \in \R^d .
		\end{cases}
	\end{equation}
	In this case, the space-time global smoothing estimate is expressed as
	\begin{equation} \label{eq:smoothing Schrodinger}
		\int_{t \in \R}\int_{x \in \R^d} w(\abs{x}) \abs{\psi(\abs{\D}) e^{it\Laplacian} u_0(x)}^2 \, dx \, dt \leq C\norm{u_0}^2_{L^2(\R^d)} , 
	\end{equation}
	or equivalently, 
	\begin{equation} 
	\frac{1}{(2\pi)^d}\int_{t \in \R}\int_{x \in \R^d} w(\abs{x}) \Abs{ \int_{\xi \in \R^d} e^{ix \cdot \xi} \psi(\abs{\xi}) e^{ - i t \abs{\xi}^2 }  \widehat{u_0}(\xi) \, d\xi }^2 \, dx \, dt \leq  C \norm{\widehat{u_0}}^2_{L^2(\R^d)} ,
	\end{equation}
	where $w, \psi \colon (0, \infty) \to [0, \infty)$ are some given functions, which are called a spatial weight and a smoothing function, respectively.
	For example, this inequality holds in the following cases:
	\begin{alignat}{8}
		&d \geq 3,  \quad && s \geq 2 , \quad && ( w(r), \psi(r) ) = ( &&(1+r^2)^{-s/2}, \quad&&(1+r^2)^{1/4} &&) ,
		\label{eq:type A} \tag{A1}\\
		&d = 2,  \quad && s > 2 , \quad && ( w(r), \psi(r) ) = ( &&(1+r^2)^{-s/2}, \quad&&(1+r^2)^{1/4} &&) ,
		\label{eq:type A 2D} \tag{A2}\\
		&d \geq 2,  \quad && 1 < s < d, \quad&& ( w(r), \psi(r) ) = ( &&r^{-s}, \quad&&r^{(2-s)/2} &&) , 
		\label{eq:type B} \tag{B}\\
		&d \geq 2,  \quad&& s > 1 , \quad&& ( w(r), \psi(r) ) = ( &&(1+r^2)^{-s/2}, \quad&&r^{1/2} &&) .
		\label{eq:type C} \tag{C}
	\end{alignat}
	The cases \eqref{eq:type A} and \eqref{eq:type A 2D} are given by \citet[Theorem 2]{KY1989} and \citet[Proposition 2]{BK1992}, respectively.
	The case \eqref{eq:type B} for $d=2$, $1 < s < 2$ and for $d \geq 3$, $1 < s \leq 2$ is by \citet[Theorem 1, Remarks (a)]{KY1989}.
	See \citet[Theorem 3]{Wat1991} and \citet[Theorem 1.1]{Sug1998} for the case \eqref{eq:type B} with the full range $1 < s < d$.
	The case \eqref{eq:type C} is by \citet[Theorem 1.(b)]{BK1992} ($d \geq 3$) and \citet[Theorem 1.1]{Chi2002} ($d \geq 2$).
	Furthermore, the ranges $s \geq 2$ in \eqref{eq:type A}, $s > 2$ in \eqref{eq:type A 2D}, $1 < s < d$ in \eqref{eq:type B}, and $s > 1$ in \eqref{eq:type C} are sharp; 
	see \citet[Theorem 2.1.(b), Theorem 2.2.(b)]{Wal1999} for \eqref{eq:type A} and \eqref{eq:type A 2D}, 
	\citet[Theorem 2]{Vil2001} for \eqref{eq:type B}, and \citet[Theorem 2.14.(b)]{Wal2000} for \eqref{eq:type C}.  
	
	In what follows, we write $\Sconstd{d} = \Sconstwpd{w}{\psi}{d}$ for the optimal constant for the inequality \eqref{eq:smoothing Schrodinger}, in other words,
	\begin{equation} \label{eq:optimal Schrodinger}
		\Sconstwpd{w}{\psi}{d} \coloneqq \sup_{\substack{ u_0 \in L^2(\R^d) \\ 
				\norm{u_0}_{L^2} = 1}} 
		\int_{t \in \R}\int_{x \in \R^d} w(\abs{x}) \abs{\psi(\abs{\D}) e^{ it \Laplacian} u_0(x)}^2 \, dx \, dt. 
	\end{equation}
	Since we are interested in optimal constants, here we clarify that the Fourier transform in this paper is defined by  
	\begin{equation}
		\mathcal{F}f(\xi) = \widehat{f}(\xi) \coloneqq \frac{1}{(2\pi)^{d/2}} \int_{x \in \R^d}f(x)e^{-ix\cdot\xi} \, dx ,
	\end{equation}
	so that we have
	\begin{equation}
		\norm{\widehat{f}}_{L^2(\R^d)}^2 = \norm{f}_{L^2(\R^d)}^2 .
	\end{equation}
	In addition, we say a function $u_0 \in L^2(\R^d) \setminus \{0\}$ is an extremiser if the equality
	\begin{equation} \label{eq:extremiser Schrodinger}
		\int_{t \in \R}\int_{x \in \R^d} w(\abs{x}) \abs{\psi(\abs{\D}) e^{ it \Laplacian} u_0(x)}^2 \, dx \, dt = \Sconstd{d} \norm{u_0}^2_{L^2(\R^d)} 
	\end{equation}
	holds.
	In \citeyear{BSS2015}, \citet*{BSS2015} proved the following abstract result (Theorem \ref{thm:Schrodinger}) by using the so-called spherical harmonics decomposition (Theorem \ref{thm:spherical harmonics decomposition}) and the Funk--Hecke theorem (Theorem \ref{thm:Funk-Hecke}):
	\begin{quotetheorem}[{\cite[Theorem 1.1]{BSS2015}}] \label{thm:Schrodinger}
		Let $d \geq 2$.
		Then we have 
		\begin{equation}
		 \Sconstwpd{w}{\psi}{d} = \lambdasupd{d} \coloneqq \sup_{k\in\N} \esssup_{r>0} \lambdakd{k}{d}(r), 
		\end{equation}
		where
		\begin{align}
			\lambdakd{k}{d}(r)
			&=
			\pi ( \psi(r) )^2 \int_0^\infty t w(t) ( \BesselJ{k + d/2 - 1}(r t) )^2 \, dt 
			\label{eq:lambdak 1} \\ 
			&= \frac{ \pi^{1/2} }{ 2^{d/2 - 1} \Gamma( (d-1)/2 ) } r^{d-2} ( \psi(r) )^2 \int_{-1}^{1} \Fwd{d}( r \sqrt{ 2(1-t) } ) \Pkd{k}{d}(t) (1-t^2)^{(d-3)/2} \, dt . 			\label{eq:lambdak 2} \noeqref{eq:lambdak 2}
		\end{align}
		Here $\BesselJ{\nu}$ is the Bessel function of order $\nu$, $\Fwd{d}$ is the Fourier transform of $w$ as a radial function on $\R^d$,
		\begin{align}
			\Fwd{d}(\abs{\xi}) 
			&= (2 \pi )^{-d/2} \int_{x \in \R^d} w( \abs{x} ) e^{- i x \cdot \xi} \, dx \\
			&=  \abs{\xi}^{-d/2 + 1} \int_0^\infty w(t) t^{d/2} \BesselJ{d/2 - 1}( \abs{\xi} t) \, dt  , 
		\end{align}	
		and $\Pkd{k}{d}$ is the Legendre polynomial of degree $k$ in $d$ dimensions, which may be defined in a number of ways, for example, 
		by the following Rodrigues' formula, 
		\begin{equation}
			\Pkd{k}{d}(t) (1-t^2)^{(d-3)/2} = (-1)^k \frac{ \Gamma( (d-1)/2 ) }{ 2^k \Gamma( k + (d-1)/2 ) } \mleft( \frac{d}{dt} \mright)^k (1-t^2)^{k + (d-3)/2} , 
		\end{equation}
		or by the following recurrence relation,
		\begin{equation}
			\begin{cases}
				\Pkd{0}{d}(t) = 1 , \\
				\Pkd{1}{d}(t) = t , \\
				(k+d-3) \Pkd{k}{d}(t) = (2k+d-4) t \Pkd{k-1}{d}(t) - (k-1) \Pkd{k-2}{d}(t) .
			\end{cases}
		\end{equation}
		Furthermore, extremisers exist if and only if there exists $k \in \N$ such that the Lebesgue measure of
		\begin{equation}
			\set{ r > 0 }{ \lambdakd{k}{d}(r) = \lambdasupd{d} }
		\end{equation}
		is non-zero.
	\end{quotetheorem}
	See Proposition \ref{prop:properties of lambdak} for some basic properties of $\lambdak{k}$.
	\begin{remark}
		To be precise, we assume that $w, \psi \colon (0, \infty) \to [0, \infty)$ are sufficiently nice to justify some formal calculation. See Proof of Theorem \ref{thm:Dirac} for details.
	\end{remark}
	\begin{remark}
		Previous to \citet{BSS2015}, \citet[Theorem 4.1]{Wal2002} established Theorem \ref{thm:Schrodinger} except for the second expression of $\lambdakd{k}{d}$ \eqref{eq:lambdak 2}.
	\end{remark}
		\citet*{BS2017} and \citet*{BSS2015} determined the explicit values of the optimal constant $\Sconstwpd{w}{\psi}{d}$ and the existence of extremisers in the cases \eqref{eq:type A}, \eqref{eq:type B}, \eqref{eq:type C} by using Theorem \ref{thm:Schrodinger}, which extend results by \citet[(2), (3)]{Sim1992} and \citet[Corollary 4]{Wat1991}.
	\begin{quotetheorem}[{\cite[Theorems 1.6, 1.7]{BS2017}, \cite[Theorem 1.4]{BSS2015}}] \label{thm:Schrodinger explicit value} \phantom{a}
		\begin{description}
			\item[\rm{\cite[Theorem 1.7]{BS2017}}]
			In the case \eqref{eq:type A} with $s = 2$, we have
			\begin{equation}
				\Sconstd{d} = \begin{cases}
					\pi , & d = 3 , \\
					\pi \sup_{r>0} (1+r^2)^{1/2} \BesselI{1}(r) \BesselK{1}(r) , & d = 4 , \\
					\pi / 2 , & d \geq 5 ,
				\end{cases} 
			\end{equation}
			and there are no extremisers.
			We note that 
			\begin{equation}
				\sup_{r>0} (1+r^2)^{1/2} \BesselI{d/2-1}(r) \BesselK{d/2-1}(r)
				= \begin{cases}
					1 , & d = 3 , \\
					0.50239\ldots , & d = 4 , \\
					1/2 , & d \geq 5 .
				\end{cases}
			\end{equation}
			See Figure \ref{fig:graph of lambda0 type A} for the graph of $(1+r^2)^{1/2} \BesselI{d/2-1}(r) \BesselK{d/2-1}(r)$ in the cases $d = 3, 4, 5$.
			\item[\rm{\cite[Theorem 1.6]{BS2017}}]
			In the case \eqref{eq:type B}, we have
			\begin{equation}
				\Sconstd{d} = \pi^{1/2} \frac{  \Gamma((s-1)/2) \Gamma((d-s)/2) }{ 2 \Gamma(s/2) \Gamma((d+s)/2-1) } , 
			\end{equation}
			and $u_0 \in L^2(\R^d) \setminus \{0\}$ is an extremiser if and only if it is radial.
			\item[\rm{\cite[Theorem 1.4]{BSS2015}}]
			Let $d \geq 3$, $w \in L^1(\positiveR) \setminus \{0\}$, and assume that $\Fwd{d}$ is non-negative.
			Then we have
			\begin{equation}
				\Sconstwpd{w}{r^{1/2}}{d} = \norm{w}_{L^1(\positiveR)} ,
			\end{equation}
			and there are no extremisers except when $d = 3$ and $\Fwd{3}$ is compactly supported (see Remark \ref{remark:type C 3D}).
			In particular, in the case \eqref{eq:type C}, we have
			\begin{equation}
				\Sconstd{d} = \pi^{1/2} \frac{ \Gamma( (s-1)/2 ) }{ 2 \Gamma(s/2) } 
			\end{equation}
			for any $d \geq 3$. 
			Note that this is no longer true in the case $d = 2$. See Theorem \ref{thm:type C 2D Schrodinger} for details.
		\end{description}
	\end{quotetheorem}
	Our aim in this paper is to establish analogues of Theorems \ref{thm:Schrodinger} and \ref{thm:Schrodinger explicit value} for the free Dirac equation on $\R^d$.
	Here we recall the definition of the Dirac equation.
	Let $d \geq 2$, $\dimN \coloneqq 2^{\lfloor(d+1)/2\rfloor}$, 
	and $\{ \gammaj{j} \}_{1 \leq j \leq d+1}$ 
	be a family of $\dimN \times \dimN$ Hermitian matrices satisfying the anti-commutation relation $\gammaj{j} \gammaj{k}  + \gammaj{k} \gammaj{j} = 2\delta_{jk} \matrixIN{\dimN}$. Such matrices are usually called the gamma matrices or the Dirac matrices. 
	The $d$-dimensional Dirac equation with mass $m \geq 0$ is given by 
	\begin{equation} \label{eq:Dirac}
		\begin{cases}
			i\partial_{t} u(x,t) = \Diracopm{m} u(x, t), & (x, t) \in \R^d \times \R , \\
			u(x,0) = u_0(x) , & x \in \R^d .
		\end{cases}
	\end{equation}
	Here $u$ and $u_0$ are $\C^\dimN$-valued, and
	the Dirac operator $\Diracopm{m}$ is the Fourier multiplier operator whose symbol is
	\begin{equation}
		\Diracphim{m}(\xi) \coloneqq \sum_{j=1}^{d} \xi_j \gammaj{j} + m\gammaj{d+1} .
	\end{equation}
	Note that $\Diracphim{m}$ satisfies
	\begin{equation}
		\Diracphim{m}(\xi)^2 = (\abs{\xi}^2 + m^2) \matrixIN{\dimN} ,
	\end{equation}
	so that the Dirac operator satisfies
	\begin{equation}
		\Diracopm{m}^2 = ( - \Laplacian + m^2 ) \matrixIN{\dimN} .
	\end{equation}
	For simplicity, hereinafter we write
	\begin{equation}
		\phim{m}(r) \coloneqq (r^2 + m^2)^{1/2} .
	\end{equation}
	Since $\Diracopm{m}^2 = ( - \Laplacian + m^2 ) \matrixIN{\dimN}$, it is natural to expect that the smoothing estimate for the Dirac equation is similar to that for the relativistic Schr\"{o}dinger equation:
	\begin{equation} \label{eq:relativistic Schrodinger}
		\begin{cases}
			i \partial_t u(x, t) = ( - \Laplacian + m^2 )^{1/2} u(x, t) , & (x, t) \in \R^d \times \R , \\
			u(x, 0) = u_0(x) , & x \in \R^d .
		\end{cases}
	\end{equation}
	On the other hand, in the case of the relativistic Schr\"{o}dinger equation, we can show that the smoothing estimate is actually equivalent to that for the Schr\"{o}dinger equation thanks to the following \emph{comparison principle} established by \citet*{BRS2020}:
	\begin{quotetheorem}[{\cite[Corollary 4.5]{BRS2020}}] \label{thm:comparison principle} 
		Let $H$ be a self-adjoint operator on $L^2(\R^d, \C^n)$ satisfying $\sigma(H) = \sigma_{\text{ac}}(H)$, and
		assume that 
		\begin{equation} 
			\int_{t \in \R} \int_{x \in \R^d} w(\abs{x}) \abs{\psi_0(H) e^{-itH} u_0(x)}^2 \, dx \, dt \leq C\norm{u_0}^2_{L^2} 
		\end{equation}
		holds for some $C > 0$, where $w \colon (0, \infty) \to [0, \infty)$ and $\psi_0 \colon \sigma(H) \to \R$ are some sufficiently nice functions (see \cite[Assumptions 3.1, 3.2]{BRS2020} for details). 
		Now let $I \subset \sigma(H)$ be an open interval, $a \in C^1(I; \R)$ be such that $a' > 0$, and $E(I)$ be the spectral projection onto $I$ associated with $H$.
		Then 
		\begin{equation} 
			\int_{t \in \R} \int_{x \in \R^d} w(\abs{x}) \abs{ 
				\abs{ a'(H) }^{1/2} \psi_0(H) e^{-it a(H)} E(I) u_0(x)}^2 \, dx \, dt \leq C\norm{ E(I) u_0}^2_{L^2}  
		\end{equation}
		holds for the same constant $C > 0$.
	\end{quotetheorem}
	In fact, if we have
	\begin{equation} \tag*{\eqref{eq:smoothing Schrodinger}}
		\int_{t \in \R}\int_{x \in \R^d} w(\abs{x}) \abs{\psi(\abs{\D}) e^{it\Laplacian} u_0(x)}^2 \, dx \, dt \leq C\norm{u_0}^2_{L^2(\R^d)}
	\end{equation}
	for some $C > 0$, then using Theorem \ref{thm:comparison principle} with 
	\begin{equation}
		H = - \Laplacian , \quad 
		\psi_0(r) = \psi(r^{1/2}) , \quad
		I = (0, \infty) , \quad
		a(r) = (r + m^2)^{1/2} ,
	\end{equation}
	we have
	\begin{equation} \label{eq:smoothing relativistic Schrodinger}
		\int_{t \in \R}\int_{x \in \R^d} w(\abs{x}) \abs{ (- \Laplacian + m^2)^{-1/4} \psi(\abs{\D}) e^{-it (- \Laplacian + m^2)^{1/2}} u_0(x)}^2 \, dx \, dt \leq 2 C\norm{u_0}^2_{L^2(\R^d)} .
	\end{equation}
	Conversely, if we have \eqref{eq:smoothing relativistic Schrodinger}, then using Theorem \ref{thm:comparison principle} with 
	\begin{equation}
		H = (- \Laplacian + m^2)^{1/2} , \,\,
		\psi_0(r) =r^{-1/2} \psi((r^2 - m^2)^{1/2}) , \,\,
		I = (m, \infty) , \,\,
		a(r) = r^2 - m^2 ,
	\end{equation}
	we have \eqref{eq:smoothing Schrodinger}. 
	Therefore, we obtain the following:
	\begin{corollary}
		Let $\RSconstwpmd{w}{\psi}{m}{d}$ be the optimal constant for \eqref{eq:smoothing relativistic Schrodinger}, that is, 
		\begin{align}
			&\quad\RSconstwpmd{w}{\psi}{m}{d} \\
			&\coloneqq \sup_{\substack{u_0 \in L^2(\R^d) \\ 
					\norm{u_0}_{L^2} = 1}} \int_{t \in \R}\int_{x \in \R^d} w(\abs{x}) \abs{ (- \Laplacian + m^2)^{-1/4} \psi(\abs{\D}) e^{-it (- \Laplacian + m^2)^{1/2}} u_0(x)}^2 \, dx \, dt .
		\end{align}
		Then we have
		\begin{equation}
			\RSconstwpmd{w}{\psi}{m}{d} = 2 \Sconstwpd{w}{\psi}{d} .
		\end{equation}
	\end{corollary}
	Now let $\Diracconstmd{m}{d} = \Diracconstwpmd{w}{\psi}{m}{d}$ be the optimal constant of the following smoothing estimate for the Dirac equation:
	\begin{equation} \label{eq:smoothing Dirac}
		\int_{t \in \R}\int_{x \in \R^d} w(\abs{x}) \abs{ (- \Laplacian + m^2)^{-1/4} \psi(\abs{\D}) e^{-it\Diracopm{m}}u_0(x)}^2 \, dx \, dt \leq C\norm{u_0}^2_{L^2(\R^d, \C^\dimN)} .
	\end{equation}
	In this case, by an argument similar to that for $2 \Sconstwpd{w}{\psi}{d} \leq \RSconstwpmd{w}{\psi}{m}{d}$, we can show that 
	\begin{equation}
		\Sconstwpd{w}{\psi}{d} \leq \Diracconstwpmd{w}{\psi}{m}{d} .
	\end{equation}
	Indeed, if we have \eqref{eq:smoothing Dirac}, then using Theorem \ref{thm:comparison principle} with 
	\begin{equation}
		H = \Diracopm{m} , \quad
		\psi_0(r) = \abs{r}^{-1/2} \psi((r^2 - m^2)^{1/2}) , \quad
		\begin{cases}
			I = I_+ \coloneqq (m, \infty) , & a(r) = r^2 - m^2 \\
			I = I_- \coloneqq (- \infty, -m) , & a(r) = - r^2 + m^2 ,
		\end{cases}
	\end{equation}
	we have
	\begin{equation}
		\int_{t \in \R} \int_{x \in \R^d} w(\abs{x}) \abs{ 
			\psi(\abs{D}) e^{\pm it \Laplacian \matrixIN{\dimN}} E(I_\pm) u_0(x)}^2 \, dx \, dt \leq \frac{1}{2} C \norm{ E(I_\pm) u_0}^2_{L^2(\R^d, \C^\dimN)} .
	\end{equation}
	Therefore, we obtain
	\begin{align}
		&\quad \int_{t \in \R} \int_{x \in \R^d} w(\abs{x}) \abs{ 
			\psi(\abs{D}) e^{it \Laplacian \matrixIN{\dimN}} u_0(x)}^2 \, dx \, dt \\
		&\leq 2 \int_{t \in \R} \int_{x \in \R^d} w(\abs{x}) \abs{ 
			\psi(\abs{D}) e^{it \Laplacian \matrixIN{\dimN}} E(I_+) u_0(x)}^2 \, dx \, dt \\
		& + 2 \int_{t \in \R} \int_{x \in \R^d} w(\abs{x}) \abs{ 
			\psi(\abs{D}) e^{- it \Laplacian \matrixIN{\dimN}} E(I_-) u_0(x)}^2 \, dx \, dt \\
		&\leq C ( \norm{ E(I_+) u_0}_{L^2(\R^d, \C^\dimN)}^2 + \norm{ E(I_-) u_0}_{L^2(\R^d, \C^\dimN)}^2 ) \\
		&= C \norm{u_0}_{L^2(\R^d, \C^\dimN)}^2 , 
	\end{align}
	which shows that $\Sconstwpd{w}{\psi}{d} \leq \Diracconstwpmd{w}{\psi}{m}{d}$.
	Meanwhile, it is impossible to obtain $\Sconstwpd{w}{\psi}{d} \gtrsim \Diracconstwpmd{w}{\psi}{m}{d}$ via Theorem \ref{thm:comparison principle}, since  
	there does not exist $a \in C^1(I; \R)$ such that $a(-\Laplacian \matrixIN{\dimN}) = \Diracopm{m}$.
	Nevertheless, \citet{Iko2022} showed the following by estimating $\Diracconstmd{m}{d}$ directly:
	\begin{quotetheorem}[{\cite[Theorem 2.1]{Iko2022}}] \label{thm:Schrodinger implies Dirac}
		Let $d \geq 2$. Then we have
	\begin{equation} \label{eq:equivalence of Dirac and Schrodinger}
	\Sconstwpd{w}{\psi}{d} \leq \Diracconstwpmd{w}{\psi}{m}{d} \leq 2 \Sconstwpd{w}{\psi}{d} .
\end{equation}
	\end{quotetheorem}
	This allows us to obtain the smoothing estimate for the Dirac equation via that for the Schr\"{o}dinger equation.
	For example, the estimates given in \citet*[Theorems 4.2, 4.3]{KOY2019} and \citet[Theorem 7.1]{BU2021} can be derived from the corresponding inequalities for the Schr\"{o}dinger equation immediately.
	On the other hand, an explicit formula for $\Diracconstwpmd{w}{\psi}{m}{d}$ similar to that for $\Sconstwpd{w}{\psi}{d}$ given in Theorem \ref{thm:Schrodinger} is known only in the cases $d=2, 3$, which were given by \citet[Theorem 2.2]{Iko2022} and \citet[Theorem 1.1]{IkS2025}, respectively.
	We extend these results to arbitrary dimensions $d \geq 2$.
	\begin{theorem} \label{thm:Dirac intro}
		Let $d \geq 2$.
		Then we have 
		\begin{equation}
		\Diracconstwpmd{w}{\psi}{m}{d} = \Diraclambdasup \coloneqq \sup_{k\in\N}\esssup_{r>0}\Diraclambdak{k}(r), 
		\end{equation}
		where
		\begin{equation} \label{eq:Diraclambdak}
			\Diraclambdak{k}(r) \coloneqq
			\lambdakd{k}{d}(r) + \lambdakd{k+1}{d}(r) + \frac{m}{\sqrt{r^2 + m^2}} \abs{ \lambdakd{k}{d}(r) - \lambdakd{k+1}{d}(r) } ,
		\end{equation}
		and $\lambdakd{k}{d}$ is that given in Theorem \ref{thm:Schrodinger}.
		Furthermore, extremisers exist if and only if there exists $k \in \N$ such that the Lebesgue measure of
		\begin{equation}
			\set{ r > 0 }{ \Diraclambdak{k}(r) = \Diraclambdasup }
		\end{equation}
		is non-zero.
	\end{theorem}
	Here we emphasize that our proof of Theorem \ref{thm:Dirac intro} is certainly simpler than those in \cite{Iko2022, IkS2025} for $d=2, 3$.
	Namely, we no longer need a tedious calculation to construct a certain orthonormal basis for $L^2(\S^{d-1}, \C^\dimN)$ satisfying some useful identities via explicit representations of basis functions for spherical harmonic polynomials.
	Rather than that, we show the existence of such a nice basis in an abstract way (see Section \ref{section:invariant subspaces}). 
	
	As an application of Theorem \ref{thm:Dirac intro}, we also give analogues of Theorem \ref{thm:Schrodinger explicit value}.
	\begin{theorem} \label{thm:type A Dirac}
		In the case \eqref{eq:type A} with $s = 2$, we have
		\begin{equation}
			\Diracconstmd{m}{d} 
			=
			\begin{cases}
				4\pi / 3 , & d = 3 , m = 0 , \\
				2\pi , & d = 3 , m > 0 , \\
				\pi , & d = 4 , m = 0 , \\
				\sup_{r > 0} \Diraclambdak{0}(r) , & d = 4 , m > 0 , \\
				\pi , & d \geq 5 , 
			\end{cases}
		\end{equation}
		where
		\begin{equation}
			\Diraclambdak{0}(r) = \pi (1+r^2)^{1/2} \mleft( \mleft( 1 + \frac{m}{ \sqrt{r^2 + m^2} } \mright) \BesselI{1}(r) \BesselK{1}(r) + \mleft( 1 - \frac{m}{ \sqrt{r^2 + m^2} } \mright) \BesselI{2}(r) \BesselK{2}(r) \mright)
		\end{equation}
		when $d = 4$ and $m > 0$, 
		and there are no extremisers.
	\end{theorem}
	\begin{theorem} \label{thm:type B Dirac}
		In the case \eqref{eq:type B}, we have
		\begin{align}
			\Diracconstmd{m}{d}
			&= \begin{cases}
				\dfrac{d - 1}{ d + s - 2 }  \pi^{1/2} \dfrac{  \Gamma((s-1)/2) \Gamma((d-s)/2) }{ \Gamma(s/2) \Gamma((d+s)/2-1) } , & m = 0 , \\
				\pi^{1/2} \dfrac{  \Gamma((s-1)/2) \Gamma((d-s)/2) }{ \Gamma(s/2) \Gamma((d+s)/2-1) }, & m > 0 .
			\end{cases}
		\end{align}
		Furthermore, $u_0 \in L^2(\R^d, \C^\dimN) \setminus \{0\}$ is an extremiser if and only if
		\begin{equation}
			\widehat{u_0}(\xi) = f(\xi) + \Diracphim{0}(\xi / \abs{\xi}) g(\xi) 
		\end{equation}
		for some radial functions $f, g \in L^2(\R^d, \C^\dimN)$ when $m = 0$, and there are no extremisers when $m>0$.
		Here we recall that $\Diracphim{0}$ denotes the symbol of the massless Dirac operator $\Diracopm{0}$.
	\end{theorem}
	\begin{theorem} \label{thm:type C Dirac}
		Let $d \geq 3$, $w \in L^1(\positiveR)\setminus \{0\}$, and assume that $\Fwd{d}$ is non-negative.
		Then we have
		\begin{equation}
			\Diracconstwpmd{w}{r^{1/2}}{m}{d}
			= 2 \norm{w}_{L^1(\positiveR)} ,
		\end{equation}	
		and there are no extremisers.
		In particular, in the case \eqref{eq:type C}, we have
		\begin{equation}
			\Diracconstwpmd{ (1+r^2)^{-s/2} }{ r^{1/2} }{m}{d}
			= \pi^{1/2} \frac{ \Gamma( (s-1)/2 ) }{ \Gamma(s/2) } 
		\end{equation}
		for any $d \geq 3$.
	\end{theorem}
	Furthermore, actually Theorem \ref{thm:type C Dirac} is still valid for the massless two-dimensional Dirac equation unlike for the Schr\"{o}dinger equations, which is a little surprising.
	\begin{theorem} \label{thm:type C 2D Dirac}
		Let $d = 2$, $m=0$, $w \in L^1(\positiveR) \setminus \{0\}$, and assume that $\Fwd{2}$ is non-negative.
		Then we have
		\begin{equation}
			\Diracconstwpmd{w}{r^{1/2}}{0}{2}
			= 2 \norm{w}_{L^1(\positiveR)} ,
		\end{equation}	
		and there are no extremisers.
	\end{theorem}
	As a consequence of \eqref{eq:equivalence of Dirac and Schrodinger} and Theorem \ref{thm:type C 2D Dirac}, we get
	\begin{equation}
		\norm{w}_{L^1(\positiveR)} \leq
		\Sconstwpd{ w }{ r^{1/2} }{2} 
		\leq 2 \norm{w}_{L^1(\positiveR)} .
	\end{equation}
	Moreover, using Theorem \ref{thm:type C 2D Dirac}, we can show that the following hold in the case \eqref{eq:type A 2D}:
	\begin{align}
		\lim_{s \downarrow 2} (s - 2) \Sconstwpd{ (1+r^2)^{-s/2} }{ (1+r^2)^{1/4} }{2} &= \pi , \\
		\lim_{s \downarrow 2} (s - 2) \Diracconstwpmd{ (1+r^2)^{-s/2} }{ (1+r^2)^{1/4} }{m}{2} &= 
		\begin{cases}
			\pi , & m = 0 , \\
			2\pi , & m > 0 .
		\end{cases}
	\end{align}
	See Section \ref{section:remarks on 2D Schrodinger} for details on these results.
	\begin{remark}
	Comparing Theorems \ref{thm:type A Dirac}, \ref{thm:type B Dirac}, \ref{thm:type C Dirac} with Theorem \ref{thm:Schrodinger explicit value}, we see that $\Diracconstmd{m}{d} = 2 \Sconstd{d}$ holds in many cases, which means that the right inequality in \eqref{eq:equivalence of Dirac and Schrodinger} is sharp. 
		On the other hand, in the case \eqref{eq:type B} with $m = 0$, we have
		\begin{equation}
			\lim_{s \uparrow d} \frac{ \Diracconstwpmd{r^{-s}}{r^{(2-s)/2}}{0}{d} }{ \Sconstwpd{r^{-s}}{r^{(2-s)/2}}{d} }
			= \lim_{s \uparrow d} \dfrac{2(d - 1)}{ d + s - 2 }
			= 1 ,
		\end{equation}
		which shows that the left inequality in \eqref{eq:equivalence of Dirac and Schrodinger} is also sharp. On the other hand, as of this writing, we do not know if there exists $(w, \psi)$ such that $\Diracconstmd{m}{d} = \Sconstd{d}$.
	\end{remark}
	\subsection*{Organization of the paper}
	In Section \ref{section:preliminaries}, we introduce some notation and basic facts on spherical harmonics.
	In Section \ref{section:invariant subspaces}, we construct a certain orthonormal basis of $L^2(\S^{d-1}, \C^\dimN)$ (Theorem \ref{thm:Ekn}), which is a modified version of the usual spherical harmonics decomposition of $L^2(\S^{d-1})$.
	In Section \ref{section:proof of thm:Dirac}, we prove Theorem \ref{thm:Dirac intro} using Theorem \ref{thm:Ekn}.
	In Section \ref{section:explicit values}, we prove Theorems \ref{thm:type A Dirac}, \ref{thm:type B Dirac}, \ref{thm:type C Dirac}, and \ref{thm:type C 2D Dirac} using Theorem \ref{thm:Dirac intro}.
	In Section \ref{section:remarks on 2D Schrodinger}, we give some remarks on the two-dimensional case. 
	\section{Preliminaries} \label{section:preliminaries}
	In order to study the optimal constant of the smoothing estimate for the Dirac equation, 
	we introduce a linear operator $\DiracS \colon L^2(\R^d, \C^\dimN) \to L^2(\R^{d+1}, \C^\dimN)$ defined by 
	\begin{equation}
		(\DiracS f)(x,t) \coloneqq (2\pi)^{-d/2} w(\abs{x})^{1/2} \int_{\xi \in \R^d} e^{ix \cdot \xi} \phim{m}(\abs{\xi})^{-1/2} \psi(\abs{\xi}) e^{-it \Diracphim{m}(\xi) } f(\xi)\, d\xi ,
	\end{equation}
	so that
	\begin{equation}
		\norm{\DiracS}_{L^2(\R^d, \C^\dimN) \to L^2(\R^{d+1}, \C^\dimN)}^2 = \Diracconstwpmd{w}{\psi}{m}{d} .
	\end{equation}
	
	In order to compute $\norm{\DiracS f}_{L^2(\R^{d+1}, \C^\dimN)}$, we will use spherical harmonic polynomials. 
	For each $k \in \N$, let $\HPk{k}(\S^{d-1})$ and $\{ \pkn{k}{n} \}_{n}$ be the space of spherical harmonic polynomials of degree $k$ on $\S^{d-1}$ and its orthonormal basis with respect to the inner product of $L^2(\S^{d-1})$, respectively.
	The spherical harmonics decomposition and the Funk--Hecke theorem are as follows:
	\begin{quotetheorem}[spherical harmonics decomposition of $L^2(\R^d)$] \label{thm:spherical harmonics decomposition}
		For any $f \in L^2(\R^d)$, there uniquely exists $\{ \fkn{k}{n} \}_{ k, n } \subset L^2(\positiveR)$ satisfying
		\begin{gather}
			f(\xi) = \abs{\xi}^{-(d-1)/2} \sumk \sum_{n} \pkn{k}{n}(\xi / \abs{\xi}) \fkn{k}{n}(\abs{\xi}) , \label{eq:spherical harmonics decomposition} \\
			\norm{f}_{L^2(\R^d)}^2 = \sum_{k=0}^{\infty} \sum_{n}  \norm{\fkn{k}{n}}_{L^2(\positiveR)}^2 . \label{eq:spherical harmonics decomposition norm}
		\end{gather}
		Conversely, for any $\{ \fkn{k}{n} \} \subset L^2(\positiveR)$ satisfying
		\begin{equation}
			\sum_{k=0}^{\infty} \sum_{n}  \norm{\fkn{k}{n}}_{L^2(\positiveR)}^2 < \infty ,
		\end{equation}
		the function $f$ given by \eqref{eq:spherical harmonics decomposition} is in $L^2(\R^d)$ and satisfies \eqref{eq:spherical harmonics decomposition norm}.
	\end{quotetheorem}
	\begin{quotetheorem}[Funk--Hecke theorem] \label{thm:Funk-Hecke}
		Let $d \geq 2$, $k \in \N$, $F \in L^1 ( [-1, 1], (1-t^2)^{(d-3)/2} \, dt )$ and write
		\begin{equation} \label{eq:mu in Funk-Hecke}
			\lambda^{(d)}[k; F] \coloneqq \frac{ 2 \pi^{(d-1)/2} }{ \Gamma( (d-1)/2 ) } \int_{-1}^{1} F( t ) \Pkd{k}{d}(t) (1-t^2)^{(d-3)/2} \, dt .
		\end{equation} 
		Then, for any $Y \in \HPk{k}(\S^{d-1})$ and $\xi \in \S^{d-1}$, we have
		\begin{equation} \label{eq:Funk-Hecke}
			\int_{\eta \in \S^{d-1}} F(\xi \cdot \eta) Y(\eta) \, d\sigma(\eta) 
			= \lambda^{(d)}[k; F] Y(\xi) . 
		\end{equation}
	\end{quotetheorem}
	Theorems \ref{thm:spherical harmonics decomposition} and \ref{thm:Funk-Hecke} are classical and well known. 
	For example, see \citet[Theorems 2.22, 2.38]{AH2012}.
	Note that the function $\lambdakd{k}{d}$ can be written as
	\begin{equation}
	(2\pi)^{d/2 - 1} \lambdakd{k}{d}(r) 
		= \frac{1}{2} r^{d-2} (\psi(r) )^2 \lambda^{(d)}[k; \Fwd{d}( r \sqrt{2 (1 - \variabledot)} ) ] , 
	\end{equation}
	in other words, it satisfies 
	\begin{equation}
	\frac{1}{2} r^{d-2}  (\psi(r) )^2 \int_{\eta \in \S^{d-1}} \Fwd{d}( r \sqrt{2 (1 - \xi \cdot \eta)} ) Y(\eta) \, d\sigma(\eta) 
		= (2\pi)^{d/2 - 1} \lambda_k(r) Y(\xi) 
	\end{equation}
	for each $k \in \N$, $Y \in \HPk{k}(\S^{d-1})$ and $\xi \in \S^{d-1}$.
	For readers' convenience, we list some other basic properties of $\lambdakd{k}{d}$ here.
	\begin{proposition} \label{prop:properties of lambdak}
		\renewcommand{\lambdakd}[2]{\lambda_{#1}^{(#2)}}
		We write $\lambdakd{k}{d}$ for $\lambdak{k}$ in dimension $d$.
		Then the following hold:
		\begin{eqenumerate}
			\item \label{item:properties of lambdak 1}
			We have $\lambdakd{k}{d}(r) \geq 0$ for any $k \in \N$ and $r > 0$.
			\item \label{item:properties of lambdak 2}
			If $2k_1 + d_1 = 2k_2 + d_2$, then we have $\lambdakd{k_1}{d_1} = \lambdakd{k_2}{d_2}$.
			\item \label{item:properties of lambdak 3}
			If $\Fwd{d}$ is non-negative, then we have $\lambdakd{k+1}{d}(r) \leq \lambdakd{0}{d}(r)$ for any $k \in \N$ and $r > 0$. 
			Furthermore, if $\lambdakd{k+1}{d}(r) = \lambdakd{0}{d}(r)$, then we have $\lambdakd{0}{d}(r) = 0$ or $\lambdakd{k+1}{d}(r) = \infty$.
			\item \label{item:properties of lambdak 4}
			If $\Fwd{d + 2j}$ is non-negative for any $j \in \N$, then we have $\lambdakd{k+1}{d}(r) \leq \lambdakd{k}{d}(r)$ for any $k \in \N$ and $r > 0$.
		\end{eqenumerate}
	\end{proposition}
	\begin{proof}[Proof of Proposition \ref{prop:properties of lambdak}]
		\renewcommand{\lambdakd}[2]{\lambda_{#1}^{(#2)}}
		\eqref{item:properties of lambdak 1} and \eqref{item:properties of lambdak 2} immediately follow from \eqref{eq:lambdak 1}.
		To see \eqref{item:properties of lambdak 3}, use \eqref{eq:lambdak 2} and the fact that the Legendre polynomial $\Pkd{k}{d}$ satisfies
		\begin{equation} \label{eq:Pkd inequality}
			\abs{ \Pkd{k}{d}(t) } < \Pkd{0}{d}(t) = 1 
		\end{equation}
		for every $k \in \N_{\geq 1}$, $d \geq 2$, and almost every $t \in [-1, 1]$ (see \citet[(2.39)]{AH2012}). 
		\eqref{item:properties of lambdak 4} follows from \eqref{item:properties of lambdak 2} and \eqref{item:properties of lambdak 3}:
		\begin{equation*}
			\lambdakd{k+1}{d}
			\underset{\eqref{item:properties of lambdak 2}}{=} 
			 \lambdakd{1}{d + 2k} 
			\underset{\eqref{item:properties of lambdak 3}}{\leq} 
			 \lambdakd{0}{d + 2k} 
			\underset{\eqref{item:properties of lambdak 2}}{=} 
			\lambdakd{k}{d}. \qedhere
		\end{equation*}
	\end{proof}
	Note that \eqref{item:properties of lambdak 3} is useful to find explicit values of $\Sconstd{d}$, since $\lambdakd{k+1}{d}(r) \leq \lambdakd{0}{d}(r)$ implies 
	\begin{equation}
		\Sconstd{d} = \sup_{r > 0} \lambdakd{0}{d}(r) .
	\end{equation}
	However, it does not work well for $\Diracconstmd{m}{d}$, since $\lambdakd{k+1}{d}(r) \leq \lambdakd{0}{d}(r)$ does not guarantee $\Diraclambdak{k+1}(r) \leq \Diraclambdak{0}(r)$.
	On the other hand, if we have $\lambdakd{k+1}{d}(r) \leq \lambdakd{k}{d}(r)$ for any $k \in \N$, then we also have $\Diraclambdak{k+1}(r) \leq \Diraclambdak{k}(r)$ for any $k \in \N$, so that 
	\begin{equation}
	\Diracconstmd{m}{d} = \sup_{r > 0} \Diraclambdak{0}(r) 
	\end{equation}
	holds. This is the reason why \eqref{item:properties of lambdak 4} is more useful than \eqref{item:properties of lambdak 3} in order to find $\Diracconstmd{m}{d}$.
	
	\section{Spherical harmonics and invariant subspaces} \label{section:invariant subspaces}
	Throughout this section, we consider $\Diracphim{m}$ as a linear operator on $L^2(\S^{d-1}, \C^{\dimN})$ via
	\begin{equation}
		(\Diracphim{m} f)(\xi) \coloneqq \Diracphim{m}(\xi) f(\xi) .
	\end{equation}
	Our aim in this section is to construct a certain orthonormal basis of $L^2(\S^{d-1}, \C^\dimN)$ satisfying some nice properties, which allows us to prove our Theorem \ref{thm:Dirac intro} by an argument similar to that for Theorem \ref{thm:Schrodinger} in \cite{BSS2015}. 
	
	We begin with the following facts:
	\begin{proposition} \label{prop:invariant subspace}
		Let $V \subset L^2(\S^{d-1}, \C^{\dimN})$ be a linear subspace of $L^2(\S^{d-1}, \C^{\dimN})$.
		Then
		\begin{equation}
			V + \Diracphim{0} V \coloneqq \set{ f + \Diracphim{0} g }{ f, g \in V } 
		\end{equation}
		is an invariant subspace for $\Diracphim{0}$.
	\end{proposition}
	\begin{proposition} \label{prop:Gauss decomposition}
		Let $1 \leq j \leq d$, $k \geq 1$, and $Y \in \HPk{k}(\R^d)$, where $\HPk{k}(\R^d)$ denotes the space of homogeneous harmonic polynomials of degree $k$ on $\R^d$. 
		Then there exist $Z \in \HPk{k-1}(\R^d)$ and $W \in \HPk{k+1}(\R^d)$ such that
		\begin{equation}
			\xi_j Y(\xi) = \abs{\xi}^2 Z(\xi) + W(\xi) .
		\end{equation}
		As a consequence, we have
		\begin{equation}
			\Diracphim{0} \HPk{k}(\S^{d-1}, \C^{\dimN}) \subset \HPk{k-1}(\S^{d-1}, \C^{\dimN}) \obot \HPk{k+1}(\S^{d-1}, \C^{\dimN}) .
		\end{equation}
	\end{proposition}
	\begin{proof}[Proof of Proposition \ref{prop:invariant subspace}]
		Let $f, g \in V$. Then we have
		\begin{equation}
			\Diracphim{0} ( f + \Diracphim{0} g ) = \Diracphim{0} f + \Diracphim{0}^2 g = \Diracphim{0} f + g \in V + \Diracphim{0} V ,
		\end{equation}
		since $\Diracphim{0}(\xi)^2 = \matrixIN{\dimN}$ for $\xi \in \S^{d-1}$.
	\end{proof}
	\begin{proof}[Proof of Proposition \ref{prop:Gauss decomposition}]
		Let 
		\begin{gather}
			Z = \frac{1}{d + 2k - 2} \partial_j Y , \\
			W = \xi_j Y - \frac{1}{d + 2k - 2} \abs{\xi}^2 \partial_j Y .
		\end{gather}
		Then it is trivial that
		\begin{itemize}
			\item $\xi_j Y = \abs{\xi}^2 Z + W$, 
			\item $Z \in \HPk{k-1}(\R^d)$, 
			\item $W$ is a homogeneous polynomial of degree $k+1$.
		\end{itemize}
		Hence, it suffices to show that $\Laplacian W = 0$. 
		To see this, note that $\Laplacian Y = 0$ implies
		\begin{gather}
			\Laplacian ( \xi_j Y ) = 2 \partial_j Y , \\
			\Laplacian ( \abs{\xi}^2 \partial_j Y ) 
			= 2d \partial_j Y + 4 \sum_{i = 1}^{d} \xi_i \partial_{i} \partial_j Y . 
		\end{gather}
		Furthermore, since $\partial_j Y$ is a homogeneous polynomial of degree $k-1$, 
		we have
		\begin{equation}
			\sum_{i = 1}^{d} \xi_i \partial_{i} \partial_j Y = (k-1) \partial_j Y .
		\end{equation}
		Therefore, we conclude that
		\begin{align}
			\Laplacian W 
			&= 2 \partial_j Y - \frac{1}{d + 2k - 2} \mleft( 2d \partial_j Y + 4 \sum_{i = 1}^{d} \xi_i \partial_{i} \partial_j Y  \mright) \\
			&= \frac{4}{d + 2k - 2} \mleft( (k - 1) \partial_j Y - \sum_{i = 1}^{d} \xi_i \partial_{i} \partial_j Y \mright) \\
			&= 0 
		\end{align}
		holds.
	\end{proof}
	Using Propositions \ref{prop:invariant subspace} and \ref{prop:Gauss decomposition}, we prove the following result:
	\begin{lemma} \label{lem:Vk}
		Let $\{ \Vk{k} \}_{k \in \N}$ be a sequence of linear subspaces of $L^2(\S^{d-1}, \C^{\dimN})$ defined by
		\begin{align}
			\Vk{0} &= \HPk{0}( \S^{d-1}, \C^\dimN ) , \\
			\Vk{k+1} &= ( \Diracphim{0} \Vk{k} )^{\perp} \cap \HPk{k+1}( \S^{d-1}, \C^\dimN ) \\
			&= \set{ Y \in \HPk{k+1}( \S^{d-1}, \C^\dimN ) }{ \forall Z \in \Vk{k} , \innerproduct{Y}{ \Diracphim{0} Z}_{L^2(\S^{d-1})} = 0 } .
		\end{align}
		Then
		\begin{gather}
			\Vk{k} \subset \HPk{k}(\S^{d-1}, \C^\dimN), 
			\label{item:Vk 1}\\
			\Diracphim{0} \Vk{k} \subset \HPk{k+1}(\S^{d-1}, \C^\dimN) , 
			\label{item:Vk 2} \\
			\Diracphim{0} \Vk{k} \obot \Vk{k+1} = \HPk{k+1}( \S^{d-1}, \C^\dimN ) 
			\label{item:Vk 3} 
		\end{gather}
		hold for any $k \in \N$.
		As a consequence, we have
		\begin{equation}
			L^2(\S^{d-1}, \C^\dimN) = \bigobot_{  k \in \N  } ( \Vk{k} \obot \Diracphim{0} \Vk{k} ) .
		\end{equation}
	\end{lemma}
	\begin{proof}[Proof of Lemma \ref{lem:Vk}]
		\eqref{item:Vk 1} is trivial from the definition of $\Vk{k}$. 
		Moreover, \eqref{item:Vk 3} immediately follows from \eqref{item:Vk 2} and the definition of $\Vk{k+1}$.
		We prove \eqref{item:Vk 2} by induction.
		
		The case $k = 0$ is trivial since $\Vk{0} = \HPk{0}( \S^{d-1}, \C^\dimN )$.
		
		To see the case $k = 1$, fix $Y \in \Vk{1}$ arbitrarily.
		Since we know that $Y \in \HPk{1}(\S^{d-1}, \C^\dimN)$ holds by \eqref{item:Vk 1},
		Proposition \ref{prop:Gauss decomposition} implies
		\begin{equation}
			\Diracphim{0} Y \in \HPk{0}(\S^{d-1}, \C^\dimN) \obot \HPk{2}(\S^{d-1}, \C^\dimN) . 
		\end{equation}
		On the other hand, using the definition of $\Vk{1}$ and the self-adjointness of $\Diracphim{0}$, we get
		\begin{equation}
			\innerproduct{ \Diracphim{0} Y }{ Z }_{L^2(\S^{d-1})} 
			= \innerproduct{ Y }{ \Diracphim{0} Z }_{L^2(\S^{d-1})} 
			= 0
		\end{equation}
		for any $Z \in \Vk{0} = \HPk{0}(\S^{d-1}, \C^\dimN)$. 
		Therefore, we conclude that \eqref{item:Vk 2} holds for $k = 1$.
		
		Now we show that if \eqref{item:Vk 2} holds for $k = k_0$, 
		then it also holds for $k = k_0 + 2$.
		Fix $Y \in \Vk{k_0 + 2}$ arbitrarily.
		Since we know that $Y \in \HPk{k_0 + 2}(\S^{d-1}, \C^\dimN)$ holds by \eqref{item:Vk 1},
		Proposition \ref{prop:Gauss decomposition} implies
		\begin{equation}
			\Diracphim{0} Y \in \HPk{k_0 + 1}(\S^{d-1}, \C^\dimN) \obot \HPk{k_0 + 3}(\S^{d-1}, \C^\dimN) .
		\end{equation}
		Moreover, since we have 
		\begin{equation}
			\HPk{k_0 + 1}(\S^{d-1}, \C^\dimN) = \Diracphim{0} \Vk{k_0} \obot \Vk{k_0+1}
		\end{equation}
		by the induction hypothesis, we obtain
		\begin{equation}
			\Diracphim{0} Y \in \Diracphim{0} \Vk{k_0} \obot \Vk{k_0 + 1} \obot \HPk{k_0 + 3}(\S^{d-1}, \C^\dimN) .
		\end{equation}
		On the other hand, using the definition of $\Vk{k_0 + 2}$ and the self-adjointness of $\Diracphim{0}$, we get
		\begin{equation}
			\innerproduct{ \Diracphim{0} Y }{ Z }_{L^2(\S^{d-1})} 
			= \innerproduct{ Y }{ \Diracphim{0} Z }_{L^2(\S^{d-1})} 
			= 0
		\end{equation}
		for any $Z \in \Vk{k_0 + 1} $. 
		Furthermore, using $\Vk{k_0 + 2} \subset \HPk{k_0 + 2}(\S^{d-1}, \C^\dimN)$ and the unitarity of $\Diracphim{0}$, we get
		\begin{equation}
			\innerproduct{ \Diracphim{0} Y }{ \Diracphim{0} Z }_{L^2(\S^{d-1})} 
			= \innerproduct{ Y }{ Z }_{L^2(\S^{d-1})} 
			= 0
		\end{equation}
		for any $Z \in \Vk{k_0} \subset \HPk{k_0}(\S^{d-1}, \C^\dimN)$.
		Therefore, we conclude that \eqref{item:Vk 2} holds for $k = k_0 + 2$.
	\end{proof}
	Furthermore, in order to compute $\gammaj{d+1} Y$ and $\gammaj{d+1} \Diracphim{0} Y$ for $Y \in \Vk{k}$, we decompose $\Vk{k}$ as follows:
	\begin{lemma} \label{lem:Vkb}
		For each $k \in \N$ and $\mu \in \{1, -1\}$, we write
		\begin{equation}
			\Vkm{k} \coloneqq \set{ Y \in \Vk{k} }{ \gammaj{d+1} Y = \mu Y } .
		\end{equation}
		Then we have
		\begin{gather}
			\Vk{k} = \bigobot_{\mu \in \{ 1, -1 \}} \Vkm{k} ,
			\label{eq:Vkb 1}\\
			\Vkm{k} = \set{ Y \in \Vk{k} }{ \gammaj{d+1} \Diracphim{0} Y = - \mu \Diracphim{0} Y } .
			\label{eq:Vkb 2}
		\end{gather}
	\end{lemma}
	\begin{proof}[Proof of Lemma \ref{lem:Vkb}]
		In order to prove \eqref{eq:Vkb 1}, it is enough to show that 
		\begin{equation} \label{eq:Vkb 3}
			\gammaj{d+1} \Vk{k} \subset \Vk{k} 
		\end{equation}
		holds, since $\gammaj{d+1}$ is a Hermitian matrix whose eigenvalues are $\pm1$.
		We prove \eqref{eq:Vkb 3} by induction.
		In the case $k = 0$, it is trivial since $\Vk{0} = \HPk{0}(\S^{d-1}, \C^\dimN)$.
		Suppose that \eqref{eq:Vkb 3} holds for $k = k_0$ and
		fix $Y \in \Vk{k_0 + 1}$ arbitrarily. 
		Since $\gammaj{d+1} Y \in \HPk{k_0 + 1}(\S^{d-1}, \C^\dimN)$ is immediate from $\Vk{k_0 + 1} \subset \HPk{k_0 + 1}(\S^{d-1}, \C^\dimN)$, it suffices to show that
		\begin{equation}
			\innerproduct{\gammaj{d+1} Y}{\Diracphim{0} Z}_{L^2(\S^{d-1})} = 0
		\end{equation}
		for any $Z \in \Vk{k_0}$. 
		To see this, notice that we have:
		\begin{itemize}
			\item $\innerproduct{Y}{\Diracphim{0} Z}_{L^2(\S^{d-1})} = 0$ for any $Z \in \Vk{k_0}$, since $Y \in \Vk{k_0 + 1} \subset ( \Diracphim{0} \Vk{k_0} )^{\perp}$.
			\item $\gammaj{d+1} \Vk{k_0} \subset \Vk{k_0}$ by the induction hypothesis.
		\end{itemize}
		Combining these, we get
		\begin{equation}
			\innerproduct{\gammaj{d+1} Y}{\Diracphim{0} Z}_{L^2(\S^{d-1})} 
			= \innerproduct{Y}{\gammaj{d+1} \Diracphim{0} Z}_{L^2(\S^{d-1})} 
			= - \innerproduct{Y}{\Diracphim{0} \gammaj{d+1} Z}_{L^2(\S^{d-1})} 
			= 0 .
		\end{equation}
		Therefore, \eqref{eq:Vkb 3} holds for any $k \in \N$.
		
		Now we prove \eqref{eq:Vkb 2}. 
		If $Y \in \Vkm{k}$, then
		\begin{equation}
			\gammaj{d+1} \Diracphim{0} Y = - \Diracphim{0} \gammaj{d+1} Y = - \mu \Diracphim{0} Y .
		\end{equation}
		On the other hand, if $Y \in \Vk{k}$ satisfies $\gammaj{d+1} \Diracphim{0} Y = - \mu \Diracphim{0} Y$, then
		\begin{equation}
			\gammaj{d+1} Y = \gammaj{d+1} \Diracphim{0}^2 Y = - \Diracphim{0} \gammaj{d+1} \Diracphim{0} Y = \mu \Diracphim{0}^2 Y = \mu Y, 
		\end{equation}
		hence $Y \in \Vkm{k}$.
	\end{proof}
	As a consequence, we obtain the following result:
	\begin{theorem} \label{thm:Ekn}
		For each $k \in \N$ and $\mu \in \{ 1, -1 \}$, let $\{ \eknm{k}{n} \}_n$ be an orthonormal basis for $\Vkm{k}$, 
		and write 
		\begin{equation}
			\Eknm{k}{n} \coloneqq 
			\begin{cases}
				\begin{pmatrix}
					\eknm{k}{n} & \Diracphim{0} \eknm{k}{n}
				\end{pmatrix} , & \mu = 1 , \\
				\begin{pmatrix}
					\Diracphim{0} \eknm{k}{n} & \eknm{k}{n}
				\end{pmatrix} , & \mu = - 1 .	
			\end{cases}		
		\end{equation}
		Then the following hold:
		\begin{eqenumerate}
			\item \label{item:spherical harmonics decomposition Dirac}
			For any $f \in L^2(\R^d, \C^\dimN)$, there uniquely exists $\{ \fknm{k}{n} \} \subset L^2(\positiveR, \C^2)$ satisfying
			\begin{gather}
				f(\xi) =  \abs{\xi}^{-(d-1)/2} \sum_{\mu \in \{ 1, -1 \}} \sumk \sumn \Eknm{k}{n}(\xi / \abs{\xi}) \fknm{k}{n}(\abs{\xi}) , 
				\label{eq:spherical harmonics decomposition Dirac} \tag{\ref*{item:spherical harmonics decomposition Dirac}.i}
				\\
				\norm{f}_{L^2(\R^d, \C^\dimN)}^2 = \sum_{\mu \in \{ 1, -1 \}} \sumk \sumn \norm{\fknm{k}{n}}_{L^2(\positiveR, \C^2)}^2 .
				\label{eq:spherical harmonics decomposition norm Dirac} \tag{\ref*{item:spherical harmonics decomposition Dirac}.ii}
			\end{gather}
			Conversely, for any $\{ \fknm{k}{n} \} \subset L^2(\positiveR, \C^2)$ satisfying
			\begin{equation}
				\sum_{\mu \in \{ 1, -1 \}}	\sum_{k=0}^{\infty} \sum_{n}  \norm{\fknm{k}{n}}_{L^2(\positiveR, \C^2)}^2 < \infty ,
			\end{equation}
			the function $f$ given by \eqref{eq:spherical harmonics decomposition Dirac} is in $L^2(\R^d, \C^\dimN)$ and satisfies \eqref{eq:spherical harmonics decomposition norm Dirac}.
			\item \label{item:transformation matrix}
			Let 
			\begin{gather}
				\Lambdak{k}(r) 
				\coloneqq \begin{cases}
					\begin{pmatrix}
		\lambdakd{k}{d}( r ) & 0 \\
		0 & \lambdakd{k+1}{d}( r )
	\end{pmatrix} , & \mu = 1 , \\
						\begin{pmatrix}
					\lambdakd{k+1}{d}( r ) & 0 \\
					0 & \lambdakd{k}{d}( r )
				\end{pmatrix} , & \mu = -1 , \\
				\end{cases}\\
				\sigmaj{1} \coloneqq \begin{pmatrix}
					0 & 1 \\
					1 & 0
				\end{pmatrix} , \quad
				\sigmaj{3}
				\coloneqq 
				\begin{pmatrix}
					1 & 0 \\
					0 & -1
				\end{pmatrix} .
			\end{gather}
			Then we have
			\begin{gather}
				\frac{1}{2} \abs{\xi}^{d-2} (\psi(\abs{\xi}))^2 \int_{ \eta \in \S^{d-1} } \Fwd{d}( \abs{ \xi -  \abs{\xi} \eta } ) \Eknm{k}{n}( \eta ) \, d\sigma(\eta)
				= (2\pi)^{d/2 - 1} \Eknm{k}{n}( \xi / \abs{\xi} ) \Lambdak{k}(\abs{\xi})  , 
				\label{eq:Funk-Hecke Ekn} \tag{\ref*{item:transformation matrix}.i}\\
				\Diracphim{0}(\xi) \Eknm{k}{n}(\xi / \abs{\xi} ) 
				= \abs{\xi} \Eknm{k}{n}(\xi / \abs{\xi} ) \sigmaj{1} ,  
				\label{eq:Ekn sigma_1} \tag{\ref*{item:transformation matrix}.ii}\\
				\gammaj{d+1} \Eknm{k}{n}(\xi / \abs{\xi} ) 
				=  \Eknm{k}{n}(\xi / \abs{\xi} ) \sigmaj{3} .
				\label{eq:Ekn sigma_3} \tag{\ref*{item:transformation matrix}.iii}\\
			\end{gather}
		\end{eqenumerate}
	\end{theorem}
	\begin{remark}
		By the definition of $\Vk{k}$, we have
		\begin{equation}
			\begin{cases}
				\dim \Vk{0} = \dimN  , \\
				\dim \Vk{k+1} + \dim \Vk{k} = \dimN \dim \HPk{k+1}(\S^{d-1}) 
			\end{cases}
		\end{equation}
		for each $k \in \N$.
		On the other hand, it is well known that
		\begin{equation}
			\dim \HPk{k+1}(\S^{d-1}) = \binom{ k + d - 1 }{ d - 2 } + \binom{k +  d - 2 }{ d - 2 } 
		\end{equation}
		holds (see \cite[(2.10)]{AH2012}, for example).
		Comparing these, we conclude that
		\begin{equation}
			\dim \Vk{k} = \binom{ k + d - 2 }{ d - 2 } \dimN
		\end{equation}
		holds for any $k \in \N$. 
		Similarly, Lemma \ref{lem:Vkb} implies
		\begin{equation}
			\begin{cases}
				\dim \Vkm{0} = \dimN / 2  , \\
				\dim \Vkmm{k+1} + \dim \Vkm{k} = ( \dimN / 2) \dim \HPk{k+1}(\S^{d-1}) 
			\end{cases}
		\end{equation}
		for each $k \in \N$ and $\mu \in \{ 1, -1 \}$, which leads to 
		\begin{equation}
			\dim \Vkm{k} = \binom{ k + d - 2 }{ d - 2 } \dimN / 2 .
		\end{equation}
		In particular, in the case $d=2$, we have
		\begin{equation}
			\dim \Vkm{k} = 1
		\end{equation}
		for any $k \in \N$ and $\mu \in \{ 1, -1 \}$.
	\end{remark}
	\begin{remark}
	When $d=2$, the most commonly used gamma matrices $\{ \gammaj{j} \}_j$ are given by
		\begin{gather}
			\gammaj{1} = \begin{pmatrix}
				0 & 1 \\
				1 & 0
			\end{pmatrix} , \quad
			\gammaj{2} = \begin{pmatrix}
				0 & -i \\
				i & 0
			\end{pmatrix} , \quad
			\gammaj{3} = \begin{pmatrix}
				1 & 0 \\
				0 & -1
			\end{pmatrix} .
		\end{gather}
In this case, we can choose an orthonormal basis $\{ \eknm{k}{} \}$ for the one-dimensional space $\Vkm{k}$ given by
		\begin{equation}
			\eknm{k}{} = \begin{cases}
				\dfrac{1}{\sqrt{2\pi}}
				\begin{pmatrix}
					e^{i k \theta} \\
					0
				\end{pmatrix} , & \mu = 1 , \\
				\dfrac{1}{\sqrt{2\pi}}
				\begin{pmatrix}
					0 \\
				e^{- i k \theta}
				\end{pmatrix} , &  \mu = -1 ,
			\end{cases}
		\end{equation}
		which leads us to 
		\begin{equation}
			\Eknm{k}{} 
			= \begin{cases}
				\dfrac{1}{\sqrt{2\pi}}
				\begin{pmatrix}
					e^{i k \theta} & 0 \\
					0 & e^{i (k+1) \theta}
				\end{pmatrix} , & \mu = 1 , \\
				\dfrac{1}{\sqrt{2\pi}} \begin{pmatrix}
					e^{- i (k+1) \theta} & 0 \\
					0 & e^{- i k \theta}
				\end{pmatrix} , & \mu = -1 .
			\end{cases} 
		\end{equation}
		Therefore, the spherical harmonics decomposition \eqref{eq:spherical harmonics decomposition Dirac}
		can be rewritten as
		\begin{equation}
			f(r \cos \theta, r \sin \theta) = \frac{1}{\sqrt{2 \pi r}} \sum_{k \in \Z} \begin{pmatrix}
				e^{i k \theta} & 0 \\
				0 & e^{i (k+1) \theta}
			\end{pmatrix} f_{k}(r) .
		\end{equation}
		We note that it becomes much more complicated when $d=3$; see \citet[Section 4]{IkS2025}.
	\end{remark}
	\section{\texorpdfstring{Proof of Theorem \ref{thm:Dirac intro}}{Proof of Theorem 1.6}} \label{section:proof of thm:Dirac}
	Using Theorem \ref{thm:Ekn}, we prove Theorem \ref{thm:Dirac intro}.
	For convenience, we divide the proof into some steps as follows:
	\begin{theorem} \label{thm:Dirac}
		Let $\mu, \nu \in \{ 1, -1 \}$ and write
		\begin{gather}
			\Pn(\xi) \coloneqq \frac{1}{2} \mleft( \matrixIN{\dimN} + \nu \frac{ \Diracphim{m}(\xi) }{ \phim{m}(\abs{\xi}) } \mright) , \\
			\Qmn(r) \coloneqq 
			\frac{1}{2} \mleft( \matrixIN{2} + \frac{\nu}{ \phim{m}(r) } ( r \sigmaj{1} + m \sigmaj{3} ) \mright) = \frac{1}{2} \begin{pmatrix}
				1 + \nu m \phim{m}(r)^{-1} & \nu r \phim{m}(r)^{-1} \\
				\nu r \phim{m}(r)^{-1} & 1 - \nu m \phim{m}(r)^{-1}
			\end{pmatrix} , \\
			\DiracLambdak{k}(r) \coloneqq 2 \sum_{ \nu \in \{ 1, -1 \} } \Qmn(r) \Lambdak{k}(r) \Qmn(r) .
		\end{gather}
		Note that $\Pn(\xi) \colon \C^\dimN \to \C^\dimN$ is an orthogonal projection onto the eigenspace of $\Diracphim{m}(\xi)$ associated with the eigenvalue $\nu \phim{m}(\abs{\xi})$.
		Then the following hold:
		\begin{eqenumerate}
			\item \label{item:S^*Sf}
			For any $f \in L^2(\R^d, \C^\dimN)$, we have
			\begin{equation}
				(2\pi)^{d/2-1} \adjoint{\DiracS} \DiracS f(\xi) 
				= \abs{\xi}^{d-2} (\psi(\abs{\xi}))^2 \sum_{ \nu \in \{ 1, -1 \} } \Pn(\xi) \int_{ \eta \in \S^{d-1} } \Fwd{d}( \abs{ \xi -  \abs{\xi} \eta } ) \Pn(\abs{\xi} \eta) f(\abs{\xi}\eta) \, d\sigma(\eta) .
			\end{equation}
			\item \label{item:PE=EQ} 
			We have
			\begin{equation}
				\Pn(\xi) \Eknm{k}{n}(\xi / \abs{\xi} ) 
				= \Eknm{k}{n}(\xi / \abs{\xi} ) 
				\Qmn(\abs{\xi}) .
			\end{equation}
			\item \label{item:S^*Sf decomposed}
			Let $f \in L^2(\R^d, \C^\dimN)$ and decompose it as
			\begin{align}
				f(\xi) &=  \abs{\xi}^{-(d-1)/2} \sum_{\mu \in \{ 1, -1 \}} \sumk \sumn \Eknm{k}{n}(\xi / \abs{\xi}) \fknm{k}{n}(\abs{\xi}) .
				\intertext{Then we have}
				\adjoint{\DiracS} \DiracS f(\xi) 
				&= \abs{\xi}^{-(d-1)/2} \sum_{\mu \in \{1, -1\}} \sumk \sumn \Eknm{k}{n}( \xi / \abs{\xi} ) \DiracLambdak{k}(\abs{\xi}) \fknm{k}{n}(\abs{\xi}) , \\
				\intertext{so that}
				\norm{\DiracS f}_{L^2(\R^{d+1}, \C^\dimN)}^2
				&= 
				\sum_{\mu \in \{1, -1\}} \sumk \sumn \int_0^\infty \innerproduct{ \DiracLambdak{k}(r) \fknm{k}{n}(r) }{ \fknm{k}{n}(r) }_{\C^2} \, dr .
			\end{align}
			\item \label{item:maximal eigenvalue}
			The maximal eigenvalue of $\DiracLambdak{k}(r)$ and its associated eigenspace are
			\begin{equation} \tag*{\eqref{eq:Diraclambdak}}
				\Diraclambdak{k}(r)
				= \lambdakd{k}{d}(r) + \lambdakd{k+1}{d}(r) + \frac{m}{\phim{m}(r)} \abs{ \lambdakd{k}{d}(r) - \lambdakd{k+1}{d}(r) } 
			\end{equation}
			and
			\begin{equation}
				\Wk{k}(r) 
				= 
				\begin{dcases}
					\C^2 , & \phantom{\mu} m ( \lambdakd{k}{d}(r) - \lambdakd{k+1}{d}(r) ) = 0 , \\
					\text{the eigenspace of $ r \sigmaj{1} + m \sigmaj{3}$ associated with $\phim{m}(r)$}, 
					& \mu m ( \lambdakd{k}{d}(r) - \lambdakd{k+1}{d}(r) )  > 0 , \\
					\text{the eigenspace of $ r \sigmaj{1} + m \sigmaj{3}$ associated with $- \phim{m}(r)$},  
					& \mu m ( \lambdakd{k}{d}(r) - \lambdakd{k+1}{d}(r) )  < 0 ,
				\end{dcases}
			\end{equation}
			respectively.
			\item \label{item:operator norm}
			We have 
			\begin{equation}
				\norm{\DiracS}_{L^2(\R^d, \C^\dimN) \to L^2(\R^{d+1}, \C^\dimN)}^2 = \Diraclambdasup = \sup_{k \in \N} \esssup_{r > 0} \Diraclambdak{k}(r) .
			\end{equation}
			Regarding extremisers, let $f \in L^2(\R^d, \C^\dimN)$ be such that
			\begin{equation}
				f(\xi) =  \abs{\xi}^{-(d-1)/2} \sum_{\mu \in \{ 1, -1 \}} \sumk \sumn \Eknm{k}{n}(\xi / \abs{\xi}) \fknm{k}{n}(\abs{\xi}) .
			\end{equation}
			Then the equality $\norm{\DiracS f}_{L^2(\R^{d+1}, \C^\dimN)}^2 = \Diraclambdasup \norm{f}_{L^2(\R^d, \C^\dimN)}^2$ holds if and only if $\fknm{k}{n}$ satisfies
			\begin{enumerate}[label={$\textup{(\ref{item:operator norm}.\roman*)}$}]
				\item \label{item:support condition}
				$\supp{ \fknm{k}{n} } \subset \set{ r > 0 }{ \Diraclambdak{k}(r) \geq \Diraclambdasup }$,
				\item \label{item:eigenspace condition}
				$\fknm{k}{n}(r) \in \Wk{k}(r)$ for almost every $r > 0$
			\end{enumerate}
			for each $k$, $n$ and $\mu$.
			As a consequence, extremisers exist if and only if there exists $k \in \N$ such that the Lebesgue measure of $\set{ r > 0 }{ \Diraclambdak{k}(r) = \Diraclambdasup }$ is non-zero.
		\end{eqenumerate}
	\end{theorem}
	\begin{remark} \label{remark:Schrodinger}
		The proof of Theorem \ref{thm:Schrodinger} given by \citet*{BSS2015} is as follows.
		Define a linear operator $\SchrodingerS \colon L^2(\R^d) \to L^2(\R^{d+1})$ by 
		\begin{align}
			(\SchrodingerS f)(x,t) &\coloneqq (2\pi)^{-d/2} w(\abs{x})^{1/2} \int_{\xi \in \R^d} e^{ix \cdot \xi}\psi(\abs{\xi})e^{-it \abs{\xi}^2 } f(\xi)\, d\xi ,
			\intertext{so that} 
		\Sconstwpd{ w }{ \psi }{d} &= \norm{\SchrodingerS}_{L^2(\R^d) \to L^2(\R^{d+1})}^2 .
		\end{align}
		\begin{eqenumerate}
			\item \label{item:S^*Sf Schrodinger}
			For any $f \in L^2(\R^d)$, we have
			\begin{equation}
				(2\pi)^{d/2-1} \adjoint{\SchrodingerS} \SchrodingerS f(\xi) 
				= 
				\frac{1}{2} \abs{\xi}^{d-2} ( \psi(\abs{\xi}) )^2  \int_{ \eta \in \S^{d-1} } \Fwd{d}( \abs{ \xi -  \abs{\xi} \eta } )  f(\abs{\xi}\eta) \, d\sigma(\eta) .
			\end{equation}
			\item \label{item:S^*Sf decomposed Schrodinger}
			Let $f \in L^2(\R^d)$ and decompose it as
			\begin{align}
				f(\xi) &= \abs{\xi}^{-(d-1)/2} \sumk \sumn \pkn{k}{n}(\xi / \abs{\xi}) \fkn{k}{n}(\abs{\xi}) .
				\intertext{Then we have}
				\adjoint{\SchrodingerS} \SchrodingerS f(\xi) 
				&= \abs{\xi}^{-(d-1)/2}
				\sumk \sumn  \pkn{k}{n}(\xi / \abs{\xi}) \lambdakd{k}{d}(\abs{\xi}) \fkn{k}{n}(\abs{\xi}) ,
				\\
				\intertext{so that}
				\norm{\SchrodingerS f}_{L^2(\R^{d+1})}^2 &= 
				\sumk \sumn \int_0^\infty  \lambdakd{k}{d}(r) \abs{\fkn{k}{n}(r) }^2 \, dr .
			\end{align}
			\item \label{item:operator norm Schrodinger}
			We have 
			\begin{equation}
				\norm{\SchrodingerS}_{L^2(\R^d) \to L^2(\R^{d+1})}^2 = \lambdasupd{d} = \sup_{k \in \N} \esssup_{r > 0} \lambdakd{k}{d}(r) .
			\end{equation}
			Regarding extremisers, let $f \in L^2(\R^d)$ be such that
			\begin{equation}
				f(\xi) = \abs{\xi}^{-(d-1)/2} \sumk \sumn \pkn{k}{n}(\xi / \abs{\xi}) \fkn{k}{n}(\abs{\xi}) .
			\end{equation}
			Then the equality $\norm{\SchrodingerS f}_{L^2(\R^{d+1})}^2 = \lambdasupd{d} \norm{f}_{L^2(\R^d)}^2$ holds if and only if $\fkn{k}{n}$ satisfies
			\begin{enumerate}[label={$\textup{(\ref{item:operator norm Schrodinger}.\roman*)}$}]
				\item \label{item:support condition Schrodinger}
				$\supp{ \fkn{k}{n} } \subset \set{ r > 0 }{ \lambdakd{k}{d}(r) \geq \lambdasupd{d} }$
			\end{enumerate}
			for each $k$ and $n$.
			As a consequence, extremisers exist if and only if there exists $k \in \N$ such that the Lebesgue measure of $\set{ r > 0 }{ \lambdakd{k}{d}(r) = \lambdasupd{d} }$ is non-zero.
		\end{eqenumerate}
		We see that \eqref{item:S^*Sf}, \eqref{item:S^*Sf decomposed}, \eqref{item:operator norm} and \ref{item:support condition} in Theorem \ref{thm:Dirac} are analogues of \eqref{item:S^*Sf Schrodinger}, \eqref{item:S^*Sf decomposed Schrodinger}, \eqref{item:operator norm Schrodinger} and \ref{item:support condition Schrodinger}, respectively.
		We also remark that in the case of the massive Dirac equation, extremisers must satisfy the condition \ref{item:eigenspace condition} in addition to \ref{item:support condition}.
	\end{remark}
	\begin{proof}[Proof of \eqref{item:S^*Sf}]
		Using
		\begin{gather}
			\matrixIN{\dimN} = \sum_{ \nu \in \{ 1, -1 \} }  \Pn(\xi) , 
			\\
			\Diracphim{m}(\xi) \Pn(\xi) = \nu \phim{m}(\abs{\xi}) \Pn(\xi) ,
		\end{gather}
		we get
		\begin{align}
			\exp( i t \Diracphim{m}(\xi) )
			= \sum_{ \nu \in \{ 1, -1 \} }  \exp( i \nu t \phim{m}(\abs{\xi}) ) \Pn(\xi) ,
		\end{align}
		so that 
		\begin{align}
			&\quad \exp( i t \Diracphim{m}(\xi) ) \exp( - i t \Diracphim{m}(\eta) ) \\
			&= \Bigg( \sum_{ \nu_1 \in \{ 1, -1 \} }  \exp( i \nu_1 t \phim{m}(\abs{\xi}) ) P_{\nu_1}(\xi) \Bigg) \Bigg( \sum_{ \nu_2 \in \{ 1, -1 \} }  \exp( - i \nu_2 t \phim{m}(\abs{\eta}) ) P_{\nu_2}(\eta) \Bigg) \\
			&= \sum_{ \nu_1, \nu_2 \in \{ 1, -1 \} } \exp( i t ( \nu_1 \phim{m}(\abs{\xi}) - \nu_2 \phim{m}(\abs{\eta}) ) ) P_{\nu_1}(\xi) P_{\nu_2}(\eta)  . 
		\end{align}
		Hence, we obtain
		\begin{equation}
			\int_{t \in \R} \exp( i t \Diracphim{m}(\xi) ) \exp( - i t \Diracphim{m}(\eta) ) \, dt 
			= 2 \pi \delta( \phim{m}(\abs{\xi}) - \phim{m}(\abs{\eta}) ) \sum_{ \nu \in \{ 1, -1 \} } \Pn(\xi) \Pn(\eta) , 
		\end{equation} 
		or precisely
		\begin{align}
			&\quad \int_{\eta \in \R^d} \int_{t \in \R} \exp( i t \Diracphim{m}(\xi) ) \exp( - i t \Diracphim{m}(\eta) ) f(\eta) \, dt \, d\eta \\
			&= 2 \pi \abs{\xi}^{d-2} \phim{m}(\abs{\xi}) \sum_{ \nu \in \{ 1, -1 \} } \int_{ \eta \in \S^{d-1} } \Pn(\xi) \Pn(\abs{\xi} \eta)  f(\abs{\xi} \eta) \, d\sigma(\eta) .
		\end{align}
		Therefore, we conclude that
		\begin{align}
			&(2\pi)^{d/2-1} \adjoint{\DiracS} \DiracS f(\xi) \\
			&= (2\pi)^{-1} \int_{t \in \R} \int_{x \in \R^d} w(\abs{x})^{1/2} e^{-ix \cdot \xi} \phim{m}(\abs{\xi})^{-1/2} \psi(\abs{\xi}) e^{it \Diracphim{m}(\xi)} \DiracS f(x, t) \, dx \, dt \\
			&= (2\pi)^{-(d/2+1)} \phim{m}(\abs{\xi})^{-1/2} \psi(\abs{\xi}) \\
			&\times \int_{t \in \R} \int_{x \in \R^d} w(\abs{x})^{1/2} e^{-ix \cdot \xi} e^{it \Diracphim{m}(\xi)} \mleft( w(\abs{x})^{1/2} \int_{ \eta \in \R^d } e^{ix \cdot \eta} \phim{m}(\abs{\eta})^{-1/2} \psi(\abs{\eta}) e^{-it \Diracphim{m}(\eta)} f(\eta) \, d \eta \mright) \, dx \, dt \\
			&=(2\pi)^{-(d/2+1)} \phim{m}(\abs{\xi})^{-1/2} \psi(\abs{\xi}) \\
			&\times \int_{ \eta \in \R^d }  \mleft(\int_{x \in \R^d} w(\abs{x}) e^{- ix \cdot (\xi - \eta)} \, dx \mright) \mleft( \int_{t \in \R} \exp( i t \Diracphim{m}(\xi) ) \exp( - i t \Diracphim{m}(\eta) ) \, dt  \mright) \, \phim{m}(\abs{\eta})^{-1/2} \psi(\abs{\eta}) f(\eta) \, d\eta \\
			&=  \phim{m}(\abs{\xi})^{-1/2} \psi(\abs{\xi}) \\
			&\times \sum_{ \nu \in \{ 1, -1 \} }  \int_{ \eta \in \R^d } \Fwd{d}( \abs{ \xi - \eta } ) \delta( \phim{m}(\abs{\xi}) - \phim{m}(\abs{\eta}) ) \Pn(\xi) \Pn(\eta) \phim{m}(\abs{\eta})^{-1/2} \psi(\abs{\eta}) f(\eta) \, d\eta \\
			&=  \abs{\xi}^{d-2} \phim{m}(\abs{\xi})^{1/2} \psi(\abs{\xi}) \sum_{ \nu \in \{ 1, -1 \} } \int_{ \eta \in \S^{d-1} } \Fwd{d}( \abs{ \xi -  \abs{\xi} \eta } ) \Pn(\xi) \Pn(\abs{\xi} \eta) \phim{m}(\abs{\xi})^{-1/2} \psi(\abs{\xi}) f(\abs{\xi}\eta) \, d\sigma(\eta) \\
			&=  \abs{\xi}^{d-2} (\psi(\abs{\xi}))^2 \sum_{ \nu \in \{ 1, -1 \} } \Pn(\xi) \int_{ \eta \in \S^{d-1} } \Fwd{d}( \abs{ \xi -  \abs{\xi} \eta } ) \Pn(\abs{\xi} \eta) f(\abs{\xi}\eta) \, d\sigma(\eta)
		\end{align}
		holds.
	\end{proof}
	\begin{proof}[Proof of \eqref{item:PE=EQ}]
		By \eqref{item:transformation matrix} of Theorem \ref{thm:Ekn}, we have
		\begin{align}
			\Pn(\xi) \Eknm{k}{n}(\xi / \abs{\xi}) 
			&= \frac{1}{2} \mleft( \matrixIN{\dimN} + \nu \frac{ \Diracphim{m}(\xi) }{ \phim{m}(\abs{\xi}) } \mright) \Eknm{k}{n}(\xi / \abs{\xi}) \\
			&= \frac{1}{2} \mleft( \matrixIN{\dimN} + \frac{ \nu }{ \phim{m}(\abs{\xi}) } \mleft( \abs{\xi} \Diracphim{0}(\xi / \abs{\xi}) + m \gammaj{d+1} \mright) \mright) \Eknm{k}{n}(\xi / \abs{\xi}) \\
			\underset{\eqref{eq:Ekn sigma_1},\eqref{eq:Ekn sigma_3}}&{=} \frac{1}{2} \Eknm{k}{n}(\xi / \abs{\xi}) \mleft( \matrixIN{2} + \frac{ \nu }{ \phim{m}(\abs{\xi}) } \mleft( \abs{\xi} \sigmaj{1} + m \sigmaj{3} \mright) \mright) \\
			&= \Eknm{k}{n}(\xi / \abs{\xi}) \Qmn(\abs{\xi}) . \quad \qedhere
		\end{align}
	\end{proof}
	\begin{proof}[Proof of \eqref{item:S^*Sf decomposed}]
		Using \eqref{eq:Funk-Hecke Ekn} and \eqref{item:PE=EQ}, we have
		\begin{align}
			&\quad  \abs{\xi}^{d-2}  (\psi(\abs{\xi}) )^2 \int_{ \eta \in \S^{d-1} } \Fwd{d}( \abs{ \xi -  \abs{\xi} \eta } ) \Pn(\abs{\xi} \eta) f(\abs{\xi}\eta) \, d\sigma(\eta)\\
			&=  \abs{\xi}^{(d-3)/2} (\psi(\abs{\xi}) )^2 \sum_{\mu \in \{1, -1\}} \sumk \sumn \int_{ \eta \in \S^{d-1} } \Fwd{d}( \abs{ \xi -  \abs{\xi} \eta } ) \Pn(\abs{\xi} \eta) \Eknm{k}{n}(\eta) \fknm{k}{n}(\abs{\xi}) \, d\sigma(\eta)\\
			\underset{\eqref{item:PE=EQ}}&{=}  \abs{\xi}^{-(d-1)/2} \sum_{\mu \in \{1, -1\}} \sumk \sumn \mleft(  \abs{\xi}^{d-2}  (\psi(\abs{\xi}) )^2 \int_{ \eta \in \S^{d-1} } \Fwd{d}(\abs{ \xi -  \abs{\xi} \eta }) \Eknm{k}{n}(\eta) \, d\sigma(\eta) \mright) \,\Qmn(\abs{\xi}) \fknm{k}{n}(\abs{\xi}) \\
			\underset{\eqref{eq:Funk-Hecke Ekn}}&{=} 2 (2\pi)^{d/2 - 1} \abs{\xi}^{-(d-1)/2} \sum_{\mu \in \{1, -1\}} \sumk \sumn \Eknm{k}{n}(\xi / \abs{\xi}) \Lambdak{k}(\abs{\xi}) \Qmn(\abs{\xi}) \fknm{k}{n}(\abs{\xi}) .
		\end{align}
		Substituting this into \eqref{item:S^*Sf}, we get
		\begin{align}
			&\quad \adjoint{\DiracS} \DiracS f(\xi) \\
			&= 2 \abs{\xi}^{-(d-1)/2} \sum_{ \nu \in \{ 1, -1 \} } \Pn(\xi) \sum_{\mu \in \{1, -1\}} \sumk \sumn \Eknm{k}{n}(\xi / \abs{\xi}) \Lambdak{k}(\abs{\xi}) \Qmn(\abs{\xi}) \fknm{k}{n}(\abs{\xi})  \\
			\underset{\eqref{item:PE=EQ}}&= 2 \abs{\xi}^{-(d-1)/2} \sum_{ \nu \in \{ 1, -1 \} } \sum_{\mu \in \{1, -1\}} \sumk \sumn \Eknm{k}{n}(\xi / \abs{\xi}) \Qmn(\abs{\xi}) \Lambdak{k}(\abs{\xi}) \Qmn(\abs{\xi}) \fknm{k}{n}(\abs{\xi})  \\
			&= \abs{\xi}^{-(d-1)/2}
			\sum_{\mu \in \{1, -1\}} \sumk \sumn \Eknm{k}{n}(\xi / \abs{\xi}) \DiracLambdak{k}(\abs{\xi}) \fknm{k}{n}(\abs{\xi})  .
			\quad \qedhere
		\end{align}
	\end{proof}
	\begin{proof}[Proof of \eqref{item:maximal eigenvalue}]
		Note that we have the following identities:
		\begin{align}
			\sigmaj{3} \Lambdak{k} \sigmaj{3} & =  \Lambdak{k} , \\
			\sigmaj{3} \Lambdak{k} \sigmaj{1} + \sigmaj{1} \Lambdak{k} \sigmaj{3} &= \mu ( \lambdakd{k}{d} - \lambdakd{k+1}{d} ) \sigmaj{1} , \\
			\Lambdak{k} + \sigmaj{1} \Lambdak{k} \sigmaj{1} &= ( \lambdakd{k}{d} + \lambdakd{k+1}{d} ) \matrixIN{2} , \\
			\Lambdak{k} - \sigmaj{1} \Lambdak{k} \sigmaj{1} &= \mu ( \lambdakd{k}{d} - \lambdakd{k+1}{d} ) \sigmaj{3} .
		\end{align}
		Using these, we obtain
		\begin{align}
			\DiracLambdak{k}(r) &= 2 \sum_{ \nu \in \{ 1, -1 \} } \Qmn(r) \Lambdak{k}(r) \Qmn(r) \\
			&= \Lambdak{k} + \frac{1}{ \phim{m}^2 } ( r \sigmaj{1} +  m \sigmaj{3} ) \Lambdak{k} ( r \sigmaj{1} +  m \sigmaj{3} ) \\
			&= \Lambdak{k} + \frac{1}{ \phim{m}^2 } ( r^2 \sigmaj{1} \Lambdak{k} \sigmaj{1} +  m r ( \sigmaj{3} \Lambdak{k} \sigmaj{1} + \sigmaj{1} \Lambdak{k} \sigmaj{3} ) + m^2 \sigmaj{3} \Lambdak{k} \sigmaj{3} ) \\
			&= \Lambdak{k} + \sigmaj{1} \Lambdak{k} \sigmaj{1} + \frac{1}{ \phim{m}^2 } (  \mu m r ( \lambdakd{k}{d} - \lambdakd{k+1}{d} ) \sigmaj{1} + m^2 ( \Lambdak{k} - \sigmaj{1} \Lambdak{k} \sigmaj{1} ) ) \\
			&= ( \lambdakd{k}{d} + \lambdakd{k+1}{d} ) \matrixIN{2} + \frac{\mu m}{ \phim{m}^2 } ( \lambdakd{k}{d} - \lambdakd{k+1}{d} ) (  r \sigmaj{1} + m \sigmaj{3}  ) .
		\end{align}
		Now notice that the eigenvalues of $ r \sigmaj{1} + m \sigmaj{3}$ are $\pm \phim{m}(r)$. 
		Therefore, the maximal eigenvalue of $\DiracLambdak{k}(r)$ and its associated eigenspace are
		\begin{equation} \tag*{\eqref{eq:Diraclambdak}}
			\Diraclambdak{k}(r)
			= \lambdakd{k}{d}(r) + \lambdakd{k+1}{d}(r) + \frac{m}{\phim{m}(r)} \abs{ \lambdakd{k}{d}(r) - \lambdakd{k+1}{d}(r) } 
		\end{equation}
		and
		\begin{align}
			&\quad \Wk{k} (r) \\
			&= 
			\begin{dcases}
				\C^2, & \phantom{\mu} m ( \lambdakd{k}{d}(r) - \lambdakd{k+1}{d}(r) ) = 0 , \\
				\text{the eigenspace of $ r \sigmaj{1} + m \sigmaj{3}$ associated with $\phim{m}(r)$}, 
				& \mu m ( \lambdakd{k}{d}(r) - \lambdakd{k+1}{d}(r) )  > 0 , \\
				\text{the eigenspace of $ r \sigmaj{1} + m \sigmaj{3}$ associated with $- \phim{m}(r)$},  
				& \mu m ( \lambdakd{k}{d}(r) - \lambdakd{k+1}{d}(r) )  < 0 ,
			\end{dcases}
		\end{align}
		respectively.
	\end{proof}
	\begin{proof}[Proof of \eqref{item:operator norm}]
		For simplicity, we write
		\begin{equation}
			\norm{\DiracS} \coloneqq \norm{\DiracS}_{L^2(\R^d, \C^\dimN) \to L^2(\R^{d+1}, \C^\dimN)} .
		\end{equation}
		By \eqref{item:S^*Sf decomposed} and \eqref{item:maximal eigenvalue}, it is obvious that
		\begin{align}
			\norm{\DiracS f}_{L^2(\R^{d+1}, \C^\dimN)}^2 
			\underset{\eqref{item:S^*Sf decomposed}}&{=} \sum_{\mu \in \{1, -1\}} \sumk \sumn \int_0^\infty \innerproduct{ \DiracLambdak{k}(r) \fknm{k}{n}(r) }{ \fknm{k}{n}(r) }_{\C^2} \, dr \\
			\underset{\eqref{item:maximal eigenvalue}}&\leq \sum_{\mu \in \{1, -1\}} \sumk \sumn \int_0^\infty \Diraclambdak{k}(r) \abs{ \fknm{k}{n}(r) }^2 \, dr  \\
			&\leq \Diraclambdasup \norm{f}_{L^2(\R^d, \C^\dimN)}^2 ,
		\end{align}
		hence
		\begin{equation}
			\norm{\DiracS}^2 \leq  \Diraclambdasup .
		\end{equation}
		
		To see the equality $\norm{\DiracS}^2 = \Diraclambdasup$, 
		we will show that for any $0 < C < \Diraclambdasup$, 
		there exists $f \in L^2(\R^d, \C^\dimN) \setminus \{0\}$ 
		such that $\norm{\DiracS f}^{2}_{L^2(\R^{d+1}, \C^\dimN)} \geq C \norm{f}_{L^2(\R^d, \C^\dimN)}^2$.
		Fix $0 < C < \Diraclambdasup$. 
		Then, by the definition of $\Diraclambdasup$, 
		we can choose $k_0 \in \N$ such that the Lebesgue measure of
		\begin{equation}
			\set{r > 0}{ \Diraclambdak{k_0}(r) \geq C }	
		\end{equation}
		is non-zero (possibly infinite). 
		Now suppose that
		\begin{eqenumerate}
			\item \label{item:k0 only}
			$\fknm{k}{n} = 0$ if $k \ne k_0$, 
			\item \label{item:support condition epsilon}
			$\supp{ \fknm{k_0}{n} } \subset \set{r > 0}{ \Diraclambdak{k_0}(r) \geq C }	$ , 
			\item \label{item:eigenspace condition k0}
			$\fknm{k_0}{n}(r) \in \Wk{k_0}(r)$ for almost every $r > 0$ .
		\end{eqenumerate}
		Then we have
		\begin{align}
			\norm{\DiracS f}_{L^2(\R^{d+1}, \C^\dimN)}^2 
			\underset{\eqref{item:k0 only}}&{=}  \sum_{\mu \in \{1, -1\}} \sumn \int_0^\infty \innerproduct{ \DiracLambdak{k_0}(r) \fknm{k_0}{n}(r) }{ \fknm{k_0}{n}(r) }_{\C^2} \, dr \\
			\underset{\eqref{item:eigenspace condition k0}}&{=} \sum_{\mu \in \{1, -1\}} \sumn \int_0^\infty \Diraclambdak{k_0}(r) \abs{ \fknm{k_0}{n}(r) }^2 \, dr  \\
			\underset{\eqref{item:support condition epsilon}}&\geq C \norm{f}_{L^2(\R^d, \C^\dimN)}^2 ,
		\end{align}
		so that the equality $\norm{\DiracS}^2 = \Diraclambdasup$ holds.
		
		Using a similar argument, we can show that if
		\begin{eqenumerate}
			\item[\ref{item:support condition}]
			$\supp{ \fknm{k}{n} } \subset \set{r > 0}{ \Diraclambdak{k}(r) \geq \Diraclambdasup }$,
			\item[\ref{item:eigenspace condition}]
			$\fknm{k}{n}(r) \in \Wk{k}(r)$ for almost every $r > 0$, 
		\end{eqenumerate}
		then $\norm{\DiracS f}_{L^2(\R^{d+1}, \C^\dimN)}^2 = \Diraclambdasup \norm{f}_{L^2(\R^d, \C^\dimN)}^2$ holds. 
		Indeed, we have
		\begin{align}
			\norm{\DiracS f}_{L^2(\R^{d+1}, \C^\dimN)}^2 
			&= \sum_{\mu \in \{1, -1\}} \sumk \sumn \int_0^\infty \innerproduct{ \DiracLambdak{k}(r) \fknm{k}{n}(r) }{ \fknm{k}{n}(r) }_{\C^2} \, dr \\
			\underset{\text{\ref{item:eigenspace condition}}}&{=} \sum_{\mu \in \{1, -1\}} \sumk \sumn \int_0^\infty \Diraclambdak{k}(r) \abs{ \fknm{k}{n}(r) }^2 \, dr  \\
			\underset{\text{\ref{item:support condition}}}&=  \Diraclambdasup \norm{f}_{L^2(\R^d, \C^\dimN)}^2 .
		\end{align}
		Conversely, suppose that $\norm{\DiracS f}_{L^2(\R^{d+1}, \C^\dimN)}^2 = \Diraclambdasup \norm{f}_{L^2(\R^d, \C^\dimN)}^2$ holds.
		Then we have
		\begin{equation}
			 \sum_{\mu \in \{1, -1\}} \sumk \sumn \int_{0}^{\infty} \mleft( \Diraclambdasup \abs{ \fknm{k}{n}(r) }^2 - \innerproduct{
				\DiracLambdak{k}(r) \fknm{k}{n}(r)
			}
			{
				\fknm{k}{n}(r)
			}_{\C^2} \mright)  \, dr
			= \Diraclambdasup \norm{f}_{L^2(\R^d, \C^\dimN)}^2 - \norm{\DiracS f}_{L^2(\R^{d+1}, \C^\dimN)}^2 
			= 0 .
		\end{equation}
		On the other hand, we also have \begin{equation}
			\Diraclambdasup \abs{ \fknm{k}{n}(r) }^2 - \innerproduct{
				\DiracLambdak{k}(r) \fknm{k}{n}(r)
			}
			{
				\fknm{k}{n}(r)
			} 
			\geq
			( \Diraclambdasup - \Diraclambdak{k}(r) ) \abs{ \fknm{k}{n}(r) }^2 
			\geq
			0 .
		\end{equation}
		Therefore, we obtain
		\begin{eqenumerate}
			\item\label{item:pre support condition}
			$( \Diraclambdasup - \Diraclambdak{k}(r) ) \abs{ \fknm{k}{n}(r) }^2 = 0$, 
			\item[\ref{item:eigenspace condition}]
			$\fknm{k}{n}(r) \in \Wk{k}(r)$ for almost every $r > 0$,
		\end{eqenumerate}
		and \eqref{item:pre support condition} implies \ref{item:support condition}.
	\end{proof}
	\section{Explicit values} \label{section:explicit values}
	In this section, we prove Theorems \ref{thm:type A Dirac}, \ref{thm:type B Dirac}, \ref{thm:type C Dirac}, and \ref{thm:type C 2D Dirac}.
	In order to prove Theorem \ref{thm:type A Dirac}, we use the following known results regarding a product of modified Bessel functions:
	\begin{lemma} \label{lem:product of modified Bessel functions}
		Let $\BesselI{\mu}$ and $\BesselK{\mu}$ denote the modified Bessel functions of the first and second kinds of order $\mu \in \R$.
		Then the following hold:
		\begin{eqenumerate}
			\item \label{item:product of modified Bessel functions integral formula}
			For any $r > 0$ and $\mu \geq -1$, 
			\begin{equation}
				\int_0^\infty \frac{t}{r^2 + t^2} ( \BesselJ{\mu}(t) )^2 \, dt = \BesselI{\mu}(r) \BesselK{\mu}(r) 
			\end{equation}
			holds (see \citet[6.535]{GR2014}).
			\item \label{item:product of modified Bessel functions IK is decreasing for r}
			For each fixed $\mu \geq -1$, $r \mapsto \BesselI{\mu}(r) \BesselK{\mu}(r)$ is strictly decreasing on $(0, \infty)$ (this follows from \eqref{item:product of modified Bessel functions integral formula} immediately).
			\item \label{item:product of modified Bessel functions rIK is increasing for r}
			For each fixed $\mu \geq 1/2$, $r \mapsto r \BesselI{\mu}(r) \BesselK{\mu}(r)$ is strictly increasing on $(0, \infty)$ (see \citet[Theorem 4.1]{Har1977} and \citet[Theorem 3.1]{Seg2021}).
			\item \label{item:product of modified Bessel functions IK is decreasing for mu}
			For each fixed $r > 0$, $\mu \mapsto \BesselI{\mu}(r) \BesselK{\mu}(r)$ is strictly decreasing on $[0, \infty)$ (see \citet[page 530]{BP2012}).
			\item \label{item:product of modified Bessel functions limits}
			For each fixed $\mu > 0$, the following hold (see \citet[\href{https://dlmf.nist.gov/10.25.E3}{10.25.3}, \href{http://dlmf.nist.gov/10.30}{\textsection10.30}]{DLMF}):
			\begin{equation}
				\lim_{r \downarrow 0} \BesselI{\mu}(r) \BesselK{\mu}(r) = \frac{1}{2 \mu}, \quad \lim_{r \to \infty} r \BesselI{\mu}(r) \BesselK{\mu}(r) = 1/2 .
			\end{equation}
			\item \label{item:product of modified Bessel functions mu = 3/2}
			For $\mu = 3/2$, $r \mapsto (1+r^2)^{1/2} \BesselI{3/2}(r) \BesselK{3/2}(r)$ is strictly increasing on $(0, \infty)$ (see \citet[Proof of Theorem 1.7]{BS2017}).
			\item \label{item:product of modified Bessel functions mu = 1/2}
			For $\mu = 1/2$, $r \mapsto (1+r^2)^{1/2} \BesselI{1/2}(r) \BesselK{1/2}(r)$ is strictly decreasing on 
			$(0, \infty)$ (see \citet[Proof of Theorem 1.7]{BS2017}).
		\end{eqenumerate}
	\end{lemma}
	\begin{proof}[Proof of Theorem \ref{thm:type A Dirac}]
		Substituting
		\begin{equation}
			w(r) = (1+r^2)^{-1} , \quad (\psi(r))^2 = (1+r^2)^{1/2} 
		\end{equation}
		into \eqref{eq:lambdak 1} and changing variable of integration, we obtain
		\begin{align}
			\lambdakd{k}{d}(r) 
			&= \pi (1+r^2)^{1/2}  \int_0^\infty \frac{t}{1 + t^2} ( \BesselJ{k + d/2 - 1}(r t) )^2 \, dt \\
			&= \pi (1+r^2)^{1/2}  \int_0^\infty \frac{t}{r^2 + t^2} ( \BesselJ{k + d/2 - 1}(t) )^2 \, dt \\
			&= \pi (1+r^2)^{1/2} \BesselI{k+d/2-1}(r) \BesselK{k+d/2-1}(r) 
		\end{align}
		by using the formula \eqref{item:product of modified Bessel functions integral formula}.
		Since $k \mapsto \BesselI{k + d/2 - 1}(r) \BesselK{k + d/2 - 1}(r)$ is strictly decreasing for each fixed $r > 0$ by \eqref{item:product of modified Bessel functions IK is decreasing for mu}, we know that $k \mapsto \Diraclambdak{k}(r)$ is also strictly decreasing.
		Hence, we have
		\begin{equation}
			\Diraclambdasup = \sup_{k \in \N} \sup_{r > 0} \Diraclambdak{k}(r) = \sup_{r > 0} \Diraclambdak{0}(r) .
		\end{equation}
		Furthermore, 
		\begin{equation}
			\set{ r > 0 }{ \Diraclambdak{k}(r) = \Diraclambdasup }
		\end{equation}
		is a null set with respect to the Lebesgue measure for each $k \in \N$, since $\Diraclambdak{k}$ is a non-constant real analytic function on $(0, \infty)$.
		Therefore, it is enough to consider 
		$\sup_{r > 0} \Diraclambdak{0}(r)$. 
		
		First, we consider the case $d \geq 5$.
		In this case, we have
		\begin{align}
			&\quad \Diraclambdak{0}(r) \\
			\underset{\eqref{item:product of modified Bessel functions IK is decreasing for mu}}&{<} 2\pi (1+r^2)^{1/2} \BesselI{3/2}(r) \BesselK{3/2}(r) \\
			\underset{\eqref{item:product of modified Bessel functions mu = 3/2}}&{<} 2\pi \lim_{r \to \infty} (1+r^2)^{1/2} \BesselI{3/2}(r) \BesselK{3/2}(r) \\
			\underset{\eqref{item:product of modified Bessel functions limits}}&{=} \pi .
		\end{align}
		On the other hand, \eqref{item:product of modified Bessel functions limits} implies
		\begin{equation}
			\lim_{r \to \infty} \Diraclambdak{0}(r) = \pi .
		\end{equation}	
		Hence, we conclude that 
		\begin{equation}
			\sup_{r > 0} \Diraclambdak{0}(r) 
			= \pi
		\end{equation}
		holds. 
		
		Second, we consider the case $d = 3$, $m > 0$.
		In this case, we have
		\begin{align}
			&\quad \Diraclambdak{0}(r) \\
			\underset{\eqref{item:product of modified Bessel functions IK is decreasing for mu}}&{<} 2\pi  (1+r^2)^{1/2} \BesselI{1/2}(r) \BesselK{1/2}(r) \\
			\underset{\eqref{item:product of modified Bessel functions mu = 1/2}}&{<} 2 \pi \lim_{r \downarrow 0} (1+r^2)^{1/2} \BesselI{1/2}(r) \BesselK{1/2}(r) \\
			\underset{\eqref{item:product of modified Bessel functions limits}}&{=} 2\pi .
		\end{align}
		On the other hand, \eqref{item:product of modified Bessel functions limits} implies
		\begin{equation}
			\lim_{r \downarrow 0} \Diraclambdak{0}(r) = 2\pi.
		\end{equation}
		Hence, we conclude that 
		\begin{equation}
			\sup_{r > 0} \Diraclambdak{0}(r) 
			= 2\pi
		\end{equation}
		holds.
		
		Now we consider the cases $d=3, 4$, $m=0$.
		For simplicity, we write
		\begin{equation} \label{eq:type A Dirac fd m=0}
			\fd{\mu}(r) \coloneqq (1+r^2)^{1/2} ( \BesselI{\mu}(r) \BesselK{\mu}(r) + \BesselI{\mu+1}(r) \BesselK{\mu+1}(r) ) 
		\end{equation}
		so that 
		\begin{equation}
			\Diraclambdak{0}(r) = \pi \fd{d/2-1}(r) .
		\end{equation}
		See Figure \ref{fig:graph of Diraclambda0 type A} for the graph of $\fd{\mu}$ in the cases $\mu = 1/2, 1, 3/2$.
		We are going to show that 
		\begin{equation}
			\sup_{r > 0} \fd{\mu}(r)
			= \begin{cases}
				4/3 , & \mu = 1/2, \\
				1 , & \mu = 1 .
			\end{cases}
		\end{equation}
		
		In the case $\mu=1/2$, substituting 
		\begin{align}
			\BesselI{1/2}(r) &= \frac{ e^r - e^{-r} }{ (2 \pi r)^{1/2} } , \\
			\BesselK{1/2}(r) &= \frac{ \pi^{1/2} e^{-r} }{ (2 r)^{1/2} } , \\
			\BesselI{3/2}(r) &= \frac{ e^{r} + e^{-r} }{ (2 \pi r)^{1/2} } - \frac{ e^{r} - e^{-r} }{ (2 \pi r^3)^{1/2} } , \\
			\BesselK{3/2}(r) &= \frac{ \pi^{1/2} e^{-r} (1+r) }{ 2^{1/2} r^{3/2} } 
		\end{align}
		into \eqref{eq:type A Dirac fd m=0}, we get
		\begin{align}
			\fd{1/2}(r)
			&= \frac{ (1+r^2)^{1/2} (2 e^{2r} r^2  - e^{2r} + 2r + 1) }{ 2 e^{2r} r^{3} } ,
			\intertext{which implies}
			\frac{d}{dr} \fd{1/2}(r) 
			&= \frac{ 3 e^{2r}  - ( 4r^4 + 4 r^3 + 6r^2 + 6r + 3 ) }{ 2 r^4 e^{2r} ( 1 + r^2 )^{1/2} } .
		\end{align}
		Let 
		\begin{equation}
			\gd{3}(r) \coloneqq 3 e^{2r}  - ( 4r^4 + 4 r^3 + 6r^2 + 6r + 3 ) . 
		\end{equation}
		Then, the $k$-th derivative of $\gd{3}$ ($1 \leq k \leq 4$) is
		\begin{align}
			\gd{3}^{(1)}(r) &= 2 ( 3 e^{2r} - 8r^3 - 6r^2 - 6r - 3 ) , \\
			\gd{3}^{(2)}(r) &= 12 ( e^{2r} - 4r^2 - 2r - 1 ) , \\
			\gd{3}^{(3)}(r) &= 24 ( e^{2r} - 4r - 1 ) , \\
			\gd{3}^{(4)}(r) &= 48 ( e^{2r} - 2 ) .
		\end{align}
		We see that 
		\begin{equation}
			\gd{3}^{(k)}(0) = 0 
		\end{equation}
		for each $0 \leq k \leq 3$, and 
		\begin{equation}
			\gd{3}^{(4)}(r) \begin{cases}
				< 0 , & 0 \leq r < \dfrac{\log {2}}{2} , \\
				= 0 , &  r = \dfrac{\log {2}}{2} , \\
				> 0 , & r > \dfrac{\log {2}}{2} 
			\end{cases}
		\end{equation}
		holds.
		Hence, for each $0 \leq k \leq 3$, there exists $r_k > 0$ such that 
		\begin{equation}
			\gd{3}^{(k)}(r) \begin{cases}
				= 0 , & r = 0 , \\
				< 0 , & 0 < r < r_k , \\
				= 0 , & r = r_k , \\
				> 0 , & r > r_k .
			\end{cases}
		\end{equation}
		As a consequence, $\fd{1/2}$ is strictly decreasing on $(0, r_0]$ and strictly increasing on $[r_0, \infty)$.
		On the other hand, using \eqref{item:product of modified Bessel functions limits},
		we get
		\begin{equation}
			\lim_{r \downarrow 0} \fd{1/2}(r) = \frac{4}{3} , \quad
			\lim_{r \to \infty} \fd{1/2}(r) = 1 .
		\end{equation}
		Therefore, we conclude that
		\begin{equation}
			\sup_{r > 0} \fd{1/2}(r) = \frac{4}{3} 
		\end{equation}
		holds.
		
		Finally, we consider the case $\mu = 1$.
		Since we have
		\begin{equation}
			\lim_{r \to \infty} \fd{1}(r) = 1 
		\end{equation}
		by \eqref{item:product of modified Bessel functions limits},
		it is enough to show that 
		\begin{equation}
			\fd{1}(r) < 1
		\end{equation}
		holds for any $r>0$.
		In the case $r \geq 1$, notice that we have
		\begin{equation}
			\fd{1}(r) < \fd{1/2}(r) 
		\end{equation}
		by \eqref{item:product of modified Bessel functions IK is decreasing for mu}.
		Furthermore, since $\fd{1/2}$ is strictly decreasing on $(0, r_0]$ and strictly increasing on $[r_0, \infty)$ as we saw in the case $\mu = 1/2$, 
		we have
		\begin{equation}
			\sup_{r \geq 1} \fd{1/2}(r) = \max\{ \fd{1/2}(1) , \fd{1/2}(\infty) \} = \max\{ \sqrt{2} ( \BesselI{1/2}(1) \BesselK{1/2}(1) + \BesselI{3/2}(1) \BesselK{3/2}(1) ) , 1 \} .
		\end{equation}
		On the other hand, for any $0 < r \leq 1$, 
		\eqref{item:product of modified Bessel functions IK is decreasing for r} and \eqref{item:product of modified Bessel functions limits} imply 
		\begin{equation}
			( \BesselI{1}(r) \BesselK{1}(r) + \BesselI{2}(r) \BesselK{2}(r) )^2 \leq ( \BesselI{1}(+0) \BesselK{1}(+0) + \BesselI{2}(+0) \BesselK{2}(+0) )^2 = 9/16 ,
		\end{equation}
		and \eqref{item:product of modified Bessel functions rIK is increasing for r} implies
		\begin{equation}
			r^2 ( \BesselI{1}(r) \BesselK{1}(r) + \BesselI{2}(r) \BesselK{2}(r) )^2 \leq ( \BesselI{1}(1) \BesselK{1}(1) + \BesselI{2}(1) \BesselK{2}(1) )^2 .
		\end{equation}	
		Hence, we have
		\begin{equation}
			\sup_{0 < r \leq 1} \fd{1}(r) \leq \mleft( ( \BesselI{1}(1) \BesselK{1}(1) + \BesselI{2}(1) \BesselK{2}(1) )^2 + 9/16 \mright)^{1/2} .
		\end{equation}
		In order to complete the proof, we need to show that
		\begin{align}
			( \BesselI{1/2}(1) \BesselK{1/2}(1) + \BesselI{3/2}(1) \BesselK{3/2}(1) )^2 &< 1/2 , \\
			( \BesselI{1}(1) \BesselK{1}(1) + \BesselI{2}(1) \BesselK{2}(1) )^2 &< 7/16 .
		\end{align}
		To verify these, we use the following:
		\begin{align}
			( \BesselI{1/2}(1) \BesselK{1/2}(1) + \BesselI{3/2}(1) \BesselK{3/2}(1) )^2 &= \frac{1}{4} ( 1 + 3 e^{-2} )^2 , \\
			e &= 2.71828 \ldots > 2.71 , \\ 
			\BesselI{1}(1) &= 0.56515 \ldots < 0.57 , \\
			\BesselI{2}(1) &= 0.13574 \ldots < 0.14 , \\
			\BesselK{1}(1) &= 0.60190 \ldots < 0.61 , \\
			\BesselK{2}(1) &= 1.62483 \ldots < 1.63 .
		\end{align}
		Then we obtain
		\begin{align}
			( \BesselI{1/2}(1) \BesselK{1/2}(1) + \BesselI{3/2}(1) \BesselK{3/2}(1) )^2 &< 0.496 < 1/2 , \\
			( \BesselI{1}(1) \BesselK{1}(1) + \BesselI{2}(1) \BesselK{2}(1) )^2 &< 0.332 < 7/16 ,
		\end{align}
		as desired.
	\end{proof}
	\begin{proof}[Proof of Theorem \ref{thm:type B Dirac}]
		Substituting 
		\begin{equation}
			w(r) = r^{-s} , \quad
			( \psi(r) )^2 = r^{2-s}
		\end{equation}
		into \eqref{eq:lambdak 1}, we obtain
		\begin{align}
			\lambdakd{k}{d}(r) 
			&= \pi r^{2-s} \int_0^\infty t^{-s+1} ( \BesselJ{k + d/2 - 1}(r t) )^2 \, dt \\
			&= \pi r^{2-s} \frac{ r^{s-2} \Gamma(k + (d-s)/2) \Gamma(s-1) }{2^{s-1} ( \Gamma(s/2) )^2 \Gamma(k+(d+s)/2-1) } \\
			&= \pi^{1/2}
			\frac{  \Gamma((s-1)/2) \Gamma(k+(d-s)/2) }{ 2 \Gamma(s/2) \Gamma(k+(d+s)/2-1) } \eqqcolon c_k .
		\end{align}
		Here we used the following formulae (see \citet[\href{https://dlmf.nist.gov/5.5.E5}{(5.5.5)}, \href{https://dlmf.nist.gov/10.22.E57}{(10.22.57)}]{DLMF}):
		\begin{gather}
			\int_{0}^{\infty} t^{- \lambda} \BesselJ{\nu}(at) \BesselJ{\mu}(at) \, dt
			= \frac{a^{\lambda-1} \Gamma( (\nu + \mu - \lambda + 1 )/2 ) \Gamma(\lambda) }
			{2^{\lambda} \Gamma( (\lambda - \nu + \mu + 1) / 2 ) \Gamma( (\lambda + \nu - \mu + 1) / 2 )\Gamma( (\lambda + \nu + \mu + 1) / 2 ) } , \\
			\Gamma(s-1) = \pi^{-1/2} 2^{s-2} \Gamma( (s-1)/2 ) \Gamma( s/2 ) .
		\end{gather}
		Since 
		\begin{equation}
			\frac{ c_{k+1} }{ c_k } = \frac{ k+(d-s)/2 }{ k+(d+s)/2-1 } = 1 - \frac{ 2(s-1) }{ 2k + d + s - 2 } < 1 , 
		\end{equation}
		we see that $\{ c_k \}_{k \in \N}$ is strictly decreasing.
		Therefore, we have
		\begin{align}
			\sup_{k \in \N} \sup_{r > 0} \Diraclambdak{k}(r)
			&= \begin{cases}
				c_0 + c_1 , & m = 0 , \\
				2 c_0 , & m > 0 
			\end{cases} \\
			&= \begin{cases}
				\dfrac{2(d - 1)}{ d + s - 2 } c_0 , & m = 0 , \\
				2 c_0 , & m > 0 
			\end{cases} 
			\intertext{and}
			\set{ r > 0 }{ \Diraclambdak{k}(r) = \Diraclambdasup }
			&= 
			\begin{cases}
				\positiveR, & m = k = 0, \\
				\emptyset, & \text{$m > 0$ or $k \geq 1$}.
			\end{cases}
		\end{align}
		Thus, Theorem \ref{thm:type B Dirac} follows from Theorem \ref{thm:Dirac intro}.
	\end{proof}
	\begin{remark}
		We remark that \eqref{item:S^*Sf Schrodinger} implies that $c_k$ is an eigenvalue of $\adjoint{ \SchrodingerS } \SchrodingerS$ for each $k\in \N$,
		and $f \in L^2(\R^d) \setminus \{0\}$ is an eigenfunction if
		\begin{equation}
			f(\xi) = \abs{\xi}^{-(d-1)/2} \sumn \pkn{k}{n}(\xi / \abs{\xi}) \fkn{k}{n}(\abs{\xi}) .
		\end{equation}
		Similarly, in the case $m = 0$, \eqref{item:S^*Sf} implies that $c_k + c_{k+1}$ is an eigenvalue of $\adjoint{ \DiracS } \DiracS$ for each $k\in \N$,
		and $f \in L^2(\R^d, \C^\dimN) \setminus \{0\}$ is an eigenfunction if
		\begin{equation}
			f(\xi) = \abs{\xi}^{-(d-1)/2} \sum_{\mu \in \{ 1, -1 \}} \sumn \Eknm{k}{n}(\xi / \abs{\xi}) \fknm{k}{n}(\abs{\xi}) .
		\end{equation}
	\end{remark}
	\begin{proof}[Proof of Theorem \ref{thm:type C Dirac}]
		Let $w \in L^1(\positiveR) \setminus \{0\}$ be such that $\Fwd{d}$ is non-negative, and $\psi(r) = r^{1/2}$. 
		First, we show that
		\begin{equation} \label{eq:lambdak < lambda0}
			\lambdakd{k+1}{d}(r) < \lambdak{0}(r)
		\end{equation}
		holds for any $k \in \N$ and $r > 0$.
		This easily follows from \eqref{item:properties of lambdak 3} of Proposition \ref{prop:properties of lambdak}.
		In fact, $\lambdakd{k+1}{d}(r) \leq \lambdak{0}(r)$ is immediate from the non-negativity of $\Fwd{d}$. 
		Furthermore, substituting $( \psi(r) )^2 = r$ into \eqref{eq:lambdak 1}, we get
		\begin{equation}
			\lambdak{k}(r) = \pi r \int_0^\infty t w(t) ( \BesselJ{k+d/2-1}(rt) )^2 \, dt .
		\end{equation}
		Combining this with $w \in L^1(\positiveR) \setminus \{0\}$, we have $0 < \lambdak{k}(r) < \infty$. Therefore, we obtain the strict inequality \eqref{eq:lambdak < lambda0}.
		Now notice that \eqref{eq:lambdak < lambda0} implies 
		\begin{gather}
			\Diraclambdak{k}(r) < 2 \lambdak{0}(r) , \\
			\sup_{r>0} \mleft( \lambdak{0}(r) + \lambdak{1}(r) \mright) \leq \sup_{k \in \N} \sup_{r > 0} \Diraclambdak{k}(r) \leq 2 \sup_{k \in \N} \sup_{r > 0} \lambdakd{k}{d}(r) = 2 \sup_{r > 0} \lambdak{0}(r) .
		\end{gather}
		Hence, it is enough to show that
		\begin{gather}
			\sup_{r>0} \mleft( \lambdak{0}(r) + \lambdak{1}(r) \mright) = 2 \sup_{r > 0} \lambdak{0}(r) = 2 \norm{w}_{L^1(\positiveR)} .
		\end{gather}
		Substituting $( \psi(r) )^2 =  r$
		into \eqref{eq:lambdak 2} and changing variable of integration, we get
		\begin{align}
			\lambdak{0}(r) 
			&= \frac{ \pi^{1/2} }{ 2^{d/2-1} \Gamma( (d-1)/2 ) } r^{d-1} \int_{-1}^{1} \Fwd{d}( r \sqrt{ 2(1-t) } ) (1-t^2)^{(d-3)/2} \, dt \\
			&= \frac{ \pi^{1/2} }{ 2^{d/2-1} \Gamma( (d-1)/2 ) } \int_{0}^{2r} \Fwd{d}(t) \mleft(1 - \frac{t^2}{ 4r^2 } \mright)^{(d-3)/2} t^{d-2} \, dt ,\\
			\lambdak{1}(r) 
			&= \frac{ \pi^{1/2} }{ 2^{d/2-1} \Gamma( (d-1)/2 ) } r^{d-1} \int_{-1}^{1} \Fwd{d}( r \sqrt{ 2(1-t) } ) t (1-t^2)^{(d-3)/2} \, dt \\
			&= \frac{ \pi^{1/2} }{ 2^{d/2-1} \Gamma( (d-1)/2 ) } \int_{0}^{2r} \Fwd{d}(t) \mleft(1 - \frac{t^2}{ 2r^2 } \mright) \mleft(1 - \frac{t^2}{ 4r^2 } \mright)^{(d-3)/2} t^{d-2} \, dt ,
			\intertext{so that}
			\lambdak{0}(r) + \lambdak{1}(r) &= \frac{ \pi^{1/2} }{ 2^{d/2-2} \Gamma( (d-1)/2 ) } \int_{0}^{2r} \Fwd{d}(t) \mleft(1 - \frac{t^2}{ 4r^2 } \mright)^{(d-1)/2} t^{d-2} \, dt .
		\end{align}
		Therefore, $\lambdak{0}$ and $\lambdak{0} + \lambdak{1}$ are non-decreasing.
		Note that here we used the assumption $d \geq 3$ in order to guarantee the monotonicity of $\lambdak{0}$.
		On the other hand, we have
		\begin{align}
			2 \lim_{r \to \infty} \lambdak{0}(r)
			&= \lim_{r \to \infty} ( \lambdak{0}(r) + \lambdak{1}(r) ) \\
			&= \frac{ \pi^{1/2} }{2^{d/2-2} \Gamma( (d-1)/2 )  } \int_{0}^{\infty} \Fwd{d}( t ) t^{d-2} \, dt \\
			&= \frac{ \pi^{1/2} }{ 2^{d/2-2} \Gamma( (d-1)/2 )  \measure{\S^{d-1}} }  \int_{\xi \in \R^d} \Fourier[w( \abs{\variabledot} )]( \xi ) \cdot \abs{\xi}^{-1} \, d\xi \\
			&= \frac{ \pi^{1/2} }{ 2^{d/2-2} \Gamma( (d-1)/2 )  \measure{\S^{d-1}} }  \int_{x \in \R^d} w( \abs{x} ) \cdot \frac{ 2^{d/2-1} \Gamma( (d-1)/2 ) }{ \pi^{1/2} } \abs{x}^{-(d-1)} \, dx \\
			&= 2 \norm{w}_{L^1(\positiveR)} ,
		\end{align}
		which completes the proof.
	\end{proof}
	\begin{remark} \label{remark:type C 3D} 
		When $d = 3$,
		\begin{equation}
			\lambdak{0}(r) = (\pi / 2)^{1/2} \int_{0}^{2r} \Fwd{3}(t) t \, dt
		\end{equation}
		is non-decreasing, but not strictly increasing in general.
		For example, consider the case $w(r) = (1 - \cos(r))/r^2$.
		In this case, we have
		\begin{gather}
			\int_0^\infty w(r) \, dr = \pi / 2 , \\
			\Fwd{3}(r) = \begin{cases}
				(\pi / 2)^{1/2} / r , & 0 < r < 1, \\
				0 , & r > 1 , \\
			\end{cases} \\
			\lambdak{k+1}(r) < \lambdak{0}(r) = \begin{cases}
				\pi r , & 0 < r < 1/2, \\
				\pi/2 , & r \geq 1/2 . 
			\end{cases}
		\end{gather}
		See \citet[12.33.5]{GR2014} for $\Fwd{3}$.
		As a consequence, $u_0 \in L^2(\R^3) \setminus \{0\}$ is an extremiser of the smoothing estimate for the Schr\"{o}dinger equation if and only if it is radial and satisfies $\supp \widehat{u_0} \subset \set{ \xi \in \R^3 }{ \abs{\xi} \geq 1/2 }$. 
		Nevertheless, there are no extremisers for the Dirac equation due to the strict inequality $\Diraclambdak{k}(r) < 2 \lambdak{0}(r)$.
	\end{remark}
	\begin{proof}[Proof of Theorem \ref{thm:type C 2D Dirac}]
		In order to prove Theorem \ref{thm:type C 2D Dirac}, recall that $\lambdak{0} + \lambdak{1}$ is non-decreasing even in the case $d=2$, as we saw in the proof of Theorem \ref{thm:type C Dirac}. Therefore, if $m=0$, then we have
		\begin{equation}
			\sup_{r > 0} \Diraclambdak{0}(r) 
			= \sup_{r > 0} ( \lambdak{0}(r) + \lambdak{1}(r) ) 
			= \lim_{r \to \infty} ( \lambdak{0}(r) + \lambdak{1}(r) )
			= 2 \norm{w}_{L^1(\positiveR)} .
		\end{equation}
		Hence, it suffices to show that
		\begin{equation}
			\lambdakd{k}{2}(r) + \lambdakd{k+1}{2}(r) < 2 \norm{w}_{L^1(\positiveR)} 
		\end{equation}
		holds for any $k \in \N$ and $r > 0$.
		To see this, notice that $\Pkd{k}{2}$ is the Chebyshev polynomial of degree $k$, in other words, it satisfies
		\begin{equation}
			\Pkd{k}{2}(\cos{\theta}) = \cos{k\theta} .
		\end{equation}
		Thus, by changing variable of integration and using trigonometric addition formulae, we get
		\begin{align}
			\lambdakd{k}{2}(r) + \lambdakd{k+1}{2}(r) 
			&= r \int_{-1}^{1} \Fwd{2}( r \sqrt{ 2(1-t) } ) ( \Pkd{k}{2}(t) + \Pkd{k+1}{2}(t) ) (1 - t^2)^{-1/2} \, dt \\
			&= 2 r \int_{0}^{\pi / 2} \Fwd{2}( r \sqrt{ 2(1- \cos{2\theta}) } ) ( \cos{ 2 k \theta } + \cos{ 2 (k+1)\theta } ) \, d\theta \\
			&= 4 r \int_{0}^{\pi / 2} \Fwd{2}( 2 r \sin{\theta} ) \cos{ (2k+1) \theta } \cos{\theta} \, d\theta .
		\end{align}
		Hence, if $\Fwd{2}( 2 r \sin{\theta} ) = 0$ for almost every $\theta \in [0, \pi/2]$, then obviously 
		\begin{equation}
			\lambdakd{k}{2}(r) + \lambdakd{k+1}{2}(r) = 0 < 2 \norm{w}_{L^1(\positiveR)}
		\end{equation}
		holds. If not, then we have
		\begin{align}
			\lambdakd{k}{2}(r) + \lambdakd{k+1}{2}(r) 
			&= 4 r \int_{0}^{\pi / 2} \Fwd{2}( 2 r \sin{\theta} ) \cos{ (2k+1) \theta } \cos{\theta} \, d\theta\\
			&< 4 r \int_{0}^{\pi / 2} \Fwd{2}( 2 r \sin{\theta} ) \cos{\theta} \, d\theta\\
			&= 2 \int_{0}^{2r} \Fwd{2}( t ) \, dt \\
			&\leq 2 \int_{0}^{\infty} \Fwd{2}( t ) \, dt \\
			&= 2 \norm{w}_{L^1(\positiveR)} .
		\end{align}
		This completes the proof.
	\end{proof}
	\section{Remarks on 2-dimensional Schr\"{o}dinger equations} \label{section:remarks on 2D Schrodinger}
	As we saw in Theorem \ref{thm:Schrodinger explicit value}, smoothing estimates for Schr\"{o}dinger equations in the case $d=2$ are somehow different from those in the case $d \geq 3$.
	The following theorems illustrate what is happening in terms of optimal constants.
	\begin{theorem} \label{thm:type A 2D Schrodinger}
		In the case \eqref{eq:type A 2D}, we have
		\begin{equation}
			\max\mleft\{ \frac{2\pi}{s-2} , \frac{ \pi^{1/2} \Gamma( (s-1)/2 ) }{  \Gamma( s/2 ) } \mright\} \leq \Diracconstmd{m}{2} \leq 2 \Sconstd{2} \leq 2 \Diracconstmd{0}{2} \leq 2 \mleft( \frac{\pi}{s-2} + \frac{ \pi^{1/2} \Gamma( (s-1)/2 ) }{  \Gamma( s/2 ) } \mright)
		\end{equation}
		for any $s > 2$ and $m > 0$. 
		In particular, the following hold:
		\begin{align}
			\lim_{s \downarrow 2} (s - 2) \Sconstwpd{ (1+r^2)^{-s/2} }{ (1+r^2)^{1/4} }{2} &= \pi , \\
			\lim_{s \downarrow 2} (s - 2) \Diracconstwpmd{ (1+r^2)^{-s/2} }{ (1+r^2)^{1/4} }{m}{2} &= 
			\begin{cases}
				\pi , & m = 0 , \\
				2\pi , & m > 0 .
			\end{cases}
		\end{align}
	\end{theorem}
	\begin{theorem} \label{thm:type C 2D Schrodinger}
		Let $d = 2$, $w \in L^1(\positiveR) \setminus \{0\}$, and assume that $\Fwd{2}$ is non-negative (the same applies to $w_1, w_2$).
		Then the following hold:
		\begin{eqenumerate}
			\item \label{item:type C 2D inequality Schrodinger}
			We have
			\begin{equation}
				\norm{w}_{L^1(\positiveR)} \leq
				\Sconstwpd{ w }{ r^{1/2} }{2} 
				\leq 2 \norm{w}_{L^1(\positiveR)} .
			\end{equation}
			\item \label{item:type C 2D continuity Schrodinger}
			We have
			\begin{equation}
				\abs{ \Sconstwpd{ w_1 }{ r^{1/2} }{2} - \Sconstwpd{ w_2 }{ r^{1/2} }{2} } 
				\leq 2 \norm{w_1 - w_2}_{L^1(\positiveR)} .
			\end{equation}
			\item \label{item:type C 2D Schrodinger decay s>3}
			If there exist $a, \varepsilon, R > 0$ such that $\Fwd{2}(r) \leq a r^{-(3+\varepsilon)}$ holds for any $r \geq R$, then we have
			\begin{equation}
				\Sconstwpd{ w }{ r^{1/2} }{2} > \norm{w}_{L^1(\positiveR)} .
			\end{equation}
			\item \label{item:type C 2D Schrodinger decay s=3}
			If there exist $a, b, R > 0$ such that $b r^{-3} \leq \Fwd{2}(r) \leq a r^{-3}$ holds for any $r \geq R$, then we have
			\begin{equation}
				\Sconstwpd{ w }{ r^{1/2} }{2} > \norm{w}_{L^1(\positiveR)} .
			\end{equation}
			\item \label{item:type C 2D Schrodinger decay 2<s<3}
			If there exist $2 < s < 3$ and $c > 0$ such that 
			\begin{equation}
				\lim_{r \to \infty} r^s \Fwd{2}(r)  = c , 
			\end{equation}
			then we have
			\begin{equation}
				\Sconstwpd{ w }{ r^{1/2} }{2} > \norm{w}_{L^1(\positiveR)} .
			\end{equation}
			\item \label{item:type C 2D Schrodinger example =}
			In the case $w(r) = \BesselK{0}(r)$, where $\BesselK{0}$ denotes the modified Bessel function of the second kind of order zero, 
			we have
			\begin{equation}
				\Fwd{2}(r) = (1+r^2)^{-1}
			\end{equation}
			and
			\begin{equation}
				\Sconstwpd{ w }{ r^{1/2} }{2} = \norm{w}_{L^1(\positiveR)} = \pi / 2 .
			\end{equation}
			This shows that the assumption $s > 2$ in \eqref{item:type C 2D Schrodinger decay 2<s<3} cannot be relaxed to $s = 2$.
			\item \label{item:type C 2D Schrodinger example >}
			The existence of $c > 0$ such that
			\begin{equation}
				\lim_{r \to \infty} r^2 \Fwd{2}(r)  = c 
			\end{equation}
			is not sufficient for $\Sconstwpd{ w }{ r^{1/2} }{2} = \norm{w}_{L^1(\positiveR)}$.
		\end{eqenumerate}
	\end{theorem}
	\begin{proof}[Proof of Theorem \ref{thm:type A 2D Schrodinger}]
		Recall that we already have $\Diracconstmd{m}{2} \leq 2 \Sconstd{2} \leq 2 \Diracconstmd{0}{2}$ by \eqref{eq:equivalence of Dirac and Schrodinger}.
		Moreover, we also know that 
		\begin{equation}
			\Diracconstmd{m}{2} = \sup_{r > 0} \mleft( \lambdak{0}(r) + \lambdak{1}(r) + \frac{m}{ \sqrt{r^2 + m^2} } (\lambdak{0}(r) - \lambdak{1}(r) ) \mright)
		\end{equation}
		holds by \eqref{item:properties of lambdak 4}, since $\Fwd{d}$ is non-negative for any $d \geq 2$ in the case $w(r) = (1+r^2)^{-s/2}$.
		Therefore, it suffices to show that 
		\begin{align}
			\sup_{r > 0} \mleft( \lambdak{0}(r) + \lambdak{1}(r) + \frac{m}{ \sqrt{r^2 + m^2} } (\lambdak{0}(r) - \lambdak{1}(r) ) \mright) &\geq \frac{2\pi}{s-2} , 
			\label{eq:type A2 estimate 1} \\
			 \sup_{r > 0} \mleft( \lambdak{0}(r) + \lambdak{1}(r) + \frac{m}{ \sqrt{r^2 + m^2} } (\lambdak{0}(r) - \lambdak{1}(r) ) \mright) &\geq \frac{ \pi^{1/2} \Gamma( (s-1)/2 ) }{  \Gamma( s/2 ) } , 
			\label{eq:type A2 estimate 2}\\
			 \sup_{r > 0} \mleft( \lambdak{0}(r) + \lambdak{1}(r)  \mright) &\leq  \frac{\pi}{s-2} + \frac{ \pi^{1/2} \Gamma( (s-1)/2 ) }{  \Gamma( s/2 ) } ,
			\label{eq:type A2 estimate 3}
		\end{align} 
		where $m > 0$.
		Substituting 
		\begin{equation}
			d=2, \quad
			w(r) = (1+r^2)^{-s/2} , \quad
			( \psi(r) )^2 = (1+r^2)^{1/2}
		\end{equation}
		into \eqref{eq:lambdak 1}, we obtain
		\begin{equation}
			\lambdak{k}(r) 
			= \pi ( 1 + r^2 )^{1/2} \int_0^\infty \frac{t}{ ( 1 + t^2 )^{s/2} } ( \BesselJ{k}(rt) )^2 \, dt .
		\end{equation}
		Now notice that the dominated convergence theorem implies
		\begin{equation}
			\lim_{r \downarrow 0} \int_0^\infty \frac{t}{ ( 1 + t^2 )^{s/2} } ( \BesselJ{k}(rt) )^2 \, dt  = \begin{dcases}
				\int_0^\infty \frac{t}{ ( 1 + t^2 )^{s/2} } \, dt = \frac{1}{s - 2} , & k = 0 , \\
				0 , \phantom{\int} & k = 1 ,
			\end{dcases} 
		\end{equation}
		since
		\begin{gather} \label{eq:BesselJ bound 1}
			\sup_{r > 0} ( ( \BesselJ{0}(rt) )^2 + ( \BesselJ{1}(rt) )^2 ) = 1 , \\
			\lim_{r \downarrow 0} ( \BesselJ{0}(rt) )^2 = 1 , \quad \lim_{r \downarrow 0} ( \BesselJ{1}(rt) )^2 = 0 .
		\end{gather}
		As a consequence, we have
		\begin{equation}
			 \lim_{r \downarrow 0} \mleft( \lambdak{0}(r) + \lambdak{1}(r) + \frac{m}{ \sqrt{r^2 + m^2} } (\lambdak{0}(r) - \lambdak{1}(r) ) \mright)
			= \frac{2 \pi}{ s - 2 } 
		\end{equation}
		whenever $m > 0$, which shows \eqref{eq:type A2 estimate 1}.
		On the other hand, we know that
		\begin{equation}
			\lim_{r \to \infty} \pi r \int_0^\infty \frac{t}{ ( 1 + t^2 )^{s/2} } ( ( \BesselJ{0}(rt) )^2 + ( \BesselJ{1}(rt) )^2 ) \, dt 
			= 2 \int_0^\infty \frac{1}{ ( 1 + t^2 )^{s/2} } \, dt 
			= \frac{ \pi^{1/2} \Gamma( (s-1)/2 ) }{  \Gamma( s/2 ) } 
		\end{equation}
		holds, as we saw in the proof of Theorem \ref{thm:type C 2D Dirac}.
		Thus, we have
		\begin{equation}
			 \lim_{r \to \infty} \mleft( \lambdak{0}(r) + \lambdak{1}(r) + \frac{m}{ \sqrt{r^2 + m^2} } (\lambdak{0}(r) - \lambdak{1}(r) ) \mright)
			= \frac{ \pi^{1/2} \Gamma( (s-1)/2 ) }{  \Gamma( s/2 ) } ,
		\end{equation}
		which shows \eqref{eq:type A2 estimate 2}.
		Furthermore, we also know that 
		\begin{equation}
			\sup_{r > 0} \mleft( \pi r \int_0^\infty \frac{t}{ ( 1 + t^2 )^{s/2} } ( ( \BesselJ{0}(rt) )^2 + ( \BesselJ{1}(rt) )^2 ) \, dt \mright)
			= \frac{ \pi^{1/2} \Gamma( (s-1)/2 ) }{  \Gamma( s/2 ) } 
		\end{equation}
		holds by Theorem \ref{thm:type C 2D Dirac}.
		Hence, we get 
		\begin{align}
			 \mleft( \lambdak{0}(r) + \lambdak{1}(r)  \mright)
			\underset{\eqref{eq:BesselJ bound 1}}&{\leq} \pi \int_0^\infty \frac{t}{ ( 1 + t^2 )^{s/2} } \, dt + \pi r \int_0^\infty \frac{t}{ ( 1 + t^2 )^{s/2} } ( ( \BesselJ{0}(rt) )^2 + ( \BesselJ{1}(rt) )^2 ) \, dt \\
			&\leq \frac{\pi}{ s - 2 } + \frac{ \pi^{1/2} \Gamma( (s-1)/2 ) }{  \Gamma( s/2 ) } ,
		\end{align}
		which shows \eqref{eq:type A2 estimate 3}.
	\end{proof}
	
	\begin{proof}[Proof of Theorem \ref{thm:type C 2D Schrodinger}]
		First, notice that \eqref{item:type C 2D inequality Schrodinger} is immediate from \eqref{eq:equivalence of Dirac and Schrodinger} and Theorem \ref{thm:type C 2D Dirac}.
		In order to prove \eqref{item:type C 2D continuity Schrodinger}, we use the following fact (see \citet[Theorem 7.31.2]{Sze1975}):
		\begin{equation} \label{eq:BesselJ bound 2}
			\sup_{t > 0} t ( \BesselJ{0}(t) )^2 = 2 / \pi .
		\end{equation}
		Using this, we obtain
		\begin{align}
			\abs{ \Sconstwpd{ w_1 }{ r^{1/2} }{2} - \Sconstwpd{ w_2 }{ r^{1/2} }{2} } 
			&= \Abs{ \sup_{r > 0} \mleft( \pi r \int_0^\infty t w_1(t) ( \BesselJ{0}(rt) )^2 \, dt \mright) - \sup_{r > 0} \mleft( \pi r \int_0^\infty t w_2(t) ( \BesselJ{0}(rt) )^2 \, dt \mright) } \\
			&\leq \sup_{r > 0 }\Abs{ \pi r \int_0^\infty t ( w_1(t) - w_2(t) ) ( \BesselJ{0}(rt) )^2 \, dt } \\
			\underset{\eqref{eq:BesselJ bound 2}}&\leq 2 \int_0^\infty \abs{ w_1(t) - w_2(t) } \, dt .
		\end{align}
		
		Next, we prove \eqref{item:type C 2D Schrodinger decay s>3}. 
		
		Let $a, \varepsilon, R > 0$ be such that $\Fwd{2}(r) \leq a r^{-(3+\varepsilon)}$ holds for any $r \geq R$.
		Since $w \not \equiv 0$, we also assume that
		\begin{equation}
			\int_{0}^{R} \Fwd{2}(t) t^2 \, dt > 0
		\end{equation}
		without loss of generality.
		Then, for any $r > R$, we have
		\begin{align}
			&\quad \lambdak{0}(r / 2) - \norm{w}_{L^1(\positiveR)} \\
			&= \int_{0}^{r} \Fwd{2}(t) \mleft(1 - \frac{t^2}{ r^2 } \mright)^{-1/2} \, dt
			- \int_{0}^{\infty} \Fwd{2}(t)\, dt \\
			&= \int_{0}^{R} \Fwd{2}(t) \mleft( \mleft(1 - \frac{t^2}{ r^2 } \mright)^{-1/2} - 1 \mright) \, dt
			+ \int_{R}^{r} \Fwd{2}(t) \mleft( \mleft(1 - \frac{t^2}{ r^2 } \mright)^{-1/2} - 1 \mright) \, dt 
			- \int_{r}^{\infty} \Fwd{2}(t) \, dt \\
			&\geq \int_{0}^{R} \Fwd{2}(t) \frac{t^2}{ 2r^2 } \, dt
			- a \int_{r}^{\infty} t^{-(3+\varepsilon)} \, dt \\
			&= \frac{1}{2r^2} \int_{0}^{R} \Fwd{2}(t) t^2 \, dt - \frac{ a }{ 2 + \varepsilon } r^{-(2 + \varepsilon)} \\
			&= \frac{1}{r^{2+\varepsilon}} \mleft( \frac{1}{2} r^{\varepsilon} \int_{0}^{R} \Fwd{2}(t) t^2 \, dt -  \frac{ a }{ 2 + \varepsilon } \mright) .
		\end{align}
		Note that we used the inequality
		\begin{equation}
			(1-t)^{-1/2} - 1 \geq t/2 ,
		\end{equation}
		which holds for any $t < 1$.
		Therefore, for sufficiently large $r >0$, we have $\lambdak{0}(r) - \norm{w}_{L^1(\positiveR)} > 0$.
		Hence, we conclude that
		\begin{equation}
			\sup_{r > 0} \lambdak{0}(r) > \norm{w}_{L^1(\positiveR)}.
		\end{equation}
		
		The proofs of \eqref{item:type C 2D Schrodinger decay s=3} and \eqref{item:type C 2D Schrodinger decay 2<s<3} are similar to that of \eqref{item:type C 2D Schrodinger decay s>3}.
		In the case \eqref{item:type C 2D Schrodinger decay s=3}, let $a, b, R > 0$ be such that $b r^{-3} \leq \Fwd{2}(r) \leq a r^{-3}$ holds for any $r \geq R$.
		Then, for any $r > R$, we have
		\begin{align}
			&\quad \lambdak{0}(r / 2) - \norm{w}_{L^1(\positiveR)} \\
			&= \int_{0}^{r} \Fwd{2}(t) \mleft(1 - \frac{t^2}{ r^2 } \mright)^{-1/2} \, dt
			- \int_{0}^{\infty} \Fwd{2}(t)\, dt \\
			&= \int_{0}^{R} \Fwd{2}(t) \mleft( \mleft(1 - \frac{t^2}{ r^2 } \mright)^{-1/2} - 1 \mright) \, dt
			+ \int_{R}^{r} \Fwd{2}(t) \mleft( \mleft(1 - \frac{t^2}{ r^2 } \mright)^{-1/2} - 1 \mright) \, dt 
			- \int_{r}^{\infty} \Fwd{2}(t) \, dt \\
			&\geq b \int_{R}^{r} t^{-3} \frac{t^2}{ 2r^2 } \, dt
			- a \int_{r}^{\infty} t^{-3} \, dt \\
			&= \frac{b}{2r^2} ( \log{r} - \log{R} ) - \frac{ a }{ 2 } r^{-2} \\
			&= \frac{1}{2r^2} \mleft( b \log{ (r / R) } - a \mright) ,
		\end{align}
		which is strictly positive for sufficiently large $r > 0$, so that \eqref{item:type C 2D Schrodinger decay s=3} holds.
		
		In the case \eqref{item:type C 2D Schrodinger decay 2<s<3}, let $0 < \varepsilon < 1$ be sufficiently small, and $R > 0$ be such that 
		\begin{equation}
			c ( 1 - \varepsilon ) r^{-s} \leq \Fwd{2}(r) \leq c ( 1 + \varepsilon ) r^{-s}
		\end{equation}
		holds for any $r \geq R$.
		Then, for any $r > R / \varepsilon$, we have
		\begin{align}
			&\quad \lambdak{0}(r / 2) - \norm{w}_{L^1(\positiveR)} \\
			&\geq c ( 1 - \varepsilon ) \int_{\varepsilon r}^{r} t^{-s} \mleft( \mleft(1 - \frac{t^2}{ r^2 } \mright)^{-1/2} - 1 \mright) \, dt
			- c ( 1 + \varepsilon ) \int_{r}^{\infty} t^{-s} \, dt \\
			&= \frac{ c }{ r^{s-1} } \mleft( ( 1 - \varepsilon ) \int_{\varepsilon}^{1} t^{-s} \mleft( \mleft(1 - t^2 \mright)^{-1/2} - 1 \mright) \, dt
			- \frac{  1 + \varepsilon  }{ s-1 } \mright).
		\end{align}
		Now notice that $s > 2$ implies
		\begin{align}
			\lim_{\varepsilon \downarrow 0} \mleft( ( 1 - \varepsilon ) \int_{\varepsilon}^{1} t^{-s} \mleft( \mleft(1 - t^2 \mright)^{-1/2} - 1 \mright) \, dt
			- \frac{  1 + \varepsilon  }{ s-1 } \mright)
			&=  \int_{0}^{1} t^{-s} \mleft( \mleft(1 - t^2 \mright)^{-1/2} - 1 \mright) \, dt
			- \frac{ 1 }{ s-1 } \\
			&>  \int_{0}^{1} t^{-2} \mleft( \mleft(1 - t^2 \mright)^{-1/2} - 1 \mright) \, dt
			- 1 \\
			&= 0 ,
		\end{align}
		so that \eqref{item:type C 2D Schrodinger decay 2<s<3} holds.
		
		Finally, we prove \eqref{item:type C 2D Schrodinger example =} and \eqref{item:type C 2D Schrodinger example >}.
		In the case $w(r) = \BesselK{0}(r)$, \cite[6.522.5]{GR2014} implies
		\begin{align}
			\lambdak{0}(r)
			&= \pi r \int_0^\infty t \BesselK{0}( t) ( \BesselJ{0}(r t) )^2 \, dt \\
			&= \frac{\pi r}{ \sqrt{1 + 4 r^2} } , 
		\end{align}
		which is strictly increasing (see Figure \ref{fig:graph of lambda0 type C}.\subref{fig:graph lambda0 for K0} for the graph).
		Hence, we conclude that
		\begin{equation}
			\Sconstwpd{ \BesselK{0} }{ r^{1/2} }{2} = \norm{w}_{L^1(\positiveR)} = \pi / 2 
		\end{equation}
		holds.
		In order to prove \eqref{item:type C 2D Schrodinger example >}, we write
		\begin{gather}
			w_\theta(r) \coloneqq (1 - \theta) e^{-r^2 / 2} + \theta \BesselK{0}(r) , \\
			f(\theta) \coloneqq \Sconstwpd{w_\theta}{r^{1/2}}{2} - \norm{w_\theta}_{L^1(\positiveR)}
		\end{gather}
		for each $\theta \in [0, 1]$.
		Then we know that $f \colon [0,1] \to \R$ is continuous by \eqref{item:type C 2D continuity Schrodinger}, and
		\begin{equation}
			f(0) = \Sconstwpd{w_0}{r^{1/2}}{2} - \norm{w_0}_{L^1(\positiveR)} > 0
		\end{equation}
		holds by \eqref{item:type C 2D Schrodinger decay s>3}. 
		Hence, there exists $\theta_0 \in (0, 1)$ such that $f(\theta_0) > 0$.
		On the other hand, we have
		\begin{align}
			\Fwthetad{\theta}{2}(r) = ( (1 - \theta) e^{-r^2 / 2} + \theta (1+r^2)^{-1} ) ,
		\end{align}
		so that 
		\begin{equation}
			\lim_{r \to \infty} r^2 \Fwthetad{\theta}{2}(r) = \theta
		\end{equation}
		holds for any $\theta \in [0, 1]$.
		This shows \eqref{item:type C 2D Schrodinger example >}.
	\end{proof}
	Finally, we provide some examples for \eqref{item:type C 2D Schrodinger decay s>3} and \eqref{item:type C 2D Schrodinger decay s=3} of Theorem \ref{thm:type C 2D Schrodinger}.
	We consider the cases
	\begin{alignat}{8}
		w(r) ={} & e^{-r^2 / 2} , \quad & & (1+r^2)^{-1} , \quad & & e^{-r} .
		\intertext{In these cases, we have}
		\norm{w}_{L^1(\positiveR)} ={} & \sqrt{\pi / 2} , \quad & & \pi / 2 , \quad & & 1 , \\
		\Fwd{2}(r) ={} & e^{-r^2/2} , \quad & & \BesselK{0}(r) , \quad & &  (1+r^2)^{-3/2} ,
	\end{alignat}
	and 
	\begin{equation}
		\norm{w}_{L^1(\positiveR)} < \Sconstwpd{ w }{ r^{1/2} }{2} \leq 2 \norm{w}_{L^1(\positiveR)} 
	\end{equation}
	holds (note that $\BesselK{0}(r) = O( r^{-1/2} e^{-r} )$ as $r \to \infty$).
	We use Mathematica in order to approximate $\Sconstwpd{ w }{ r^{1/2} }{2} / \norm{w}_{L^1(\positiveR)}$ numerically.
	\begin{example}
		In the case $w(r) = e^{-r^2 / 2}$, \cite[6.633.2]{GR2014} implies
		\begin{align}
			\lambdak{0}(r)
			&=
			\pi r \int_0^\infty t e^{- t^2 / 2} ( \BesselJ{0}(r t) )^2 \, dt \\
			&= \pi r e^{- r^2} \BesselI{0}(r^2) .
		\end{align}
		See Figure \ref{fig:graph of lambda0 type C}.\subref{fig:graph lambda0 for exp(-r^2)} for the graph.
		Using Mathematica, we get the following result:
		\begin{mmaCell}{InputCode}
NMaximize[{(2 Pi)^(1/2) r Exp[-r^2] BesselI[0, r^2], r > 0}, r]
		\end{mmaCell}
		\begin{mmaCell}{OutputCode}
\{1.17516,\{r -> 0.888807\}\} 
		\end{mmaCell}
		This means that 
		\begin{equation}
			\frac{ \Sconstwpd{ w }{ r^{1/2} }{2} }{ \norm{w}_{L^1(\positiveR)} } = \sqrt{\frac{2}{\pi}} \sup_{r > 0} \lambdak{0}(r)  \approx \sqrt{\frac{2}{\pi}} \lambdak{0}(0.888807) \approx 1.17516 .
		\end{equation}
	\end{example}
	\begin{example}
		In the case $w(r) = (1+r^2)^{-1}$, \eqref{item:product of modified Bessel functions integral formula} implies
		\begin{align}
			\lambdak{0}(r)
			&= 
			 \pi r \int_0^\infty \frac{t}{1 + t^2} ( \BesselJ{0}(r t) )^2 \, dt \\
			&= \pi r \BesselI{0}(r) \BesselK{0}(r) ,
		\end{align}
		where $\BesselI{0}$ and $\BesselK{0}$ denote the modified Bessel functions of the first and second kinds of order zero, respectively.
		See Figure \ref{fig:graph of lambda0 type C}.\subref{fig:graph lambda0 for (1+r^2)^(-1)} for the graph.
		Using Mathematica, we get the following result:
		\begin{mmaCell}{InputCode}
NMaximize[{2 r BesselI[0, r] BesselK[0, r], r > 0}, r]
		\end{mmaCell}
		\begin{mmaCell}{OutputCode}
\{1.06673, \{r -> 1.07503\}\}
		\end{mmaCell}
		This means that 
		\begin{equation}
			\frac{ \Sconstwpd{ w }{ r^{1/2} }{2} }{ \norm{w}_{L^1(\positiveR)} } = \frac{2}{\pi} \sup_{r > 0} \lambdak{0}(r) \approx \frac{2}{\pi} \lambdak{0}(1.07503)  \approx 1.06673 .
		\end{equation}
	\end{example}
	\begin{example}
		In the case $w(r) = e^{-r}$, \cite[3.133.4]{GR2014} implies
		\begin{align}
			\lambdak{0}(r) 
			&= r \int_{-1}^{1} (1 + 2r^2 (1-t) )^{-3/2} (1-t^2)^{-1/2} \, dt \\
			&= 2 \frac{ \EllipticE( \sqrt{ 4r^2/(1+4r^2) } ) }{ \sqrt{1/r^2 + 4} } ,
		\end{align}
		where $\EllipticE$ denotes the complete elliptic integral of the second kind:
		\begin{equation}
			\EllipticE(t) \coloneqq \int_0^{\pi/2} \sqrt{ 1 - t^2 \sin^2{\theta} } \, d\theta .
		\end{equation}
		See Figure \ref{fig:graph of lambda0 type C}.\subref{fig:graph lambda0 for e^(-r)} for the graph.
		Using Mathematica, we get the following result:
		\begin{mmaCell}{InputCode}
NMaximize[{2 EllipticE[4 r^2/(1 + 4 r^2)]/(1/r^2 + 4)^(1/2), r > 0}, r]
		\end{mmaCell}
		\begin{mmaCell}{OutputCode}
\{1.05481, \{r -> 1.08983\}\}
		\end{mmaCell}
		This means that 
		\begin{equation}
			\frac{ \Sconstwpd{ w }{ r^{1/2} }{2} }{ \norm{w}_{L^1(\positiveR)} } =  \sup_{r > 0} \lambdak{0}(r)  \approx \lambdak{0}(1.08983) \approx 1.05481  .
		\end{equation}
	\end{example}
\begin{figure}[p]
	\centering
	\caption{The graph of
		$(1+r^2)^{1/2}
		\BesselI{d/2-1}(r)\BesselK{d/2-1}(r)$.}
	\label{fig:graph of lambda0 type A}
	\setlength{\lineskip}{\medskipamount}
	
	\begin{subfigure}[t]{.48\linewidth}
		\centering
		\FitGraphs[\linewidth]{%
			\input{graph_A-3D}%
		}
		\caption{The case $d=3$.}
		\label{fig:graph lambda0 d=3}
	\end{subfigure}
	
	\par
	
	\begin{subfigure}[t]{.96\linewidth}
		\centering
		\FitGraphs[\linewidth]{%
			\input{graph_A-4D-1}
			\input{graph_A-4D-2}%
		}
		\caption{The case $d=4$.}
		\label{fig:graph lambda0 d=4}
	\end{subfigure}
	
	\par
	
	\begin{subfigure}[t]{.48\linewidth}
		\centering
		\FitGraphs[\linewidth]{%
			\input{graph_A-5D}%
		}
		\caption{The case $d=5$.}
		\label{fig:graph lambda0 d=5}
	\end{subfigure}
\end{figure}
\begin{figure}[p]
	\centering
	\caption{The graph of
		$(1+r^2)^{1/2}
		\bigl(
		\BesselI{d/2-1}(r)\BesselK{d/2-1}(r)
		+
		\BesselI{d/2}(r)\BesselK{d/2}(r)
		\bigr)$.}
	\label{fig:graph of Diraclambda0 type A}
	\setlength{\lineskip}{\medskipamount}
	
	\begin{subfigure}[t]{.96\linewidth}
		\centering
		\FitGraphs[\linewidth]{%
			\input{graph_DiracA-3D-1}
			\input{graph_DiracA-3D-2}%
		}
		\caption{The case $d=3$.}
		\label{fig:graph Diraclambda0 d=3}
	\end{subfigure}
	
	\par
	
	\begin{subfigure}[t]{.48\linewidth}
		\centering
		\FitGraphs[\linewidth]{%
			\input{graph_DiracA-4D}%
		}
		\caption{The case $d=4$.}
		\label{fig:graph Diraclambda0 d=4}
	\end{subfigure}
	
	\par
	
	\begin{subfigure}[t]{.48\linewidth}
		\centering
		\FitGraphs[\linewidth]{%
			\input{graph_DiracA-5D}%
		}
		\caption{The case $d=5$.}
		\label{fig:graph Diraclambda0 d=5}
	\end{subfigure}
\end{figure}
\begin{figure}[p]
	\centering
	\caption{The graph of
		$\lambdak{0} / \norm{w}_{L^1(\positiveR)}$.}
	\label{fig:graph of lambda0 type C}
	\setlength{\lineskip}{\medskipamount}
	
	\begin{subfigure}[t]{.48\linewidth}
		\centering
		\FitGraphs[\linewidth]{\input{graph_C-K0}}
		\caption{The case $w(r)=\BesselK{0}(r)$.}
		\label{fig:graph lambda0 for K0}
	\end{subfigure}\hfill
	\begin{subfigure}[t]{.48\linewidth}
		\centering
		\FitGraphs[\linewidth]{\input{graph_C-Gaussian}}
		\caption{The case $w(r)=e^{-r^2/2}$.}
		\label{fig:graph lambda0 for exp(-r^2)}
	\end{subfigure}
	
	\par
	
	\begin{subfigure}[t]{.48\linewidth}
		\centering
		\FitGraphs[\linewidth]{%
			\input{graph_C-rational}}
		\caption{The case $w(r)=(1+r^2)^{-1}$.}
		\label{fig:graph lambda0 for (1+r^2)^(-1)}
	\end{subfigure}\hfill
	\begin{subfigure}[t]{.48\linewidth}
		\centering
		\FitGraphs[\linewidth]{\input{graph_C-exp}}
		\caption{The case $w(r)=e^{-r}$.}
		\label{fig:graph lambda0 for e^(-r)}
	\end{subfigure}
\end{figure}
	\setcitestyle{numbers} 

\input{d-dim_Dirac_arxiv_v3.bbl}
\end{document}

%% file: graph_A-3D.tex
			\begin{tikzpicture}[gnuplot]
	\tikzset{every node/.append style={font={\footnotesize}}}
	\path (0.000,0.000) rectangle (6.604,4.572);
	\gpcolor{color=gp lt color border}
	\gpsetlinetype{gp lt border}
	\gpsetdashtype{gp dt solid}
	\gpsetlinewidth{1.00}
	\draw[gp path] (1.012,0.616)--(1.192,0.616);
	\draw[gp path] (6.051,0.616)--(5.871,0.616);
	\node[gp node right] at (0.828,0.616) {$0$};
	\draw[gp path] (1.012,1.528)--(1.192,1.528);
	\draw[gp path] (6.051,1.528)--(5.871,1.528);
	\node[gp node right] at (0.828,1.528) {$0.5$};
	\draw[gp path] (1.012,2.440)--(1.192,2.440);
	\draw[gp path] (6.051,2.440)--(5.871,2.440);
	\node[gp node right] at (0.828,2.440) {$1$};
	\draw[gp path] (1.012,3.351)--(1.192,3.351);
	\draw[gp path] (6.051,3.351)--(5.871,3.351);
	\node[gp node right] at (0.828,3.351) {$1.5$};
	\draw[gp path] (1.012,4.263)--(1.192,4.263);
	\draw[gp path] (6.051,4.263)--(5.871,4.263);
	\node[gp node right] at (0.828,4.263) {$2$};
	\draw[gp path] (1.012,0.616)--(1.012,0.796);
	\draw[gp path] (1.012,4.263)--(1.012,4.083);
	\node[gp node center] at (1.012,0.308) {$0$};
	\draw[gp path] (1.852,0.616)--(1.852,0.796);
	\draw[gp path] (1.852,4.263)--(1.852,4.083);
	\node[gp node center] at (1.852,0.308) {$1$};
	\draw[gp path] (2.692,0.616)--(2.692,0.796);
	\draw[gp path] (2.692,4.263)--(2.692,4.083);
	\node[gp node center] at (2.692,0.308) {$2$};
	\draw[gp path] (3.532,0.616)--(3.532,0.796);
	\draw[gp path] (3.532,4.263)--(3.532,4.083);
	\node[gp node center] at (3.532,0.308) {$3$};
	\draw[gp path] (4.371,0.616)--(4.371,0.796);
	\draw[gp path] (4.371,4.263)--(4.371,4.083);
	\node[gp node center] at (4.371,0.308) {$4$};
	\draw[gp path] (5.211,0.616)--(5.211,0.796);
	\draw[gp path] (5.211,4.263)--(5.211,4.083);
	\node[gp node center] at (5.211,0.308) {$5$};
	\draw[gp path] (6.051,0.616)--(6.051,0.796);
	\draw[gp path] (6.051,4.263)--(6.051,4.083);
	\node[gp node center] at (6.051,0.308) {$6$};
	\draw[gp path] (1.012,4.263)--(1.012,0.616)--(6.051,0.616)--(6.051,4.263)--cycle;
	\draw[gp path] (1.012,2.440)--(1.063,2.336)--(1.114,2.247)--(1.165,2.170)--(1.216,2.103)--(1.266,2.045)%
	--(1.317,1.995)--(1.368,1.951)--(1.419,1.913)--(1.470,1.880)--(1.521,1.852)--(1.572,1.826)%
	--(1.623,1.804)--(1.674,1.785)--(1.725,1.767)--(1.775,1.751)--(1.826,1.737)--(1.877,1.725)%
	--(1.928,1.713)--(1.979,1.703)--(2.030,1.693)--(2.081,1.685)--(2.132,1.676)--(2.183,1.669)%
	--(2.234,1.662)--(2.284,1.656)--(2.335,1.650)--(2.386,1.644)--(2.437,1.639)--(2.488,1.634)%
	--(2.539,1.629)--(2.590,1.625)--(2.641,1.621)--(2.692,1.617)--(2.743,1.613)--(2.793,1.609)%
	--(2.844,1.606)--(2.895,1.603)--(2.946,1.600)--(2.997,1.597)--(3.048,1.595)--(3.099,1.592)%
	--(3.150,1.590)--(3.201,1.587)--(3.252,1.585)--(3.302,1.583)--(3.353,1.581)--(3.404,1.579)%
	--(3.455,1.577)--(3.506,1.576)--(3.557,1.574)--(3.608,1.572)--(3.659,1.571)--(3.710,1.569)%
	--(3.761,1.568)--(3.811,1.567)--(3.862,1.565)--(3.913,1.564)--(3.964,1.563)--(4.015,1.562)%
	--(4.066,1.561)--(4.117,1.560)--(4.168,1.559)--(4.219,1.558)--(4.270,1.557)--(4.320,1.556)%
	--(4.371,1.555)--(4.422,1.555)--(4.473,1.554)--(4.524,1.553)--(4.575,1.553)--(4.626,1.552)%
	--(4.677,1.551)--(4.728,1.551)--(4.779,1.550)--(4.829,1.549)--(4.880,1.549)--(4.931,1.548)%
	--(4.982,1.548)--(5.033,1.547)--(5.084,1.547)--(5.135,1.546)--(5.186,1.546)--(5.237,1.546)%
	--(5.288,1.545)--(5.338,1.545)--(5.389,1.544)--(5.440,1.544)--(5.491,1.544)--(5.542,1.543)%
	--(5.593,1.543)--(5.644,1.543)--(5.695,1.542)--(5.746,1.542)--(5.797,1.542)--(5.847,1.541)%
	--(5.898,1.541)--(5.949,1.541)--(6.000,1.541)--(6.051,1.540);
	\draw[gp path,dash dot dot] (1.012,1.528)--(1.063,1.528)--(1.114,1.528)--(1.165,1.528)--(1.216,1.528)%
	--(1.266,1.528)--(1.317,1.528)--(1.368,1.528)--(1.419,1.528)--(1.470,1.528)--(1.521,1.528)%
	--(1.572,1.528)--(1.623,1.528)--(1.674,1.528)--(1.725,1.528)--(1.775,1.528)--(1.826,1.528)%
	--(1.877,1.528)--(1.928,1.528)--(1.979,1.528)--(2.030,1.528)--(2.081,1.528)--(2.132,1.528)%
	--(2.183,1.528)--(2.234,1.528)--(2.284,1.528)--(2.335,1.528)--(2.386,1.528)--(2.437,1.528)%
	--(2.488,1.528)--(2.539,1.528)--(2.590,1.528)--(2.641,1.528)--(2.692,1.528)--(2.743,1.528)%
	--(2.793,1.528)--(2.844,1.528)--(2.895,1.528)--(2.946,1.528)--(2.997,1.528)--(3.048,1.528)%
	--(3.099,1.528)--(3.150,1.528)--(3.201,1.528)--(3.252,1.528)--(3.302,1.528)--(3.353,1.528)%
	--(3.404,1.528)--(3.455,1.528)--(3.506,1.528)--(3.557,1.528)--(3.608,1.528)--(3.659,1.528)%
	--(3.710,1.528)--(3.761,1.528)--(3.811,1.528)--(3.862,1.528)--(3.913,1.528)--(3.964,1.528)%
	--(4.015,1.528)--(4.066,1.528)--(4.117,1.528)--(4.168,1.528)--(4.219,1.528)--(4.270,1.528)%
	--(4.320,1.528)--(4.371,1.528)--(4.422,1.528)--(4.473,1.528)--(4.524,1.528)--(4.575,1.528)%
	--(4.626,1.528)--(4.677,1.528)--(4.728,1.528)--(4.779,1.528)--(4.829,1.528)--(4.880,1.528)%
	--(4.931,1.528)--(4.982,1.528)--(5.033,1.528)--(5.084,1.528)--(5.135,1.528)--(5.186,1.528)%
	--(5.237,1.528)--(5.288,1.528)--(5.338,1.528)--(5.389,1.528)--(5.440,1.528)--(5.491,1.528)%
	--(5.542,1.528)--(5.593,1.528)--(5.644,1.528)--(5.695,1.528)--(5.746,1.528)--(5.797,1.528)%
	--(5.847,1.528)--(5.898,1.528)--(5.949,1.528)--(6.000,1.528)--(6.051,1.528);
	\gpcolor{color=gp lt color border}
	\draw[gp path] (1.012,4.263)--(1.012,0.616)--(6.051,0.616)--(6.051,4.263)--cycle;
	\gpdefrectangularnode{gp plot 1}{\pgfpoint{1.012cm}{0.616cm}}{\pgfpoint{6.051cm}{4.263cm}}
\end{tikzpicture}

%% file: graph_A-4D-1.tex
			\begin{tikzpicture}[gnuplot]
	\tikzset{every node/.append style={font={\footnotesize}}}
	\path (0.000,0.000) rectangle (6.604,4.572);
	\gpcolor{color=gp lt color border}
	\gpsetlinetype{gp lt border}
	\gpsetdashtype{gp dt solid}
	\gpsetlinewidth{1.00}
	\draw[gp path] (1.012,0.616)--(1.192,0.616);
	\draw[gp path] (6.051,0.616)--(5.871,0.616);
	\node[gp node right] at (0.828,0.616) {$0$};
	\draw[gp path] (1.012,1.528)--(1.192,1.528);
	\draw[gp path] (6.051,1.528)--(5.871,1.528);
	\node[gp node right] at (0.828,1.528) {$0.5$};
	\draw[gp path] (1.012,2.440)--(1.192,2.440);
	\draw[gp path] (6.051,2.440)--(5.871,2.440);
	\node[gp node right] at (0.828,2.440) {$1$};
	\draw[gp path] (1.012,3.351)--(1.192,3.351);
	\draw[gp path] (6.051,3.351)--(5.871,3.351);
	\node[gp node right] at (0.828,3.351) {$1.5$};
	\draw[gp path] (1.012,4.263)--(1.192,4.263);
	\draw[gp path] (6.051,4.263)--(5.871,4.263);
	\node[gp node right] at (0.828,4.263) {$2$};
	\draw[gp path] (1.012,0.616)--(1.012,0.796);
	\draw[gp path] (1.012,4.263)--(1.012,4.083);
	\node[gp node center] at (1.012,0.308) {$0$};
	\draw[gp path] (1.852,0.616)--(1.852,0.796);
	\draw[gp path] (1.852,4.263)--(1.852,4.083);
	\node[gp node center] at (1.852,0.308) {$1$};
	\draw[gp path] (2.692,0.616)--(2.692,0.796);
	\draw[gp path] (2.692,4.263)--(2.692,4.083);
	\node[gp node center] at (2.692,0.308) {$2$};
	\draw[gp path] (3.532,0.616)--(3.532,0.796);
	\draw[gp path] (3.532,4.263)--(3.532,4.083);
	\node[gp node center] at (3.532,0.308) {$3$};
	\draw[gp path] (4.371,0.616)--(4.371,0.796);
	\draw[gp path] (4.371,4.263)--(4.371,4.083);
	\node[gp node center] at (4.371,0.308) {$4$};
	\draw[gp path] (5.211,0.616)--(5.211,0.796);
	\draw[gp path] (5.211,4.263)--(5.211,4.083);
	\node[gp node center] at (5.211,0.308) {$5$};
	\draw[gp path] (6.051,0.616)--(6.051,0.796);
	\draw[gp path] (6.051,4.263)--(6.051,4.083);
	\node[gp node center] at (6.051,0.308) {$6$};
	\draw[gp path] (1.012,4.263)--(1.012,0.616)--(6.051,0.616)--(6.051,4.263)--cycle;
	\draw[gp path] (1.012,1.528)--(1.063,1.524)--(1.114,1.518)--(1.165,1.511)--(1.216,1.504)--(1.266,1.498)%
	--(1.317,1.494)--(1.368,1.490)--(1.419,1.487)--(1.470,1.486)--(1.521,1.485)--(1.572,1.485)%
	--(1.623,1.486)--(1.674,1.487)--(1.725,1.488)--(1.775,1.490)--(1.826,1.492)--(1.877,1.494)%
	--(1.928,1.497)--(1.979,1.499)--(2.030,1.501)--(2.081,1.503)--(2.132,1.506)--(2.183,1.508)%
	--(2.234,1.510)--(2.284,1.512)--(2.335,1.513)--(2.386,1.515)--(2.437,1.517)--(2.488,1.518)%
	--(2.539,1.520)--(2.590,1.521)--(2.641,1.522)--(2.692,1.523)--(2.743,1.524)--(2.793,1.525)%
	--(2.844,1.526)--(2.895,1.527)--(2.946,1.527)--(2.997,1.528)--(3.048,1.528)--(3.099,1.529)%
	--(3.150,1.529)--(3.201,1.530)--(3.252,1.530)--(3.302,1.530)--(3.353,1.531)--(3.404,1.531)%
	--(3.455,1.531)--(3.506,1.531)--(3.557,1.532)--(3.608,1.532)--(3.659,1.532)--(3.710,1.532)%
	--(3.761,1.532)--(3.811,1.532)--(3.862,1.532)--(3.913,1.532)--(3.964,1.532)--(4.015,1.532)%
	--(4.066,1.532)--(4.117,1.532)--(4.168,1.532)--(4.219,1.532)--(4.270,1.532)--(4.320,1.532)%
	--(4.371,1.532)--(4.422,1.532)--(4.473,1.532)--(4.524,1.532)--(4.575,1.532)--(4.626,1.532)%
	--(4.677,1.532)--(4.728,1.532)--(4.779,1.532)--(4.829,1.532)--(4.880,1.532)--(4.931,1.531)%
	--(4.982,1.531)--(5.033,1.531)--(5.084,1.531)--(5.135,1.531)--(5.186,1.531)--(5.237,1.531)%
	--(5.288,1.531)--(5.338,1.531)--(5.389,1.531)--(5.440,1.531)--(5.491,1.531)--(5.542,1.531)%
	--(5.593,1.531)--(5.644,1.531)--(5.695,1.531)--(5.746,1.531)--(5.797,1.531)--(5.847,1.531)%
	--(5.898,1.531)--(5.949,1.530)--(6.000,1.530)--(6.051,1.530);
	\gpcolor{color=gp lt color border}
	\draw[gp path] (1.012,4.263)--(1.012,0.616)--(6.051,0.616)--(6.051,4.263)--cycle;
	\gpdefrectangularnode{gp plot 1}{\pgfpoint{1.012cm}{0.616cm}}{\pgfpoint{6.051cm}{4.263cm}}
\end{tikzpicture}

%% file: graph_A-4D-2.tex
			\begin{tikzpicture}[gnuplot]
	\tikzset{every node/.append style={font={\footnotesize}}}
	\path (0.000,0.000) rectangle (6.604,4.572);
	\gpcolor{color=gp lt color border}
	\gpsetlinetype{gp lt border}
	\gpsetdashtype{gp dt solid}
	\gpsetlinewidth{1.00}
	\draw[gp path] (1.196,0.616)--(1.376,0.616);
	\draw[gp path] (6.051,0.616)--(5.871,0.616);
	\node[gp node right] at (1.012,0.616) {$0.47$};
	\draw[gp path] (1.196,1.224)--(1.376,1.224);
	\draw[gp path] (6.051,1.224)--(5.871,1.224);
	\node[gp node right] at (1.012,1.224) {$0.48$};
	\draw[gp path] (1.196,1.832)--(1.376,1.832);
	\draw[gp path] (6.051,1.832)--(5.871,1.832);
	\node[gp node right] at (1.012,1.832) {$0.49$};
	\draw[gp path] (1.196,2.440)--(1.376,2.440);
	\draw[gp path] (6.051,2.440)--(5.871,2.440);
	\node[gp node right] at (1.012,2.440) {$0.5$};
	\draw[gp path] (1.196,3.047)--(1.376,3.047);
	\draw[gp path] (6.051,3.047)--(5.871,3.047);
	\node[gp node right] at (1.012,3.047) {$0.51$};
	\draw[gp path] (1.196,3.655)--(1.376,3.655);
	\draw[gp path] (6.051,3.655)--(5.871,3.655);
	\node[gp node right] at (1.012,3.655) {$0.52$};
	\draw[gp path] (1.196,4.263)--(1.376,4.263);
	\draw[gp path] (6.051,4.263)--(5.871,4.263);
	\node[gp node right] at (1.012,4.263) {$0.53$};
	\draw[gp path] (1.196,0.616)--(1.196,0.796);
	\draw[gp path] (1.196,4.263)--(1.196,4.083);
	\node[gp node center] at (1.196,0.308) {$0$};
	\draw[gp path] (2.005,0.616)--(2.005,0.796);
	\draw[gp path] (2.005,4.263)--(2.005,4.083);
	\node[gp node center] at (2.005,0.308) {$1$};
	\draw[gp path] (2.814,0.616)--(2.814,0.796);
	\draw[gp path] (2.814,4.263)--(2.814,4.083);
	\node[gp node center] at (2.814,0.308) {$2$};
	\draw[gp path] (3.624,0.616)--(3.624,0.796);
	\draw[gp path] (3.624,4.263)--(3.624,4.083);
	\node[gp node center] at (3.624,0.308) {$3$};
	\draw[gp path] (4.433,0.616)--(4.433,0.796);
	\draw[gp path] (4.433,4.263)--(4.433,4.083);
	\node[gp node center] at (4.433,0.308) {$4$};
	\draw[gp path] (5.242,0.616)--(5.242,0.796);
	\draw[gp path] (5.242,4.263)--(5.242,4.083);
	\node[gp node center] at (5.242,0.308) {$5$};
	\draw[gp path] (6.051,0.616)--(6.051,0.796);
	\draw[gp path] (6.051,4.263)--(6.051,4.083);
	\node[gp node center] at (6.051,0.308) {$6$};
	\draw[gp path] (1.196,4.263)--(1.196,0.616)--(6.051,0.616)--(6.051,4.263)--cycle;
	\draw[gp path] (1.196,2.440)--(1.245,2.318)--(1.294,2.103)--(1.343,1.869)--(1.392,1.649)--(1.441,1.457)%
	--(1.490,1.300)--(1.539,1.180)--(1.588,1.095)--(1.637,1.042)--(1.686,1.018)--(1.735,1.018)%
	--(1.784,1.039)--(1.834,1.076)--(1.883,1.126)--(1.932,1.186)--(1.981,1.254)--(2.030,1.326)%
	--(2.079,1.401)--(2.128,1.477)--(2.177,1.553)--(2.226,1.628)--(2.275,1.700)--(2.324,1.771)%
	--(2.373,1.838)--(2.422,1.902)--(2.471,1.962)--(2.520,2.019)--(2.569,2.072)--(2.618,2.122)%
	--(2.667,2.168)--(2.716,2.211)--(2.765,2.250)--(2.814,2.286)--(2.863,2.319)--(2.912,2.349)%
	--(2.961,2.376)--(3.010,2.401)--(3.060,2.424)--(3.109,2.444)--(3.158,2.463)--(3.207,2.479)%
	--(3.256,2.494)--(3.305,2.507)--(3.354,2.519)--(3.403,2.529)--(3.452,2.539)--(3.501,2.547)%
	--(3.550,2.554)--(3.599,2.560)--(3.648,2.565)--(3.697,2.569)--(3.746,2.573)--(3.795,2.576)%
	--(3.844,2.579)--(3.893,2.581)--(3.942,2.582)--(3.991,2.583)--(4.040,2.584)--(4.089,2.585)%
	--(4.138,2.585)--(4.187,2.585)--(4.237,2.584)--(4.286,2.584)--(4.335,2.583)--(4.384,2.582)%
	--(4.433,2.581)--(4.482,2.580)--(4.531,2.578)--(4.580,2.577)--(4.629,2.576)--(4.678,2.574)%
	--(4.727,2.572)--(4.776,2.571)--(4.825,2.569)--(4.874,2.567)--(4.923,2.566)--(4.972,2.564)%
	--(5.021,2.562)--(5.070,2.560)--(5.119,2.559)--(5.168,2.557)--(5.217,2.555)--(5.266,2.553)%
	--(5.315,2.552)--(5.364,2.550)--(5.413,2.548)--(5.463,2.546)--(5.512,2.545)--(5.561,2.543)%
	--(5.610,2.542)--(5.659,2.540)--(5.708,2.538)--(5.757,2.537)--(5.806,2.535)--(5.855,2.534)%
	--(5.904,2.532)--(5.953,2.531)--(6.002,2.529)--(6.051,2.528);
	\draw[gp path,dash dot dot] (1.196,2.440)--(1.245,2.440)--(1.294,2.440)--(1.343,2.440)--(1.392,2.440)%
	--(1.441,2.440)--(1.490,2.440)--(1.539,2.440)--(1.588,2.440)--(1.637,2.440)--(1.686,2.440)%
	--(1.735,2.440)--(1.784,2.440)--(1.834,2.440)--(1.883,2.440)--(1.932,2.440)--(1.981,2.440)%
	--(2.030,2.440)--(2.079,2.440)--(2.128,2.440)--(2.177,2.440)--(2.226,2.440)--(2.275,2.440)%
	--(2.324,2.440)--(2.373,2.440)--(2.422,2.440)--(2.471,2.440)--(2.520,2.440)--(2.569,2.440)%
	--(2.618,2.440)--(2.667,2.440)--(2.716,2.440)--(2.765,2.440)--(2.814,2.440)--(2.863,2.440)%
	--(2.912,2.440)--(2.961,2.440)--(3.010,2.440)--(3.060,2.440)--(3.109,2.440)--(3.158,2.440)%
	--(3.207,2.440)--(3.256,2.440)--(3.305,2.440)--(3.354,2.440)--(3.403,2.440)--(3.452,2.440)%
	--(3.501,2.440)--(3.550,2.440)--(3.599,2.440)--(3.648,2.440)--(3.697,2.440)--(3.746,2.440)%
	--(3.795,2.440)--(3.844,2.440)--(3.893,2.440)--(3.942,2.440)--(3.991,2.440)--(4.040,2.440)%
	--(4.089,2.440)--(4.138,2.440)--(4.187,2.440)--(4.237,2.440)--(4.286,2.440)--(4.335,2.440)%
	--(4.384,2.440)--(4.433,2.440)--(4.482,2.440)--(4.531,2.440)--(4.580,2.440)--(4.629,2.440)%
	--(4.678,2.440)--(4.727,2.440)--(4.776,2.440)--(4.825,2.440)--(4.874,2.440)--(4.923,2.440)%
	--(4.972,2.440)--(5.021,2.440)--(5.070,2.440)--(5.119,2.440)--(5.168,2.440)--(5.217,2.440)%
	--(5.266,2.440)--(5.315,2.440)--(5.364,2.440)--(5.413,2.440)--(5.463,2.440)--(5.512,2.440)%
	--(5.561,2.440)--(5.610,2.440)--(5.659,2.440)--(5.708,2.440)--(5.757,2.440)--(5.806,2.440)%
	--(5.855,2.440)--(5.904,2.440)--(5.953,2.440)--(6.002,2.440)--(6.051,2.440);
	\draw[gp path,dash dot dot] (4.138,0.616)--(4.138,2.585);
	\gpcolor{color=gp lt color border}
	\draw[gp path] (1.196,4.263)--(1.196,0.616)--(6.051,0.616)--(6.051,4.263)--cycle;
	\gpdefrectangularnode{gp plot 1}{\pgfpoint{1.196cm}{0.616cm}}{\pgfpoint{6.051cm}{4.263cm}}
\end{tikzpicture}

%% file: graph_A-5D.tex
		\begin{tikzpicture}[gnuplot]
	\tikzset{every node/.append style={font={\footnotesize}}}
	\path (0.000,0.000) rectangle (6.604,4.572);
	\gpcolor{color=gp lt color border}
	\gpsetlinetype{gp lt border}
	\gpsetdashtype{gp dt solid}
	\gpsetlinewidth{1.00}
	\draw[gp path] (1.012,0.616)--(1.192,0.616);
	\draw[gp path] (6.051,0.616)--(5.871,0.616);
	\node[gp node right] at (0.828,0.616) {$0$};
	\draw[gp path] (1.012,1.528)--(1.192,1.528);
	\draw[gp path] (6.051,1.528)--(5.871,1.528);
	\node[gp node right] at (0.828,1.528) {$0.5$};
	\draw[gp path] (1.012,2.440)--(1.192,2.440);
	\draw[gp path] (6.051,2.440)--(5.871,2.440);
	\node[gp node right] at (0.828,2.440) {$1$};
	\draw[gp path] (1.012,3.351)--(1.192,3.351);
	\draw[gp path] (6.051,3.351)--(5.871,3.351);
	\node[gp node right] at (0.828,3.351) {$1.5$};
	\draw[gp path] (1.012,4.263)--(1.192,4.263);
	\draw[gp path] (6.051,4.263)--(5.871,4.263);
	\node[gp node right] at (0.828,4.263) {$2$};
	\draw[gp path] (1.012,0.616)--(1.012,0.796);
	\draw[gp path] (1.012,4.263)--(1.012,4.083);
	\node[gp node center] at (1.012,0.308) {$0$};
	\draw[gp path] (1.852,0.616)--(1.852,0.796);
	\draw[gp path] (1.852,4.263)--(1.852,4.083);
	\node[gp node center] at (1.852,0.308) {$1$};
	\draw[gp path] (2.692,0.616)--(2.692,0.796);
	\draw[gp path] (2.692,4.263)--(2.692,4.083);
	\node[gp node center] at (2.692,0.308) {$2$};
	\draw[gp path] (3.532,0.616)--(3.532,0.796);
	\draw[gp path] (3.532,4.263)--(3.532,4.083);
	\node[gp node center] at (3.532,0.308) {$3$};
	\draw[gp path] (4.371,0.616)--(4.371,0.796);
	\draw[gp path] (4.371,4.263)--(4.371,4.083);
	\node[gp node center] at (4.371,0.308) {$4$};
	\draw[gp path] (5.211,0.616)--(5.211,0.796);
	\draw[gp path] (5.211,4.263)--(5.211,4.083);
	\node[gp node center] at (5.211,0.308) {$5$};
	\draw[gp path] (6.051,0.616)--(6.051,0.796);
	\draw[gp path] (6.051,4.263)--(6.051,4.083);
	\node[gp node center] at (6.051,0.308) {$6$};
	\draw[gp path] (1.012,4.263)--(1.012,0.616)--(6.051,0.616)--(6.051,4.263)--cycle;
	\draw[gp path] (1.012,1.224)--(1.063,1.224)--(1.114,1.225)--(1.165,1.227)--(1.216,1.229)--(1.266,1.233)%
	--(1.317,1.237)--(1.368,1.242)--(1.419,1.248)--(1.470,1.255)--(1.521,1.262)--(1.572,1.269)%
	--(1.623,1.277)--(1.674,1.285)--(1.725,1.293)--(1.775,1.302)--(1.826,1.310)--(1.877,1.318)%
	--(1.928,1.326)--(1.979,1.334)--(2.030,1.342)--(2.081,1.350)--(2.132,1.357)--(2.183,1.364)%
	--(2.234,1.371)--(2.284,1.378)--(2.335,1.384)--(2.386,1.391)--(2.437,1.397)--(2.488,1.402)%
	--(2.539,1.408)--(2.590,1.413)--(2.641,1.418)--(2.692,1.423)--(2.743,1.427)--(2.793,1.431)%
	--(2.844,1.435)--(2.895,1.439)--(2.946,1.443)--(2.997,1.447)--(3.048,1.450)--(3.099,1.453)%
	--(3.150,1.456)--(3.201,1.459)--(3.252,1.462)--(3.302,1.464)--(3.353,1.467)--(3.404,1.469)%
	--(3.455,1.471)--(3.506,1.473)--(3.557,1.476)--(3.608,1.477)--(3.659,1.479)--(3.710,1.481)%
	--(3.761,1.483)--(3.811,1.484)--(3.862,1.486)--(3.913,1.487)--(3.964,1.489)--(4.015,1.490)%
	--(4.066,1.491)--(4.117,1.492)--(4.168,1.493)--(4.219,1.495)--(4.270,1.496)--(4.320,1.497)%
	--(4.371,1.498)--(4.422,1.498)--(4.473,1.499)--(4.524,1.500)--(4.575,1.501)--(4.626,1.502)%
	--(4.677,1.502)--(4.728,1.503)--(4.779,1.504)--(4.829,1.505)--(4.880,1.505)--(4.931,1.506)%
	--(4.982,1.506)--(5.033,1.507)--(5.084,1.507)--(5.135,1.508)--(5.186,1.508)--(5.237,1.509)%
	--(5.288,1.509)--(5.338,1.510)--(5.389,1.510)--(5.440,1.511)--(5.491,1.511)--(5.542,1.511)%
	--(5.593,1.512)--(5.644,1.512)--(5.695,1.513)--(5.746,1.513)--(5.797,1.513)--(5.847,1.513)%
	--(5.898,1.514)--(5.949,1.514)--(6.000,1.514)--(6.051,1.515);
	\draw[gp path,dash dot dot] (1.012,1.528)--(1.063,1.528)--(1.114,1.528)--(1.165,1.528)--(1.216,1.528)%
	--(1.266,1.528)--(1.317,1.528)--(1.368,1.528)--(1.419,1.528)--(1.470,1.528)--(1.521,1.528)%
	--(1.572,1.528)--(1.623,1.528)--(1.674,1.528)--(1.725,1.528)--(1.775,1.528)--(1.826,1.528)%
	--(1.877,1.528)--(1.928,1.528)--(1.979,1.528)--(2.030,1.528)--(2.081,1.528)--(2.132,1.528)%
	--(2.183,1.528)--(2.234,1.528)--(2.284,1.528)--(2.335,1.528)--(2.386,1.528)--(2.437,1.528)%
	--(2.488,1.528)--(2.539,1.528)--(2.590,1.528)--(2.641,1.528)--(2.692,1.528)--(2.743,1.528)%
	--(2.793,1.528)--(2.844,1.528)--(2.895,1.528)--(2.946,1.528)--(2.997,1.528)--(3.048,1.528)%
	--(3.099,1.528)--(3.150,1.528)--(3.201,1.528)--(3.252,1.528)--(3.302,1.528)--(3.353,1.528)%
	--(3.404,1.528)--(3.455,1.528)--(3.506,1.528)--(3.557,1.528)--(3.608,1.528)--(3.659,1.528)%
	--(3.710,1.528)--(3.761,1.528)--(3.811,1.528)--(3.862,1.528)--(3.913,1.528)--(3.964,1.528)%
	--(4.015,1.528)--(4.066,1.528)--(4.117,1.528)--(4.168,1.528)--(4.219,1.528)--(4.270,1.528)%
	--(4.320,1.528)--(4.371,1.528)--(4.422,1.528)--(4.473,1.528)--(4.524,1.528)--(4.575,1.528)%
	--(4.626,1.528)--(4.677,1.528)--(4.728,1.528)--(4.779,1.528)--(4.829,1.528)--(4.880,1.528)%
	--(4.931,1.528)--(4.982,1.528)--(5.033,1.528)--(5.084,1.528)--(5.135,1.528)--(5.186,1.528)%
	--(5.237,1.528)--(5.288,1.528)--(5.338,1.528)--(5.389,1.528)--(5.440,1.528)--(5.491,1.528)%
	--(5.542,1.528)--(5.593,1.528)--(5.644,1.528)--(5.695,1.528)--(5.746,1.528)--(5.797,1.528)%
	--(5.847,1.528)--(5.898,1.528)--(5.949,1.528)--(6.000,1.528)--(6.051,1.528);
	\gpcolor{color=gp lt color border}
	\draw[gp path] (1.012,4.263)--(1.012,0.616)--(6.051,0.616)--(6.051,4.263)--cycle;
	\gpdefrectangularnode{gp plot 1}{\pgfpoint{1.012cm}{0.616cm}}{\pgfpoint{6.051cm}{4.263cm}}
\end{tikzpicture}

%% file: graph_DiracA-3D-1.tex
			\begin{tikzpicture}[gnuplot]
	\tikzset{every node/.append style={font={\footnotesize}}}
	\path (0.000,0.000) rectangle (6.604,4.572);
	\gpcolor{color=gp lt color border}
	\gpsetlinetype{gp lt border}
	\gpsetdashtype{gp dt solid}
	\gpsetlinewidth{1.00}
	\draw[gp path] (1.012,0.616)--(1.192,0.616);
	\draw[gp path] (6.051,0.616)--(5.871,0.616);
	\node[gp node right] at (0.828,0.616) {$0$};
	\draw[gp path] (1.012,1.528)--(1.192,1.528);
	\draw[gp path] (6.051,1.528)--(5.871,1.528);
	\node[gp node right] at (0.828,1.528) {$0.5$};
	\draw[gp path] (1.012,2.440)--(1.192,2.440);
	\draw[gp path] (6.051,2.440)--(5.871,2.440);
	\node[gp node right] at (0.828,2.440) {$1$};
	\draw[gp path] (1.012,3.351)--(1.192,3.351);
	\draw[gp path] (6.051,3.351)--(5.871,3.351);
	\node[gp node right] at (0.828,3.351) {$1.5$};
	\draw[gp path] (1.012,4.263)--(1.192,4.263);
	\draw[gp path] (6.051,4.263)--(5.871,4.263);
	\node[gp node right] at (0.828,4.263) {$2$};
	\draw[gp path] (1.012,0.616)--(1.012,0.796);
	\draw[gp path] (1.012,4.263)--(1.012,4.083);
	\node[gp node center] at (1.012,0.308) {$0$};
	\draw[gp path] (1.852,0.616)--(1.852,0.796);
	\draw[gp path] (1.852,4.263)--(1.852,4.083);
	\node[gp node center] at (1.852,0.308) {$1$};
	\draw[gp path] (2.692,0.616)--(2.692,0.796);
	\draw[gp path] (2.692,4.263)--(2.692,4.083);
	\node[gp node center] at (2.692,0.308) {$2$};
	\draw[gp path] (3.532,0.616)--(3.532,0.796);
	\draw[gp path] (3.532,4.263)--(3.532,4.083);
	\node[gp node center] at (3.532,0.308) {$3$};
	\draw[gp path] (4.371,0.616)--(4.371,0.796);
	\draw[gp path] (4.371,4.263)--(4.371,4.083);
	\node[gp node center] at (4.371,0.308) {$4$};
	\draw[gp path] (5.211,0.616)--(5.211,0.796);
	\draw[gp path] (5.211,4.263)--(5.211,4.083);
	\node[gp node center] at (5.211,0.308) {$5$};
	\draw[gp path] (6.051,0.616)--(6.051,0.796);
	\draw[gp path] (6.051,4.263)--(6.051,4.083);
	\node[gp node center] at (6.051,0.308) {$6$};
	\draw[gp path] (1.012,4.263)--(1.012,0.616)--(6.051,0.616)--(6.051,4.263)--cycle;
	\draw[gp path] (1.012,3.048)--(1.063,2.945)--(1.114,2.856)--(1.165,2.781)--(1.216,2.716)--(1.266,2.662)%
	--(1.317,2.616)--(1.368,2.578)--(1.419,2.546)--(1.470,2.519)--(1.521,2.498)--(1.572,2.480)%
	--(1.623,2.465)--(1.674,2.454)--(1.725,2.444)--(1.775,2.437)--(1.826,2.431)--(1.877,2.427)%
	--(1.928,2.424)--(1.979,2.421)--(2.030,2.419)--(2.081,2.418)--(2.132,2.418)--(2.183,2.417)%
	--(2.234,2.417)--(2.284,2.418)--(2.335,2.418)--(2.386,2.419)--(2.437,2.419)--(2.488,2.420)%
	--(2.539,2.421)--(2.590,2.422)--(2.641,2.422)--(2.692,2.423)--(2.743,2.424)--(2.793,2.425)%
	--(2.844,2.426)--(2.895,2.426)--(2.946,2.427)--(2.997,2.428)--(3.048,2.428)--(3.099,2.429)%
	--(3.150,2.430)--(3.201,2.430)--(3.252,2.431)--(3.302,2.431)--(3.353,2.432)--(3.404,2.432)%
	--(3.455,2.433)--(3.506,2.433)--(3.557,2.433)--(3.608,2.434)--(3.659,2.434)--(3.710,2.434)%
	--(3.761,2.435)--(3.811,2.435)--(3.862,2.435)--(3.913,2.435)--(3.964,2.436)--(4.015,2.436)%
	--(4.066,2.436)--(4.117,2.436)--(4.168,2.436)--(4.219,2.437)--(4.270,2.437)--(4.320,2.437)%
	--(4.371,2.437)--(4.422,2.437)--(4.473,2.437)--(4.524,2.437)--(4.575,2.438)--(4.626,2.438)%
	--(4.677,2.438)--(4.728,2.438)--(4.779,2.438)--(4.829,2.438)--(4.880,2.438)--(4.931,2.438)%
	--(4.982,2.438)--(5.033,2.438)--(5.084,2.438)--(5.135,2.438)--(5.186,2.438)--(5.237,2.438)%
	--(5.288,2.439)--(5.338,2.439)--(5.389,2.439)--(5.440,2.439)--(5.491,2.439)--(5.542,2.439)%
	--(5.593,2.439)--(5.644,2.439)--(5.695,2.439)--(5.746,2.439)--(5.797,2.439)--(5.847,2.439)%
	--(5.898,2.439)--(5.949,2.439)--(6.000,2.439)--(6.051,2.439);
	\gpcolor{color=gp lt color border}
	\draw[gp path] (1.012,4.263)--(1.012,0.616)--(6.051,0.616)--(6.051,4.263)--cycle;
	\gpdefrectangularnode{gp plot 1}{\pgfpoint{1.012cm}{0.616cm}}{\pgfpoint{6.051cm}{4.263cm}}
\end{tikzpicture}

%% file: graph_DiracA-3D-2.tex
			\begin{tikzpicture}[gnuplot]
	\tikzset{every node/.append style={font={\footnotesize}}}
	\path (0.000,0.000) rectangle (6.604,4.572);
	\gpcolor{color=gp lt color border}
	\gpsetlinetype{gp lt border}
	\gpsetdashtype{gp dt solid}
	\gpsetlinewidth{1.00}
	\draw[gp path] (1.196,0.616)--(1.376,0.616);
	\draw[gp path] (6.051,0.616)--(5.871,0.616);
	\node[gp node right] at (1.012,0.616) {$0.97$};
	\draw[gp path] (1.196,1.224)--(1.376,1.224);
	\draw[gp path] (6.051,1.224)--(5.871,1.224);
	\node[gp node right] at (1.012,1.224) {$0.98$};
	\draw[gp path] (1.196,1.832)--(1.376,1.832);
	\draw[gp path] (6.051,1.832)--(5.871,1.832);
	\node[gp node right] at (1.012,1.832) {$0.99$};
	\draw[gp path] (1.196,2.440)--(1.376,2.440);
	\draw[gp path] (6.051,2.440)--(5.871,2.440);
	\node[gp node right] at (1.012,2.440) {$1$};
	\draw[gp path] (1.196,3.047)--(1.376,3.047);
	\draw[gp path] (6.051,3.047)--(5.871,3.047);
	\node[gp node right] at (1.012,3.047) {$1.01$};
	\draw[gp path] (1.196,3.655)--(1.376,3.655);
	\draw[gp path] (6.051,3.655)--(5.871,3.655);
	\node[gp node right] at (1.012,3.655) {$1.02$};
	\draw[gp path] (1.196,4.263)--(1.376,4.263);
	\draw[gp path] (6.051,4.263)--(5.871,4.263);
	\node[gp node right] at (1.012,4.263) {$1.03$};
	\draw[gp path] (1.196,0.616)--(1.196,0.796);
	\draw[gp path] (1.196,4.263)--(1.196,4.083);
	\node[gp node center] at (1.196,0.308) {$0$};
	\draw[gp path] (2.005,0.616)--(2.005,0.796);
	\draw[gp path] (2.005,4.263)--(2.005,4.083);
	\node[gp node center] at (2.005,0.308) {$1$};
	\draw[gp path] (2.814,0.616)--(2.814,0.796);
	\draw[gp path] (2.814,4.263)--(2.814,4.083);
	\node[gp node center] at (2.814,0.308) {$2$};
	\draw[gp path] (3.624,0.616)--(3.624,0.796);
	\draw[gp path] (3.624,4.263)--(3.624,4.083);
	\node[gp node center] at (3.624,0.308) {$3$};
	\draw[gp path] (4.433,0.616)--(4.433,0.796);
	\draw[gp path] (4.433,4.263)--(4.433,4.083);
	\node[gp node center] at (4.433,0.308) {$4$};
	\draw[gp path] (5.242,0.616)--(5.242,0.796);
	\draw[gp path] (5.242,4.263)--(5.242,4.083);
	\node[gp node center] at (5.242,0.308) {$5$};
	\draw[gp path] (6.051,0.616)--(6.051,0.796);
	\draw[gp path] (6.051,4.263)--(6.051,4.083);
	\node[gp node center] at (6.051,0.308) {$6$};
	\draw[gp path] (1.196,4.263)--(1.196,0.616)--(6.051,0.616)--(6.051,4.263)--cycle;
	\draw[gp path] (1.695,4.263)--(1.735,3.781)--(1.784,3.300)--(1.834,2.912)--(1.883,2.602)%
	--(1.932,2.357)--(1.981,2.165)--(2.030,2.018)--(2.079,1.907)--(2.128,1.826)--(2.177,1.769)%
	--(2.226,1.732)--(2.275,1.710)--(2.324,1.701)--(2.373,1.702)--(2.422,1.710)--(2.471,1.725)%
	--(2.520,1.744)--(2.569,1.766)--(2.618,1.790)--(2.667,1.816)--(2.716,1.843)--(2.765,1.870)%
	--(2.814,1.897)--(2.863,1.924)--(2.912,1.951)--(2.961,1.976)--(3.010,2.001)--(3.060,2.025)%
	--(3.109,2.048)--(3.158,2.070)--(3.207,2.091)--(3.256,2.111)--(3.305,2.130)--(3.354,2.148)%
	--(3.403,2.165)--(3.452,2.181)--(3.501,2.196)--(3.550,2.210)--(3.599,2.223)--(3.648,2.236)%
	--(3.697,2.248)--(3.746,2.259)--(3.795,2.269)--(3.844,2.279)--(3.893,2.288)--(3.942,2.297)%
	--(3.991,2.305)--(4.040,2.312)--(4.089,2.319)--(4.138,2.326)--(4.187,2.332)--(4.237,2.338)%
	--(4.286,2.344)--(4.335,2.349)--(4.384,2.354)--(4.433,2.358)--(4.482,2.362)--(4.531,2.366)%
	--(4.580,2.370)--(4.629,2.374)--(4.678,2.377)--(4.727,2.380)--(4.776,2.383)--(4.825,2.386)%
	--(4.874,2.389)--(4.923,2.391)--(4.972,2.393)--(5.021,2.396)--(5.070,2.398)--(5.119,2.400)%
	--(5.168,2.402)--(5.217,2.403)--(5.266,2.405)--(5.315,2.406)--(5.364,2.408)--(5.413,2.409)%
	--(5.463,2.411)--(5.512,2.412)--(5.561,2.413)--(5.610,2.414)--(5.659,2.415)--(5.708,2.416)%
	--(5.757,2.417)--(5.806,2.418)--(5.855,2.419)--(5.904,2.420)--(5.953,2.421)--(6.002,2.421)%
	--(6.051,2.422);
	\draw[gp path,dash dot dot] (1.196,2.440)--(1.245,2.440)--(1.294,2.440)--(1.343,2.440)--(1.392,2.440)%
	--(1.441,2.440)--(1.490,2.440)--(1.539,2.440)--(1.588,2.440)--(1.637,2.440)--(1.686,2.440)%
	--(1.735,2.440)--(1.784,2.440)--(1.834,2.440)--(1.883,2.440)--(1.932,2.440)--(1.981,2.440)%
	--(2.030,2.440)--(2.079,2.440)--(2.128,2.440)--(2.177,2.440)--(2.226,2.440)--(2.275,2.440)%
	--(2.324,2.440)--(2.373,2.440)--(2.422,2.440)--(2.471,2.440)--(2.520,2.440)--(2.569,2.440)%
	--(2.618,2.440)--(2.667,2.440)--(2.716,2.440)--(2.765,2.440)--(2.814,2.440)--(2.863,2.440)%
	--(2.912,2.440)--(2.961,2.440)--(3.010,2.440)--(3.060,2.440)--(3.109,2.440)--(3.158,2.440)%
	--(3.207,2.440)--(3.256,2.440)--(3.305,2.440)--(3.354,2.440)--(3.403,2.440)--(3.452,2.440)%
	--(3.501,2.440)--(3.550,2.440)--(3.599,2.440)--(3.648,2.440)--(3.697,2.440)--(3.746,2.440)%
	--(3.795,2.440)--(3.844,2.440)--(3.893,2.440)--(3.942,2.440)--(3.991,2.440)--(4.040,2.440)%
	--(4.089,2.440)--(4.138,2.440)--(4.187,2.440)--(4.237,2.440)--(4.286,2.440)--(4.335,2.440)%
	--(4.384,2.440)--(4.433,2.440)--(4.482,2.440)--(4.531,2.440)--(4.580,2.440)--(4.629,2.440)%
	--(4.678,2.440)--(4.727,2.440)--(4.776,2.440)--(4.825,2.440)--(4.874,2.440)--(4.923,2.440)%
	--(4.972,2.440)--(5.021,2.440)--(5.070,2.440)--(5.119,2.440)--(5.168,2.440)--(5.217,2.440)%
	--(5.266,2.440)--(5.315,2.440)--(5.364,2.440)--(5.413,2.440)--(5.463,2.440)--(5.512,2.440)%
	--(5.561,2.440)--(5.610,2.440)--(5.659,2.440)--(5.708,2.440)--(5.757,2.440)--(5.806,2.440)%
	--(5.855,2.440)--(5.904,2.440)--(5.953,2.440)--(6.002,2.440)--(6.051,2.440);
	\draw[gp path,dash dot dot] (1.920,0.616)--(1.920,2.440);
	\gpcolor{color=gp lt color border}
	\draw[gp path] (1.196,4.263)--(1.196,0.616)--(6.051,0.616)--(6.051,4.263)--cycle;
	\gpdefrectangularnode{gp plot 1}{\pgfpoint{1.196cm}{0.616cm}}{\pgfpoint{6.051cm}{4.263cm}}
\end{tikzpicture}

%% file: graph_DiracA-4D.tex
			\begin{tikzpicture}[gnuplot]
	\tikzset{every node/.append style={font={\footnotesize}}}
	\path (0.000,0.000) rectangle (6.604,4.572);
	\gpcolor{color=gp lt color border}
	\gpsetlinetype{gp lt border}
	\gpsetlinewidth{1.00}
	\draw[gp path] (1.012,0.616)--(1.192,0.616);
	\draw[gp path] (6.051,0.616)--(5.871,0.616);
	\node[gp node right] at (0.828,0.616) {$0$};
	\draw[gp path] (1.012,1.528)--(1.192,1.528);
	\draw[gp path] (6.051,1.528)--(5.871,1.528);
	\node[gp node right] at (0.828,1.528) {$0.5$};
	\draw[gp path] (1.012,2.440)--(1.192,2.440);
	\draw[gp path] (6.051,2.440)--(5.871,2.440);
	\node[gp node right] at (0.828,2.440) {$1$};
	\draw[gp path] (1.012,3.351)--(1.192,3.351);
	\draw[gp path] (6.051,3.351)--(5.871,3.351);
	\node[gp node right] at (0.828,3.351) {$1.5$};
	\draw[gp path] (1.012,4.263)--(1.192,4.263);
	\draw[gp path] (6.051,4.263)--(5.871,4.263);
	\node[gp node right] at (0.828,4.263) {$2$};
	\draw[gp path] (1.012,0.616)--(1.012,0.796);
	\draw[gp path] (1.012,4.263)--(1.012,4.083);
	\node[gp node center] at (1.012,0.308) {$0$};
	\draw[gp path] (1.852,0.616)--(1.852,0.796);
	\draw[gp path] (1.852,4.263)--(1.852,4.083);
	\node[gp node center] at (1.852,0.308) {$1$};
	\draw[gp path] (2.692,0.616)--(2.692,0.796);
	\draw[gp path] (2.692,4.263)--(2.692,4.083);
	\node[gp node center] at (2.692,0.308) {$2$};
	\draw[gp path] (3.532,0.616)--(3.532,0.796);
	\draw[gp path] (3.532,4.263)--(3.532,4.083);
	\node[gp node center] at (3.532,0.308) {$3$};
	\draw[gp path] (4.371,0.616)--(4.371,0.796);
	\draw[gp path] (4.371,4.263)--(4.371,4.083);
	\node[gp node center] at (4.371,0.308) {$4$};
	\draw[gp path] (5.211,0.616)--(5.211,0.796);
	\draw[gp path] (5.211,4.263)--(5.211,4.083);
	\node[gp node center] at (5.211,0.308) {$5$};
	\draw[gp path] (6.051,0.616)--(6.051,0.796);
	\draw[gp path] (6.051,4.263)--(6.051,4.083);
	\node[gp node center] at (6.051,0.308) {$6$};
	\draw[gp path] (1.012,4.263)--(1.012,0.616)--(6.051,0.616)--(6.051,4.263)--cycle;
	\draw[gp path,dash dot dot] (1.012,2.440)--(1.037,2.440)--(1.062,2.440)--(1.088,2.440)--(1.113,2.440)%
	--(1.138,2.440)--(1.163,2.440)--(1.188,2.440)--(1.214,2.440)--(1.239,2.440)--(1.264,2.440)%
	--(1.289,2.440)--(1.314,2.440)--(1.340,2.440)--(1.365,2.440)--(1.390,2.440)--(1.415,2.440)%
	--(1.440,2.440)--(1.466,2.440)--(1.491,2.440)--(1.516,2.440)--(1.541,2.440)--(1.566,2.440)%
	--(1.591,2.440)--(1.617,2.440)--(1.642,2.440)--(1.667,2.440)--(1.692,2.440)--(1.717,2.440)%
	--(1.743,2.440)--(1.768,2.440)--(1.793,2.440)--(1.818,2.440)--(1.843,2.440)--(1.869,2.440)%
	--(1.894,2.440)--(1.919,2.440)--(1.944,2.440)--(1.969,2.440)--(1.995,2.440)--(2.020,2.440)%
	--(2.045,2.440)--(2.070,2.440)--(2.095,2.440)--(2.121,2.440)--(2.146,2.440)--(2.171,2.440)%
	--(2.196,2.440)--(2.221,2.440)--(2.247,2.440)--(2.272,2.440)--(2.297,2.440)--(2.322,2.440)%
	--(2.347,2.440)--(2.373,2.440)--(2.398,2.440)--(2.423,2.440)--(2.448,2.440)--(2.473,2.440)%
	--(2.499,2.440)--(2.524,2.440)--(2.549,2.440)--(2.574,2.440)--(2.599,2.440)--(2.624,2.440)%
	--(2.650,2.440)--(2.675,2.440)--(2.700,2.440)--(2.725,2.440)--(2.750,2.440)--(2.776,2.440)%
	--(2.801,2.440)--(2.826,2.440)--(2.851,2.440)--(2.876,2.440)--(2.902,2.440)--(2.927,2.440)%
	--(2.952,2.440)--(2.977,2.440)--(3.002,2.440)--(3.028,2.440)--(3.053,2.440)--(3.078,2.440)%
	--(3.103,2.440)--(3.128,2.440)--(3.154,2.440)--(3.179,2.440)--(3.204,2.440)--(3.229,2.440)%
	--(3.254,2.440)--(3.280,2.440)--(3.305,2.440)--(3.330,2.440)--(3.355,2.440)--(3.380,2.440)%
	--(3.406,2.440)--(3.431,2.440)--(3.456,2.440)--(3.481,2.440)--(3.506,2.440)--(3.532,2.440)%
	--(3.557,2.440)--(3.582,2.440)--(3.607,2.440)--(3.632,2.440)--(3.657,2.440)--(3.683,2.440)%
	--(3.708,2.440)--(3.733,2.440)--(3.758,2.440)--(3.783,2.440)--(3.809,2.440)--(3.834,2.440)%
	--(3.859,2.440)--(3.884,2.440)--(3.909,2.440)--(3.935,2.440)--(3.960,2.440)--(3.985,2.440)%
	--(4.010,2.440)--(4.035,2.440)--(4.061,2.440)--(4.086,2.440)--(4.111,2.440)--(4.136,2.440)%
	--(4.161,2.440)--(4.187,2.440)--(4.212,2.440)--(4.237,2.440)--(4.262,2.440)--(4.287,2.440)%
	--(4.313,2.440)--(4.338,2.440)--(4.363,2.440)--(4.388,2.440)--(4.413,2.440)--(4.439,2.440)%
	--(4.464,2.440)--(4.489,2.440)--(4.514,2.440)--(4.539,2.440)--(4.564,2.440)--(4.590,2.440)%
	--(4.615,2.440)--(4.640,2.440)--(4.665,2.440)--(4.690,2.440)--(4.716,2.440)--(4.741,2.440)%
	--(4.766,2.440)--(4.791,2.440)--(4.816,2.440)--(4.842,2.440)--(4.867,2.440)--(4.892,2.440)%
	--(4.917,2.440)--(4.942,2.440)--(4.968,2.440)--(4.993,2.440)--(5.018,2.440)--(5.043,2.440)%
	--(5.068,2.440)--(5.094,2.440)--(5.119,2.440)--(5.144,2.440)--(5.169,2.440)--(5.194,2.440)%
	--(5.220,2.440)--(5.245,2.440)--(5.270,2.440)--(5.295,2.440)--(5.320,2.440)--(5.346,2.440)%
	--(5.371,2.440)--(5.396,2.440)--(5.421,2.440)--(5.446,2.440)--(5.472,2.440)--(5.497,2.440)%
	--(5.522,2.440)--(5.547,2.440)--(5.572,2.440)--(5.597,2.440)--(5.623,2.440)--(5.648,2.440)%
	--(5.673,2.440)--(5.698,2.440)--(5.723,2.440)--(5.749,2.440)--(5.774,2.440)--(5.799,2.440)%
	--(5.824,2.440)--(5.849,2.440)--(5.875,2.440)--(5.900,2.440)--(5.925,2.440)--(5.950,2.440)%
	--(5.975,2.440)--(6.001,2.440)--(6.026,2.440)--(6.051,2.440);
	\draw[gp path] (1.012,1.984)--(1.063,1.981)--(1.114,1.976)--(1.165,1.972)--(1.216,1.969)--(1.266,1.968)%
	--(1.317,1.969)--(1.368,1.972)--(1.419,1.977)--(1.470,1.983)--(1.521,1.991)--(1.572,2.000)%
	--(1.623,2.010)--(1.674,2.021)--(1.725,2.032)--(1.775,2.044)--(1.826,2.056)--(1.877,2.068)%
	--(1.928,2.080)--(1.979,2.093)--(2.030,2.105)--(2.081,2.117)--(2.132,2.129)--(2.183,2.140)%
	--(2.234,2.151)--(2.284,2.162)--(2.335,2.173)--(2.386,2.183)--(2.437,2.192)--(2.488,2.202)%
	--(2.539,2.211)--(2.590,2.220)--(2.641,2.228)--(2.692,2.236)--(2.743,2.244)--(2.793,2.251)%
	--(2.844,2.258)--(2.895,2.265)--(2.946,2.271)--(2.997,2.277)--(3.048,2.283)--(3.099,2.289)%
	--(3.150,2.294)--(3.201,2.299)--(3.252,2.304)--(3.302,2.309)--(3.353,2.314)--(3.404,2.318)%
	--(3.455,2.322)--(3.506,2.326)--(3.557,2.330)--(3.608,2.333)--(3.659,2.337)--(3.710,2.340)%
	--(3.761,2.343)--(3.811,2.346)--(3.862,2.349)--(3.913,2.352)--(3.964,2.354)--(4.015,2.357)%
	--(4.066,2.359)--(4.117,2.362)--(4.168,2.364)--(4.219,2.366)--(4.270,2.368)--(4.320,2.370)%
	--(4.371,2.372)--(4.422,2.374)--(4.473,2.376)--(4.524,2.378)--(4.575,2.379)--(4.626,2.381)%
	--(4.677,2.382)--(4.728,2.384)--(4.779,2.385)--(4.829,2.387)--(4.880,2.388)--(4.931,2.389)%
	--(4.982,2.390)--(5.033,2.391)--(5.084,2.393)--(5.135,2.394)--(5.186,2.395)--(5.237,2.396)%
	--(5.288,2.397)--(5.338,2.398)--(5.389,2.399)--(5.440,2.400)--(5.491,2.400)--(5.542,2.401)%
	--(5.593,2.402)--(5.644,2.403)--(5.695,2.404)--(5.746,2.404)--(5.797,2.405)--(5.847,2.406)%
	--(5.898,2.407)--(5.949,2.407)--(6.000,2.408)--(6.051,2.408);
	\gpcolor{color=gp lt color border}
	\draw[gp path] (1.012,4.263)--(1.012,0.616)--(6.051,0.616)--(6.051,4.263)--cycle;
	\gpdefrectangularnode{gp plot 1}{\pgfpoint{1.012cm}{0.616cm}}{\pgfpoint{6.051cm}{4.263cm}}
\end{tikzpicture}

%% file: graph_DiracA-5D.tex
			\begin{tikzpicture}[gnuplot]
	\tikzset{every node/.append style={font={\footnotesize}}}
	\path (0.000,0.000) rectangle (6.604,4.572);
	\gpcolor{color=gp lt color border}
	\gpsetlinetype{gp lt border}
	\gpsetdashtype{gp dt solid}
	\gpsetlinewidth{1.00}
	\draw[gp path] (1.012,0.616)--(1.192,0.616);
	\draw[gp path] (6.051,0.616)--(5.871,0.616);
	\node[gp node right] at (0.828,0.616) {$0$};
	\draw[gp path] (1.012,1.528)--(1.192,1.528);
	\draw[gp path] (6.051,1.528)--(5.871,1.528);
	\node[gp node right] at (0.828,1.528) {$0.5$};
	\draw[gp path] (1.012,2.440)--(1.192,2.440);
	\draw[gp path] (6.051,2.440)--(5.871,2.440);
	\node[gp node right] at (0.828,2.440) {$1$};
	\draw[gp path] (1.012,3.351)--(1.192,3.351);
	\draw[gp path] (6.051,3.351)--(5.871,3.351);
	\node[gp node right] at (0.828,3.351) {$1.5$};
	\draw[gp path] (1.012,4.263)--(1.192,4.263);
	\draw[gp path] (6.051,4.263)--(5.871,4.263);
	\node[gp node right] at (0.828,4.263) {$2$};
	\draw[gp path] (1.012,0.616)--(1.012,0.796);
	\draw[gp path] (1.012,4.263)--(1.012,4.083);
	\node[gp node center] at (1.012,0.308) {$0$};
	\draw[gp path] (1.852,0.616)--(1.852,0.796);
	\draw[gp path] (1.852,4.263)--(1.852,4.083);
	\node[gp node center] at (1.852,0.308) {$1$};
	\draw[gp path] (2.692,0.616)--(2.692,0.796);
	\draw[gp path] (2.692,4.263)--(2.692,4.083);
	\node[gp node center] at (2.692,0.308) {$2$};
	\draw[gp path] (3.532,0.616)--(3.532,0.796);
	\draw[gp path] (3.532,4.263)--(3.532,4.083);
	\node[gp node center] at (3.532,0.308) {$3$};
	\draw[gp path] (4.371,0.616)--(4.371,0.796);
	\draw[gp path] (4.371,4.263)--(4.371,4.083);
	\node[gp node center] at (4.371,0.308) {$4$};
	\draw[gp path] (5.211,0.616)--(5.211,0.796);
	\draw[gp path] (5.211,4.263)--(5.211,4.083);
	\node[gp node center] at (5.211,0.308) {$5$};
	\draw[gp path] (6.051,0.616)--(6.051,0.796);
	\draw[gp path] (6.051,4.263)--(6.051,4.083);
	\node[gp node center] at (6.051,0.308) {$6$};
	\draw[gp path] (1.012,4.263)--(1.012,0.616)--(6.051,0.616)--(6.051,4.263)--cycle;
	\draw[gp path,dash dot dot] (1.012,2.440)--(1.037,2.440)--(1.062,2.440)--(1.088,2.440)--(1.113,2.440)%
	--(1.138,2.440)--(1.163,2.440)--(1.188,2.440)--(1.214,2.440)--(1.239,2.440)--(1.264,2.440)%
	--(1.289,2.440)--(1.314,2.440)--(1.340,2.440)--(1.365,2.440)--(1.390,2.440)--(1.415,2.440)%
	--(1.440,2.440)--(1.466,2.440)--(1.491,2.440)--(1.516,2.440)--(1.541,2.440)--(1.566,2.440)%
	--(1.591,2.440)--(1.617,2.440)--(1.642,2.440)--(1.667,2.440)--(1.692,2.440)--(1.717,2.440)%
	--(1.743,2.440)--(1.768,2.440)--(1.793,2.440)--(1.818,2.440)--(1.843,2.440)--(1.869,2.440)%
	--(1.894,2.440)--(1.919,2.440)--(1.944,2.440)--(1.969,2.440)--(1.995,2.440)--(2.020,2.440)%
	--(2.045,2.440)--(2.070,2.440)--(2.095,2.440)--(2.121,2.440)--(2.146,2.440)--(2.171,2.440)%
	--(2.196,2.440)--(2.221,2.440)--(2.247,2.440)--(2.272,2.440)--(2.297,2.440)--(2.322,2.440)%
	--(2.347,2.440)--(2.373,2.440)--(2.398,2.440)--(2.423,2.440)--(2.448,2.440)--(2.473,2.440)%
	--(2.499,2.440)--(2.524,2.440)--(2.549,2.440)--(2.574,2.440)--(2.599,2.440)--(2.624,2.440)%
	--(2.650,2.440)--(2.675,2.440)--(2.700,2.440)--(2.725,2.440)--(2.750,2.440)--(2.776,2.440)%
	--(2.801,2.440)--(2.826,2.440)--(2.851,2.440)--(2.876,2.440)--(2.902,2.440)--(2.927,2.440)%
	--(2.952,2.440)--(2.977,2.440)--(3.002,2.440)--(3.028,2.440)--(3.053,2.440)--(3.078,2.440)%
	--(3.103,2.440)--(3.128,2.440)--(3.154,2.440)--(3.179,2.440)--(3.204,2.440)--(3.229,2.440)%
	--(3.254,2.440)--(3.280,2.440)--(3.305,2.440)--(3.330,2.440)--(3.355,2.440)--(3.380,2.440)%
	--(3.406,2.440)--(3.431,2.440)--(3.456,2.440)--(3.481,2.440)--(3.506,2.440)--(3.532,2.440)%
	--(3.557,2.440)--(3.582,2.440)--(3.607,2.440)--(3.632,2.440)--(3.657,2.440)--(3.683,2.440)%
	--(3.708,2.440)--(3.733,2.440)--(3.758,2.440)--(3.783,2.440)--(3.809,2.440)--(3.834,2.440)%
	--(3.859,2.440)--(3.884,2.440)--(3.909,2.440)--(3.935,2.440)--(3.960,2.440)--(3.985,2.440)%
	--(4.010,2.440)--(4.035,2.440)--(4.061,2.440)--(4.086,2.440)--(4.111,2.440)--(4.136,2.440)%
	--(4.161,2.440)--(4.187,2.440)--(4.212,2.440)--(4.237,2.440)--(4.262,2.440)--(4.287,2.440)%
	--(4.313,2.440)--(4.338,2.440)--(4.363,2.440)--(4.388,2.440)--(4.413,2.440)--(4.439,2.440)%
	--(4.464,2.440)--(4.489,2.440)--(4.514,2.440)--(4.539,2.440)--(4.564,2.440)--(4.590,2.440)%
	--(4.615,2.440)--(4.640,2.440)--(4.665,2.440)--(4.690,2.440)--(4.716,2.440)--(4.741,2.440)%
	--(4.766,2.440)--(4.791,2.440)--(4.816,2.440)--(4.842,2.440)--(4.867,2.440)--(4.892,2.440)%
	--(4.917,2.440)--(4.942,2.440)--(4.968,2.440)--(4.993,2.440)--(5.018,2.440)--(5.043,2.440)%
	--(5.068,2.440)--(5.094,2.440)--(5.119,2.440)--(5.144,2.440)--(5.169,2.440)--(5.194,2.440)%
	--(5.220,2.440)--(5.245,2.440)--(5.270,2.440)--(5.295,2.440)--(5.320,2.440)--(5.346,2.440)%
	--(5.371,2.440)--(5.396,2.440)--(5.421,2.440)--(5.446,2.440)--(5.472,2.440)--(5.497,2.440)%
	--(5.522,2.440)--(5.547,2.440)--(5.572,2.440)--(5.597,2.440)--(5.623,2.440)--(5.648,2.440)%
	--(5.673,2.440)--(5.698,2.440)--(5.723,2.440)--(5.749,2.440)--(5.774,2.440)--(5.799,2.440)%
	--(5.824,2.440)--(5.849,2.440)--(5.875,2.440)--(5.900,2.440)--(5.925,2.440)--(5.950,2.440)%
	--(5.975,2.440)--(6.001,2.440)--(6.026,2.440)--(6.051,2.440);
	\draw[gp path] (1.012,1.5888)--(1.063,1.589)--(1.114,1.592)--(1.165,1.596)--(1.216,1.603)--(1.266,1.611)%
	--(1.317,1.621)--(1.368,1.632)--(1.419,1.645)--(1.470,1.659)--(1.521,1.675)--(1.572,1.691)%
	--(1.623,1.708)--(1.674,1.725)--(1.725,1.743)--(1.775,1.762)--(1.826,1.780)--(1.877,1.798)%
	--(1.928,1.817)--(1.979,1.835)--(2.030,1.853)--(2.081,1.871)--(2.132,1.888)--(2.183,1.906)%
	--(2.234,1.922)--(2.284,1.939)--(2.335,1.955)--(2.386,1.970)--(2.437,1.985)--(2.488,2.000)%
	--(2.539,2.014)--(2.590,2.028)--(2.641,2.041)--(2.692,2.054)--(2.743,2.066)--(2.793,2.078)%
	--(2.844,2.089)--(2.895,2.100)--(2.946,2.111)--(2.997,2.121)--(3.048,2.131)--(3.099,2.141)%
	--(3.150,2.150)--(3.201,2.159)--(3.252,2.167)--(3.302,2.176)--(3.353,2.184)--(3.404,2.191)%
	--(3.455,2.199)--(3.506,2.206)--(3.557,2.212)--(3.608,2.219)--(3.659,2.225)--(3.710,2.231)%
	--(3.761,2.237)--(3.811,2.243)--(3.862,2.248)--(3.913,2.254)--(3.964,2.259)--(4.015,2.263)%
	--(4.066,2.268)--(4.117,2.273)--(4.168,2.277)--(4.219,2.281)--(4.270,2.285)--(4.320,2.289)%
	--(4.371,2.293)--(4.422,2.297)--(4.473,2.300)--(4.524,2.304)--(4.575,2.307)--(4.626,2.310)%
	--(4.677,2.313)--(4.728,2.316)--(4.779,2.319)--(4.829,2.322)--(4.880,2.325)--(4.931,2.327)%
	--(4.982,2.330)--(5.033,2.332)--(5.084,2.335)--(5.135,2.337)--(5.186,2.339)--(5.237,2.341)%
	--(5.288,2.343)--(5.338,2.345)--(5.389,2.347)--(5.440,2.349)--(5.491,2.351)--(5.542,2.353)%
	--(5.593,2.355)--(5.644,2.356)--(5.695,2.358)--(5.746,2.360)--(5.797,2.361)--(5.847,2.363)%
	--(5.898,2.364)--(5.949,2.366)--(6.000,2.367)--(6.051,2.368);
	\gpcolor{color=gp lt color border}
	\draw[gp path] (1.012,4.263)--(1.012,0.616)--(6.051,0.616)--(6.051,4.263)--cycle;
	\gpdefrectangularnode{gp plot 1}{\pgfpoint{1.012cm}{0.616cm}}{\pgfpoint{6.051cm}{4.263cm}}
\end{tikzpicture}

%% file: graph_C-K0.tex
			\begin{tikzpicture}[gnuplot]
	\tikzset{every node/.append style={font={\footnotesize}}}
	\path (0.000,0.000) rectangle (6.604,4.572);
	\gpsetlinetype{gp lt border}
	\gpsetlinewidth{1.00}
	\draw[gp path] (1.012,0.616)--(1.192,0.616);
	\draw[gp path] (6.051,0.616)--(5.871,0.616);
	\node[gp node right] at (0.828,0.616) {$0$};
	\draw[gp path] (1.012,1.528)--(1.192,1.528);
	\draw[gp path] (6.051,1.528)--(5.871,1.528);
	\node[gp node right] at (0.828,1.528) {$0.5$};
	\draw[gp path] (1.012,2.440)--(1.192,2.440);
	\draw[gp path] (6.051,2.440)--(5.871,2.440);
	\node[gp node right] at (0.828,2.440) {$1$};
	\draw[gp path] (1.012,3.351)--(1.192,3.351);
	\draw[gp path] (6.051,3.351)--(5.871,3.351);
	\node[gp node right] at (0.828,3.351) {$1.5$};
	\draw[gp path] (1.012,4.263)--(1.192,4.263);
	\draw[gp path] (6.051,4.263)--(5.871,4.263);
	\node[gp node right] at (0.828,4.263) {$2$};
	\draw[gp path] (1.012,0.616)--(1.012,0.796);
	\draw[gp path] (1.012,4.263)--(1.012,4.083);
	\node[gp node center] at (1.012,0.308) {$0$};
	\draw[gp path] (1.852,0.616)--(1.852,0.796);
	\draw[gp path] (1.852,4.263)--(1.852,4.083);
	\node[gp node center] at (1.852,0.308) {$1$};
	\draw[gp path] (2.692,0.616)--(2.692,0.796);
	\draw[gp path] (2.692,4.263)--(2.692,4.083);
	\node[gp node center] at (2.692,0.308) {$2$};
	\draw[gp path] (3.532,0.616)--(3.532,0.796);
	\draw[gp path] (3.532,4.263)--(3.532,4.083);
	\node[gp node center] at (3.532,0.308) {$3$};
	\draw[gp path] (4.371,0.616)--(4.371,0.796);
	\draw[gp path] (4.371,4.263)--(4.371,4.083);
	\node[gp node center] at (4.371,0.308) {$4$};
	\draw[gp path] (5.211,0.616)--(5.211,0.796);
	\draw[gp path] (5.211,4.263)--(5.211,4.083);
	\node[gp node center] at (5.211,0.308) {$5$};
	\draw[gp path] (6.051,0.616)--(6.051,0.796);
	\draw[gp path] (6.051,4.263)--(6.051,4.083);
	\node[gp node center] at (6.051,0.308) {$6$};
	\draw[gp path] (1.012,4.263)--(1.012,0.616)--(6.051,0.616)--(6.051,4.263)--cycle;
	\draw[gp path] (1.012,0.616)--(1.037,0.725)--(1.062,0.833)--(1.088,0.939)--(1.113,1.042)--(1.138,1.140)%
	--(1.163,1.234)--(1.188,1.322)--(1.214,1.405)--(1.239,1.482)--(1.264,1.554)--(1.289,1.620)%
	--(1.314,1.681)--(1.340,1.738)--(1.365,1.789)--(1.390,1.836)--(1.415,1.879)--(1.440,1.918)%
	--(1.466,1.954)--(1.491,1.987)--(1.516,2.017)--(1.541,2.044)--(1.566,2.069)--(1.591,2.093)%
	--(1.617,2.114)--(1.642,2.133)--(1.667,2.151)--(1.692,2.168)--(1.717,2.183)--(1.743,2.197)%
	--(1.768,2.210)--(1.793,2.222)--(1.818,2.233)--(1.843,2.244)--(1.869,2.253)--(1.894,2.262)%
	--(1.919,2.271)--(1.944,2.279)--(1.969,2.286)--(1.995,2.293)--(2.020,2.299)--(2.045,2.305)%
	--(2.070,2.311)--(2.095,2.316)--(2.121,2.321)--(2.146,2.326)--(2.171,2.330)--(2.196,2.335)%
	--(2.221,2.339)--(2.247,2.342)--(2.272,2.346)--(2.297,2.349)--(2.322,2.352)--(2.347,2.356)%
	--(2.373,2.358)--(2.398,2.361)--(2.423,2.364)--(2.448,2.366)--(2.473,2.369)--(2.499,2.371)%
	--(2.524,2.373)--(2.549,2.375)--(2.574,2.377)--(2.599,2.379)--(2.624,2.381)--(2.650,2.382)%
	--(2.675,2.384)--(2.700,2.386)--(2.725,2.387)--(2.750,2.389)--(2.776,2.390)--(2.801,2.391)%
	--(2.826,2.393)--(2.851,2.394)--(2.876,2.395)--(2.902,2.396)--(2.927,2.397)--(2.952,2.398)%
	--(2.977,2.399)--(3.002,2.400)--(3.028,2.401)--(3.053,2.402)--(3.078,2.403)--(3.103,2.404)%
	--(3.128,2.405)--(3.154,2.405)--(3.179,2.406)--(3.204,2.407)--(3.229,2.408)--(3.254,2.408)%
	--(3.280,2.409)--(3.305,2.410)--(3.330,2.410)--(3.355,2.411)--(3.380,2.411)--(3.406,2.412)%
	--(3.431,2.413)--(3.456,2.413)--(3.481,2.414)--(3.506,2.414)--(3.532,2.415)--(3.557,2.415)%
	--(3.582,2.416)--(3.607,2.416)--(3.632,2.417)--(3.657,2.417)--(3.683,2.417)--(3.708,2.418)%
	--(3.733,2.418)--(3.758,2.419)--(3.783,2.419)--(3.809,2.419)--(3.834,2.420)--(3.859,2.420)%
	--(3.884,2.420)--(3.909,2.421)--(3.935,2.421)--(3.960,2.421)--(3.985,2.422)--(4.010,2.422)%
	--(4.035,2.422)--(4.061,2.422)--(4.086,2.423)--(4.111,2.423)--(4.136,2.423)--(4.161,2.424)%
	--(4.187,2.424)--(4.212,2.424)--(4.237,2.424)--(4.262,2.424)--(4.287,2.425)--(4.313,2.425)%
	--(4.338,2.425)--(4.363,2.425)--(4.388,2.426)--(4.413,2.426)--(4.439,2.426)--(4.464,2.426)%
	--(4.489,2.426)--(4.514,2.427)--(4.539,2.427)--(4.564,2.427)--(4.590,2.427)--(4.615,2.427)%
	--(4.640,2.427)--(4.665,2.428)--(4.690,2.428)--(4.716,2.428)--(4.741,2.428)--(4.766,2.428)%
	--(4.791,2.428)--(4.816,2.428)--(4.842,2.429)--(4.867,2.429)--(4.892,2.429)--(4.917,2.429)%
	--(4.942,2.429)--(4.968,2.429)--(4.993,2.429)--(5.018,2.430)--(5.043,2.430)--(5.068,2.430)%
	--(5.094,2.430)--(5.119,2.430)--(5.144,2.430)--(5.169,2.430)--(5.194,2.430)--(5.220,2.430)%
	--(5.245,2.431)--(5.270,2.431)--(5.295,2.431)--(5.320,2.431)--(5.346,2.431)--(5.371,2.431)%
	--(5.396,2.431)--(5.421,2.431)--(5.446,2.431)--(5.472,2.431)--(5.497,2.432)--(5.522,2.432)%
	--(5.547,2.432)--(5.572,2.432)--(5.597,2.432)--(5.623,2.432)--(5.648,2.432)--(5.673,2.432)%
	--(5.698,2.432)--(5.723,2.432)--(5.749,2.432)--(5.774,2.432)--(5.799,2.433)--(5.824,2.433)%
	--(5.849,2.433)--(5.875,2.433)--(5.900,2.433)--(5.925,2.433)--(5.950,2.433)--(5.975,2.433)%
	--(6.001,2.433)--(6.026,2.433)--(6.051,2.433);
	\draw[gp path,dash dot dot] (1.012,2.440)--(1.037,2.440)--(1.062,2.440)--(1.088,2.440)--(1.113,2.440)%
	--(1.138,2.440)--(1.163,2.440)--(1.188,2.440)--(1.214,2.440)--(1.239,2.440)--(1.264,2.440)%
	--(1.289,2.440)--(1.314,2.440)--(1.340,2.440)--(1.365,2.440)--(1.390,2.440)--(1.415,2.440)%
	--(1.440,2.440)--(1.466,2.440)--(1.491,2.440)--(1.516,2.440)--(1.541,2.440)--(1.566,2.440)%
	--(1.591,2.440)--(1.617,2.440)--(1.642,2.440)--(1.667,2.440)--(1.692,2.440)--(1.717,2.440)%
	--(1.743,2.440)--(1.768,2.440)--(1.793,2.440)--(1.818,2.440)--(1.843,2.440)--(1.869,2.440)%
	--(1.894,2.440)--(1.919,2.440)--(1.944,2.440)--(1.969,2.440)--(1.995,2.440)--(2.020,2.440)%
	--(2.045,2.440)--(2.070,2.440)--(2.095,2.440)--(2.121,2.440)--(2.146,2.440)--(2.171,2.440)%
	--(2.196,2.440)--(2.221,2.440)--(2.247,2.440)--(2.272,2.440)--(2.297,2.440)--(2.322,2.440)%
	--(2.347,2.440)--(2.373,2.440)--(2.398,2.440)--(2.423,2.440)--(2.448,2.440)--(2.473,2.440)%
	--(2.499,2.440)--(2.524,2.440)--(2.549,2.440)--(2.574,2.440)--(2.599,2.440)--(2.624,2.440)%
	--(2.650,2.440)--(2.675,2.440)--(2.700,2.440)--(2.725,2.440)--(2.750,2.440)--(2.776,2.440)%
	--(2.801,2.440)--(2.826,2.440)--(2.851,2.440)--(2.876,2.440)--(2.902,2.440)--(2.927,2.440)%
	--(2.952,2.440)--(2.977,2.440)--(3.002,2.440)--(3.028,2.440)--(3.053,2.440)--(3.078,2.440)%
	--(3.103,2.440)--(3.128,2.440)--(3.154,2.440)--(3.179,2.440)--(3.204,2.440)--(3.229,2.440)%
	--(3.254,2.440)--(3.280,2.440)--(3.305,2.440)--(3.330,2.440)--(3.355,2.440)--(3.380,2.440)%
	--(3.406,2.440)--(3.431,2.440)--(3.456,2.440)--(3.481,2.440)--(3.506,2.440)--(3.532,2.440)%
	--(3.557,2.440)--(3.582,2.440)--(3.607,2.440)--(3.632,2.440)--(3.657,2.440)--(3.683,2.440)%
	--(3.708,2.440)--(3.733,2.440)--(3.758,2.440)--(3.783,2.440)--(3.809,2.440)--(3.834,2.440)%
	--(3.859,2.440)--(3.884,2.440)--(3.909,2.440)--(3.935,2.440)--(3.960,2.440)--(3.985,2.440)%
	--(4.010,2.440)--(4.035,2.440)--(4.061,2.440)--(4.086,2.440)--(4.111,2.440)--(4.136,2.440)%
	--(4.161,2.440)--(4.187,2.440)--(4.212,2.440)--(4.237,2.440)--(4.262,2.440)--(4.287,2.440)%
	--(4.313,2.440)--(4.338,2.440)--(4.363,2.440)--(4.388,2.440)--(4.413,2.440)--(4.439,2.440)%
	--(4.464,2.440)--(4.489,2.440)--(4.514,2.440)--(4.539,2.440)--(4.564,2.440)--(4.590,2.440)%
	--(4.615,2.440)--(4.640,2.440)--(4.665,2.440)--(4.690,2.440)--(4.716,2.440)--(4.741,2.440)%
	--(4.766,2.440)--(4.791,2.440)--(4.816,2.440)--(4.842,2.440)--(4.867,2.440)--(4.892,2.440)%
	--(4.917,2.440)--(4.942,2.440)--(4.968,2.440)--(4.993,2.440)--(5.018,2.440)--(5.043,2.440)%
	--(5.068,2.440)--(5.094,2.440)--(5.119,2.440)--(5.144,2.440)--(5.169,2.440)--(5.194,2.440)%
	--(5.220,2.440)--(5.245,2.440)--(5.270,2.440)--(5.295,2.440)--(5.320,2.440)--(5.346,2.440)%
	--(5.371,2.440)--(5.396,2.440)--(5.421,2.440)--(5.446,2.440)--(5.472,2.440)--(5.497,2.440)%
	--(5.522,2.440)--(5.547,2.440)--(5.572,2.440)--(5.597,2.440)--(5.623,2.440)--(5.648,2.440)%
	--(5.673,2.440)--(5.698,2.440)--(5.723,2.440)--(5.749,2.440)--(5.774,2.440)--(5.799,2.440)%
	--(5.824,2.440)--(5.849,2.440)--(5.875,2.440)--(5.900,2.440)--(5.925,2.440)--(5.950,2.440)%
	--(5.975,2.440)--(6.001,2.440)--(6.026,2.440)--(6.051,2.440);
	\draw[gp path] (1.012,4.263)--(1.012,0.616)--(6.051,0.616)--(6.051,4.263)--cycle;
	\gpdefrectangularnode{gp plot 1}{\pgfpoint{1.012cm}{0.616cm}}{\pgfpoint{6.051cm}{4.263cm}}
\end{tikzpicture}

%% file: graph_C-Gaussian.tex
		\begin{tikzpicture}[gnuplot]
	\tikzset{every node/.append style={font={\footnotesize}}}
	\path (0.000,0.000) rectangle (6.604,4.572);
	\gpsetlinetype{gp lt border}
	\gpsetlinewidth{1.00}
	\gpsetlinewidth{1.00}
	\draw[gp path] (1.012,0.616)--(1.192,0.616);
	\draw[gp path] (6.051,0.616)--(5.871,0.616);
	\node[gp node right] at (0.828,0.616) {$0$};
	\draw[gp path] (1.012,1.528)--(1.192,1.528);
	\draw[gp path] (6.051,1.528)--(5.871,1.528);
	\node[gp node right] at (0.828,1.528) {$0.5$};
	\draw[gp path] (1.012,2.440)--(1.192,2.440);
	\draw[gp path] (6.051,2.440)--(5.871,2.440);
	\node[gp node right] at (0.828,2.440) {$1$};
	\draw[gp path] (1.012,3.351)--(1.192,3.351);
	\draw[gp path] (6.051,3.351)--(5.871,3.351);
	\node[gp node right] at (0.828,3.351) {$1.5$};
	\draw[gp path] (1.012,4.263)--(1.192,4.263);
	\draw[gp path] (6.051,4.263)--(5.871,4.263);
	\node[gp node right] at (0.828,4.263) {$2$};
	\draw[gp path] (1.012,0.616)--(1.012,0.796);
	\draw[gp path] (1.012,4.263)--(1.012,4.083);
	\node[gp node center] at (1.012,0.308) {$0$};
	\draw[gp path] (1.852,0.616)--(1.852,0.796);
	\draw[gp path] (1.852,4.263)--(1.852,4.083);
	\node[gp node center] at (1.852,0.308) {$1$};
	\draw[gp path] (2.692,0.616)--(2.692,0.796);
	\draw[gp path] (2.692,4.263)--(2.692,4.083);
	\node[gp node center] at (2.692,0.308) {$2$};
	\draw[gp path] (3.532,0.616)--(3.532,0.796);
	\draw[gp path] (3.532,4.263)--(3.532,4.083);
	\node[gp node center] at (3.532,0.308) {$3$};
	\draw[gp path] (4.371,0.616)--(4.371,0.796);
	\draw[gp path] (4.371,4.263)--(4.371,4.083);
	\node[gp node center] at (4.371,0.308) {$4$};
	\draw[gp path] (5.211,0.616)--(5.211,0.796);
	\draw[gp path] (5.211,4.263)--(5.211,4.083);
	\node[gp node center] at (5.211,0.308) {$5$};
	\draw[gp path] (6.051,0.616)--(6.051,0.796);
	\draw[gp path] (6.051,4.263)--(6.051,4.083);
	\node[gp node center] at (6.051,0.308) {$6$};
	\draw[gp path] (1.012,4.263)--(1.012,0.616)--(6.051,0.616)--(6.051,4.263)--cycle;
	\draw[gp path] (1.012,0.616)--(1.037,0.753)--(1.062,0.889)--(1.088,1.024)--(1.113,1.157)%
	--(1.138,1.286)--(1.163,1.413)--(1.188,1.535)--(1.214,1.652)--(1.239,1.765)--(1.264,1.872)%
	--(1.289,1.973)--(1.314,2.068)--(1.340,2.156)--(1.365,2.238)--(1.390,2.313)--(1.415,2.382)%
	--(1.440,2.444)--(1.466,2.499)--(1.491,2.549)--(1.516,2.592)--(1.541,2.629)--(1.566,2.661)%
	--(1.591,2.688)--(1.617,2.710)--(1.642,2.727)--(1.667,2.740)--(1.692,2.749)--(1.717,2.755)%
	--(1.743,2.758)--(1.768,2.759)--(1.793,2.757)--(1.818,2.753)--(1.843,2.747)--(1.869,2.740)%
	--(1.894,2.732)--(1.919,2.723)--(1.944,2.713)--(1.969,2.703)--(1.995,2.692)--(2.020,2.682)%
	--(2.045,2.671)--(2.070,2.660)--(2.095,2.650)--(2.121,2.640)--(2.146,2.630)--(2.171,2.621)%
	--(2.196,2.611)--(2.221,2.603)--(2.247,2.595)--(2.272,2.587)--(2.297,2.579)--(2.322,2.572)%
	--(2.347,2.566)--(2.373,2.560)--(2.398,2.554)--(2.423,2.548)--(2.448,2.543)--(2.473,2.539)%
	--(2.499,2.534)--(2.524,2.530)--(2.549,2.526)--(2.574,2.522)--(2.599,2.519)--(2.624,2.516)%
	--(2.650,2.513)--(2.675,2.510)--(2.700,2.507)--(2.725,2.505)--(2.750,2.503)--(2.776,2.500)%
	--(2.801,2.498)--(2.826,2.496)--(2.851,2.495)--(2.876,2.493)--(2.902,2.491)--(2.927,2.490)%
	--(2.952,2.488)--(2.977,2.487)--(3.002,2.485)--(3.028,2.484)--(3.053,2.483)--(3.078,2.482)%
	--(3.103,2.480)--(3.128,2.479)--(3.154,2.478)--(3.179,2.477)--(3.204,2.476)--(3.229,2.475)%
	--(3.254,2.474)--(3.280,2.474)--(3.305,2.473)--(3.330,2.472)--(3.355,2.471)--(3.380,2.471)%
	--(3.406,2.470)--(3.431,2.469)--(3.456,2.468)--(3.481,2.468)--(3.506,2.467)--(3.532,2.467)%
	--(3.557,2.466)--(3.582,2.466)--(3.607,2.465)--(3.632,2.464)--(3.657,2.464)--(3.683,2.463)%
	--(3.708,2.463)--(3.733,2.463)--(3.758,2.462)--(3.783,2.462)--(3.809,2.461)--(3.834,2.461)%
	--(3.859,2.460)--(3.884,2.460)--(3.909,2.460)--(3.935,2.459)--(3.960,2.459)--(3.985,2.459)%
	--(4.010,2.458)--(4.035,2.458)--(4.061,2.458)--(4.086,2.457)--(4.111,2.457)--(4.136,2.457)%
	--(4.161,2.456)--(4.187,2.456)--(4.212,2.456)--(4.237,2.456)--(4.262,2.455)--(4.287,2.455)%
	--(4.313,2.455)--(4.338,2.455)--(4.363,2.454)--(4.388,2.454)--(4.413,2.454)--(4.439,2.454)%
	--(4.464,2.453)--(4.489,2.453)--(4.514,2.453)--(4.539,2.453)--(4.564,2.453)--(4.590,2.452)%
	--(4.615,2.452)--(4.640,2.452)--(4.665,2.452)--(4.690,2.452)--(4.716,2.452)--(4.741,2.451)%
	--(4.766,2.451)--(4.791,2.451)--(4.816,2.451)--(4.842,2.451)--(4.867,2.451)--(4.892,2.450)%
	--(4.917,2.450)--(4.942,2.450)--(4.968,2.450)--(4.993,2.450)--(5.018,2.450)--(5.043,2.450)%
	--(5.068,2.450)--(5.094,2.449)--(5.119,2.449)--(5.144,2.449)--(5.169,2.449)--(5.194,2.449)%
	--(5.220,2.449)--(5.245,2.449)--(5.270,2.449)--(5.295,2.448)--(5.320,2.448)--(5.346,2.448)%
	--(5.371,2.448)--(5.396,2.448)--(5.421,2.448)--(5.446,2.448)--(5.472,2.448)--(5.497,2.448)%
	--(5.522,2.448)--(5.547,2.447)--(5.572,2.447)--(5.597,2.447)--(5.623,2.447)--(5.648,2.447)%
	--(5.673,2.447)--(5.698,2.447)--(5.723,2.447)--(5.749,2.447)--(5.774,2.447)--(5.799,2.447)%
	--(5.824,2.447)--(5.849,2.446)--(5.875,2.446)--(5.900,2.446)--(5.925,2.446)--(5.950,2.446)%
	--(5.975,2.446)--(6.001,2.446)--(6.026,2.446)--(6.051,2.446);
	\draw[gp path,dash dot dot] (1.012,2.440)--(1.037,2.440)--(1.062,2.440)--(1.088,2.440)--(1.113,2.440)%
	--(1.138,2.440)--(1.163,2.440)--(1.188,2.440)--(1.214,2.440)--(1.239,2.440)--(1.264,2.440)%
	--(1.289,2.440)--(1.314,2.440)--(1.340,2.440)--(1.365,2.440)--(1.390,2.440)--(1.415,2.440)%
	--(1.440,2.440)--(1.466,2.440)--(1.491,2.440)--(1.516,2.440)--(1.541,2.440)--(1.566,2.440)%
	--(1.591,2.440)--(1.617,2.440)--(1.642,2.440)--(1.667,2.440)--(1.692,2.440)--(1.717,2.440)%
	--(1.743,2.440)--(1.768,2.440)--(1.793,2.440)--(1.818,2.440)--(1.843,2.440)--(1.869,2.440)%
	--(1.894,2.440)--(1.919,2.440)--(1.944,2.440)--(1.969,2.440)--(1.995,2.440)--(2.020,2.440)%
	--(2.045,2.440)--(2.070,2.440)--(2.095,2.440)--(2.121,2.440)--(2.146,2.440)--(2.171,2.440)%
	--(2.196,2.440)--(2.221,2.440)--(2.247,2.440)--(2.272,2.440)--(2.297,2.440)--(2.322,2.440)%
	--(2.347,2.440)--(2.373,2.440)--(2.398,2.440)--(2.423,2.440)--(2.448,2.440)--(2.473,2.440)%
	--(2.499,2.440)--(2.524,2.440)--(2.549,2.440)--(2.574,2.440)--(2.599,2.440)--(2.624,2.440)%
	--(2.650,2.440)--(2.675,2.440)--(2.700,2.440)--(2.725,2.440)--(2.750,2.440)--(2.776,2.440)%
	--(2.801,2.440)--(2.826,2.440)--(2.851,2.440)--(2.876,2.440)--(2.902,2.440)--(2.927,2.440)%
	--(2.952,2.440)--(2.977,2.440)--(3.002,2.440)--(3.028,2.440)--(3.053,2.440)--(3.078,2.440)%
	--(3.103,2.440)--(3.128,2.440)--(3.154,2.440)--(3.179,2.440)--(3.204,2.440)--(3.229,2.440)%
	--(3.254,2.440)--(3.280,2.440)--(3.305,2.440)--(3.330,2.440)--(3.355,2.440)--(3.380,2.440)%
	--(3.406,2.440)--(3.431,2.440)--(3.456,2.440)--(3.481,2.440)--(3.506,2.440)--(3.532,2.440)%
	--(3.557,2.440)--(3.582,2.440)--(3.607,2.440)--(3.632,2.440)--(3.657,2.440)--(3.683,2.440)%
	--(3.708,2.440)--(3.733,2.440)--(3.758,2.440)--(3.783,2.440)--(3.809,2.440)--(3.834,2.440)%
	--(3.859,2.440)--(3.884,2.440)--(3.909,2.440)--(3.935,2.440)--(3.960,2.440)--(3.985,2.440)%
	--(4.010,2.440)--(4.035,2.440)--(4.061,2.440)--(4.086,2.440)--(4.111,2.440)--(4.136,2.440)%
	--(4.161,2.440)--(4.187,2.440)--(4.212,2.440)--(4.237,2.440)--(4.262,2.440)--(4.287,2.440)%
	--(4.313,2.440)--(4.338,2.440)--(4.363,2.440)--(4.388,2.440)--(4.413,2.440)--(4.439,2.440)%
	--(4.464,2.440)--(4.489,2.440)--(4.514,2.440)--(4.539,2.440)--(4.564,2.440)--(4.590,2.440)%
	--(4.615,2.440)--(4.640,2.440)--(4.665,2.440)--(4.690,2.440)--(4.716,2.440)--(4.741,2.440)%
	--(4.766,2.440)--(4.791,2.440)--(4.816,2.440)--(4.842,2.440)--(4.867,2.440)--(4.892,2.440)%
	--(4.917,2.440)--(4.942,2.440)--(4.968,2.440)--(4.993,2.440)--(5.018,2.440)--(5.043,2.440)%
	--(5.068,2.440)--(5.094,2.440)--(5.119,2.440)--(5.144,2.440)--(5.169,2.440)--(5.194,2.440)%
	--(5.220,2.440)--(5.245,2.440)--(5.270,2.440)--(5.295,2.440)--(5.320,2.440)--(5.346,2.440)%
	--(5.371,2.440)--(5.396,2.440)--(5.421,2.440)--(5.446,2.440)--(5.472,2.440)--(5.497,2.440)%
	--(5.522,2.440)--(5.547,2.440)--(5.572,2.440)--(5.597,2.440)--(5.623,2.440)--(5.648,2.440)%
	--(5.673,2.440)--(5.698,2.440)--(5.723,2.440)--(5.749,2.440)--(5.774,2.440)--(5.799,2.440)%
	--(5.824,2.440)--(5.849,2.440)--(5.875,2.440)--(5.900,2.440)--(5.925,2.440)--(5.950,2.440)%
	--(5.975,2.440)--(6.001,2.440)--(6.026,2.440)--(6.051,2.440);
	\draw[gp path,dash dot dot] (1.758,0.616)--(1.758,2.759);
	\draw[gp path] (1.012,4.263)--(1.012,0.616)--(6.051,0.616)--(6.051,4.263)--cycle;
	\gpdefrectangularnode{gp plot 1}{\pgfpoint{1.012cm}{0.616cm}}{\pgfpoint{6.051cm}{4.263cm}}
\end{tikzpicture}

%% file: graph_C-rational.tex
			\begin{tikzpicture}[gnuplot]
	\tikzset{every node/.append style={font={\footnotesize}}}
	\path (0.000,0.000) rectangle (6.604,4.572);
	\gpsetlinetype{gp lt border}
	\gpsetlinewidth{1.00}
	\draw[gp path] (1.012,0.616)--(1.192,0.616);
	\draw[gp path] (6.051,0.616)--(5.871,0.616);
	\node[gp node right] at (0.828,0.616) {$0$};
	\draw[gp path] (1.012,1.528)--(1.192,1.528);
	\draw[gp path] (6.051,1.528)--(5.871,1.528);
	\node[gp node right] at (0.828,1.528) {$0.5$};
	\draw[gp path] (1.012,2.440)--(1.192,2.440);
	\draw[gp path] (6.051,2.440)--(5.871,2.440);
	\node[gp node right] at (0.828,2.440) {$1$};
	\draw[gp path] (1.012,3.351)--(1.192,3.351);
	\draw[gp path] (6.051,3.351)--(5.871,3.351);
	\node[gp node right] at (0.828,3.351) {$1.5$};
	\draw[gp path] (1.012,4.263)--(1.192,4.263);
	\draw[gp path] (6.051,4.263)--(5.871,4.263);
	\node[gp node right] at (0.828,4.263) {$2$};
	\draw[gp path] (1.012,0.616)--(1.012,0.796);
	\draw[gp path] (1.012,4.263)--(1.012,4.083);
	\node[gp node center] at (1.012,0.308) {$0$};
	\draw[gp path] (1.852,0.616)--(1.852,0.796);
	\draw[gp path] (1.852,4.263)--(1.852,4.083);
	\node[gp node center] at (1.852,0.308) {$1$};
	\draw[gp path] (2.692,0.616)--(2.692,0.796);
	\draw[gp path] (2.692,4.263)--(2.692,4.083);
	\node[gp node center] at (2.692,0.308) {$2$};
	\draw[gp path] (3.532,0.616)--(3.532,0.796);
	\draw[gp path] (3.532,4.263)--(3.532,4.083);
	\node[gp node center] at (3.532,0.308) {$3$};
	\draw[gp path] (4.371,0.616)--(4.371,0.796);
	\draw[gp path] (4.371,4.263)--(4.371,4.083);
	\node[gp node center] at (4.371,0.308) {$4$};
	\draw[gp path] (5.211,0.616)--(5.211,0.796);
	\draw[gp path] (5.211,4.263)--(5.211,4.083);
	\node[gp node center] at (5.211,0.308) {$5$};
	\draw[gp path] (6.051,0.616)--(6.051,0.796);
	\draw[gp path] (6.051,4.263)--(6.051,4.083);
	\node[gp node center] at (6.051,0.308) {$6$};
	\draw[gp path] (1.012,4.263)--(1.012,0.616)--(6.051,0.616)--(6.051,4.263)--cycle;
	\draw[gp path] (1.012,0.616)--(1.037,1.013)--(1.062,1.258)--(1.088,1.448)--(1.113,1.603)--(1.138,1.733)%
	--(1.163,1.843)--(1.188,1.937)--(1.214,2.019)--(1.239,2.090)--(1.264,2.152)--(1.289,2.206)%
	--(1.314,2.253)--(1.340,2.295)--(1.365,2.331)--(1.390,2.364)--(1.415,2.392)--(1.440,2.417)%
	--(1.466,2.438)--(1.491,2.457)--(1.516,2.474)--(1.541,2.488)--(1.566,2.501)--(1.591,2.512)%
	--(1.617,2.521)--(1.642,2.529)--(1.667,2.536)--(1.692,2.542)--(1.717,2.547)--(1.743,2.551)%
	--(1.768,2.554)--(1.793,2.556)--(1.818,2.558)--(1.843,2.560)--(1.869,2.561)--(1.894,2.561)%
	--(1.919,2.561)--(1.944,2.561)--(1.969,2.560)--(1.995,2.560)--(2.020,2.559)--(2.045,2.558)%
	--(2.070,2.556)--(2.095,2.555)--(2.121,2.553)--(2.146,2.552)--(2.171,2.550)--(2.196,2.548)%
	--(2.221,2.546)--(2.247,2.544)--(2.272,2.542)--(2.297,2.540)--(2.322,2.538)--(2.347,2.536)%
	--(2.373,2.534)--(2.398,2.532)--(2.423,2.530)--(2.448,2.528)--(2.473,2.526)--(2.499,2.524)%
	--(2.524,2.522)--(2.549,2.520)--(2.574,2.518)--(2.599,2.516)--(2.624,2.514)--(2.650,2.513)%
	--(2.675,2.511)--(2.700,2.509)--(2.725,2.507)--(2.750,2.506)--(2.776,2.504)--(2.801,2.503)%
	--(2.826,2.501)--(2.851,2.500)--(2.876,2.498)--(2.902,2.497)--(2.927,2.495)--(2.952,2.494)%
	--(2.977,2.493)--(3.002,2.491)--(3.028,2.490)--(3.053,2.489)--(3.078,2.488)--(3.103,2.487)%
	--(3.128,2.485)--(3.154,2.484)--(3.179,2.483)--(3.204,2.482)--(3.229,2.481)--(3.254,2.480)%
	--(3.280,2.479)--(3.305,2.478)--(3.330,2.477)--(3.355,2.477)--(3.380,2.476)--(3.406,2.475)%
	--(3.431,2.474)--(3.456,2.473)--(3.481,2.473)--(3.506,2.472)--(3.532,2.471)--(3.557,2.470)%
	--(3.582,2.470)--(3.607,2.469)--(3.632,2.468)--(3.657,2.468)--(3.683,2.467)--(3.708,2.467)%
	--(3.733,2.466)--(3.758,2.466)--(3.783,2.465)--(3.809,2.464)--(3.834,2.464)--(3.859,2.463)%
	--(3.884,2.463)--(3.909,2.463)--(3.935,2.462)--(3.960,2.462)--(3.985,2.461)--(4.010,2.461)%
	--(4.035,2.460)--(4.061,2.460)--(4.086,2.460)--(4.111,2.459)--(4.136,2.459)--(4.161,2.459)%
	--(4.187,2.458)--(4.212,2.458)--(4.237,2.458)--(4.262,2.457)--(4.287,2.457)--(4.313,2.457)%
	--(4.338,2.456)--(4.363,2.456)--(4.388,2.456)--(4.413,2.455)--(4.439,2.455)--(4.464,2.455)%
	--(4.489,2.455)--(4.514,2.454)--(4.539,2.454)--(4.564,2.454)--(4.590,2.454)--(4.615,2.454)%
	--(4.640,2.453)--(4.665,2.453)--(4.690,2.453)--(4.716,2.453)--(4.741,2.452)--(4.766,2.452)%
	--(4.791,2.452)--(4.816,2.452)--(4.842,2.452)--(4.867,2.452)--(4.892,2.451)--(4.917,2.451)%
	--(4.942,2.451)--(4.968,2.451)--(4.993,2.451)--(5.018,2.451)--(5.043,2.450)--(5.068,2.450)%
	--(5.094,2.450)--(5.119,2.450)--(5.144,2.450)--(5.169,2.450)--(5.194,2.450)--(5.220,2.449)%
	--(5.245,2.449)--(5.270,2.449)--(5.295,2.449)--(5.320,2.449)--(5.346,2.449)--(5.371,2.449)%
	--(5.396,2.449)--(5.421,2.448)--(5.446,2.448)--(5.472,2.448)--(5.497,2.448)--(5.522,2.448)%
	--(5.547,2.448)--(5.572,2.448)--(5.597,2.448)--(5.623,2.448)--(5.648,2.448)--(5.673,2.447)%
	--(5.698,2.447)--(5.723,2.447)--(5.749,2.447)--(5.774,2.447)--(5.799,2.447)--(5.824,2.447)%
	--(5.849,2.447)--(5.875,2.447)--(5.900,2.447)--(5.925,2.447)--(5.950,2.446)--(5.975,2.446)%
	--(6.001,2.446)--(6.026,2.446)--(6.051,2.446);
	\draw[gp path,dash dot dot] (1.012,2.440)--(1.037,2.440)--(1.062,2.440)--(1.088,2.440)--(1.113,2.440)%
	--(1.138,2.440)--(1.163,2.440)--(1.188,2.440)--(1.214,2.440)--(1.239,2.440)--(1.264,2.440)%
	--(1.289,2.440)--(1.314,2.440)--(1.340,2.440)--(1.365,2.440)--(1.390,2.440)--(1.415,2.440)%
	--(1.440,2.440)--(1.466,2.440)--(1.491,2.440)--(1.516,2.440)--(1.541,2.440)--(1.566,2.440)%
	--(1.591,2.440)--(1.617,2.440)--(1.642,2.440)--(1.667,2.440)--(1.692,2.440)--(1.717,2.440)%
	--(1.743,2.440)--(1.768,2.440)--(1.793,2.440)--(1.818,2.440)--(1.843,2.440)--(1.869,2.440)%
	--(1.894,2.440)--(1.919,2.440)--(1.944,2.440)--(1.969,2.440)--(1.995,2.440)--(2.020,2.440)%
	--(2.045,2.440)--(2.070,2.440)--(2.095,2.440)--(2.121,2.440)--(2.146,2.440)--(2.171,2.440)%
	--(2.196,2.440)--(2.221,2.440)--(2.247,2.440)--(2.272,2.440)--(2.297,2.440)--(2.322,2.440)%
	--(2.347,2.440)--(2.373,2.440)--(2.398,2.440)--(2.423,2.440)--(2.448,2.440)--(2.473,2.440)%
	--(2.499,2.440)--(2.524,2.440)--(2.549,2.440)--(2.574,2.440)--(2.599,2.440)--(2.624,2.440)%
	--(2.650,2.440)--(2.675,2.440)--(2.700,2.440)--(2.725,2.440)--(2.750,2.440)--(2.776,2.440)%
	--(2.801,2.440)--(2.826,2.440)--(2.851,2.440)--(2.876,2.440)--(2.902,2.440)--(2.927,2.440)%
	--(2.952,2.440)--(2.977,2.440)--(3.002,2.440)--(3.028,2.440)--(3.053,2.440)--(3.078,2.440)%
	--(3.103,2.440)--(3.128,2.440)--(3.154,2.440)--(3.179,2.440)--(3.204,2.440)--(3.229,2.440)%
	--(3.254,2.440)--(3.280,2.440)--(3.305,2.440)--(3.330,2.440)--(3.355,2.440)--(3.380,2.440)%
	--(3.406,2.440)--(3.431,2.440)--(3.456,2.440)--(3.481,2.440)--(3.506,2.440)--(3.532,2.440)%
	--(3.557,2.440)--(3.582,2.440)--(3.607,2.440)--(3.632,2.440)--(3.657,2.440)--(3.683,2.440)%
	--(3.708,2.440)--(3.733,2.440)--(3.758,2.440)--(3.783,2.440)--(3.809,2.440)--(3.834,2.440)%
	--(3.859,2.440)--(3.884,2.440)--(3.909,2.440)--(3.935,2.440)--(3.960,2.440)--(3.985,2.440)%
	--(4.010,2.440)--(4.035,2.440)--(4.061,2.440)--(4.086,2.440)--(4.111,2.440)--(4.136,2.440)%
	--(4.161,2.440)--(4.187,2.440)--(4.212,2.440)--(4.237,2.440)--(4.262,2.440)--(4.287,2.440)%
	--(4.313,2.440)--(4.338,2.440)--(4.363,2.440)--(4.388,2.440)--(4.413,2.440)--(4.439,2.440)%
	--(4.464,2.440)--(4.489,2.440)--(4.514,2.440)--(4.539,2.440)--(4.564,2.440)--(4.590,2.440)%
	--(4.615,2.440)--(4.640,2.440)--(4.665,2.440)--(4.690,2.440)--(4.716,2.440)--(4.741,2.440)%
	--(4.766,2.440)--(4.791,2.440)--(4.816,2.440)--(4.842,2.440)--(4.867,2.440)--(4.892,2.440)%
	--(4.917,2.440)--(4.942,2.440)--(4.968,2.440)--(4.993,2.440)--(5.018,2.440)--(5.043,2.440)%
	--(5.068,2.440)--(5.094,2.440)--(5.119,2.440)--(5.144,2.440)--(5.169,2.440)--(5.194,2.440)%
	--(5.220,2.440)--(5.245,2.440)--(5.270,2.440)--(5.295,2.440)--(5.320,2.440)--(5.346,2.440)%
	--(5.371,2.440)--(5.396,2.440)--(5.421,2.440)--(5.446,2.440)--(5.472,2.440)--(5.497,2.440)%
	--(5.522,2.440)--(5.547,2.440)--(5.572,2.440)--(5.597,2.440)--(5.623,2.440)--(5.648,2.440)%
	--(5.673,2.440)--(5.698,2.440)--(5.723,2.440)--(5.749,2.440)--(5.774,2.440)--(5.799,2.440)%
	--(5.824,2.440)--(5.849,2.440)--(5.875,2.440)--(5.900,2.440)--(5.925,2.440)--(5.950,2.440)%
	--(5.975,2.440)--(6.001,2.440)--(6.026,2.440)--(6.051,2.440);
	\draw[gp path,dash dot dot] (1.916,0.616)--(1.916,2.561);
	\draw[gp path] (1.012,4.263)--(1.012,0.616)--(6.051,0.616)--(6.051,4.263)--cycle;
	\gpdefrectangularnode{gp plot 1}{\pgfpoint{1.012cm}{0.616cm}}{\pgfpoint{6.051cm}{4.263cm}}
\end{tikzpicture}

%% file: graph_C-exp.tex
			\begin{tikzpicture}[gnuplot]
	\tikzset{every node/.append style={font={\footnotesize}}}
	\path (0.000,0.000) rectangle (6.604,4.572);
	\gpcolor{color=gp lt color border}
	\gpsetlinetype{gp lt border}
	\gpsetdashtype{gp dt solid}
	\gpsetlinewidth{1.00}
	\draw[gp path] (1.012,0.616)--(1.192,0.616);
	\draw[gp path] (6.051,0.616)--(5.871,0.616);
	\node[gp node right] at (0.828,0.616) {$0$};
	\draw[gp path] (1.012,1.528)--(1.192,1.528);
	\draw[gp path] (6.051,1.528)--(5.871,1.528);
	\node[gp node right] at (0.828,1.528) {$0.5$};
	\draw[gp path] (1.012,2.440)--(1.192,2.440);
	\draw[gp path] (6.051,2.440)--(5.871,2.440);
	\node[gp node right] at (0.828,2.440) {$1$};
	\draw[gp path] (1.012,3.351)--(1.192,3.351);
	\draw[gp path] (6.051,3.351)--(5.871,3.351);
	\node[gp node right] at (0.828,3.351) {$1.5$};
	\draw[gp path] (1.012,4.263)--(1.192,4.263);
	\draw[gp path] (6.051,4.263)--(5.871,4.263);
	\node[gp node right] at (0.828,4.263) {$2$};
	\draw[gp path] (1.012,0.616)--(1.012,0.796);
	\draw[gp path] (1.012,4.263)--(1.012,4.083);
	\node[gp node center] at (1.012,0.308) {$0$};
	\draw[gp path] (1.852,0.616)--(1.852,0.796);
	\draw[gp path] (1.852,4.263)--(1.852,4.083);
	\node[gp node center] at (1.852,0.308) {$1$};
	\draw[gp path] (2.692,0.616)--(2.692,0.796);
	\draw[gp path] (2.692,4.263)--(2.692,4.083);
	\node[gp node center] at (2.692,0.308) {$2$};
	\draw[gp path] (3.532,0.616)--(3.532,0.796);
	\draw[gp path] (3.532,4.263)--(3.532,4.083);
	\node[gp node center] at (3.532,0.308) {$3$};
	\draw[gp path] (4.371,0.616)--(4.371,0.796);
	\draw[gp path] (4.371,4.263)--(4.371,4.083);
	\node[gp node center] at (4.371,0.308) {$4$};
	\draw[gp path] (5.211,0.616)--(5.211,0.796);
	\draw[gp path] (5.211,4.263)--(5.211,4.083);
	\node[gp node center] at (5.211,0.308) {$5$};
	\draw[gp path] (6.051,0.616)--(6.051,0.796);
	\draw[gp path] (6.051,4.263)--(6.051,4.083);
	\node[gp node center] at (6.051,0.308) {$6$};
	\draw[gp path] (1.012,4.263)--(1.012,0.616)--(6.051,0.616)--(6.051,4.263)--cycle;
	\draw[gp path] (1.012,0.616)--(1.063,0.959)--(1.114,1.281)--(1.165,1.566)--(1.216,1.804)--(1.266,1.995)%
	--(1.317,2.144)--(1.368,2.257)--(1.419,2.340)--(1.470,2.402)--(1.521,2.446)--(1.572,2.478)%
	--(1.623,2.500)--(1.674,2.516)--(1.725,2.526)--(1.775,2.533)--(1.826,2.537)--(1.877,2.539)%
	--(1.928,2.539)--(1.979,2.539)--(2.030,2.538)--(2.081,2.536)--(2.132,2.534)--(2.183,2.532)%
	--(2.234,2.529)--(2.284,2.527)--(2.335,2.524)--(2.386,2.521)--(2.437,2.519)--(2.488,2.516)%
	--(2.539,2.514)--(2.590,2.511)--(2.641,2.509)--(2.692,2.507)--(2.743,2.505)--(2.793,2.503)%
	--(2.844,2.501)--(2.895,2.499)--(2.946,2.497)--(2.997,2.495)--(3.048,2.494)--(3.099,2.492)%
	--(3.150,2.490)--(3.201,2.489)--(3.252,2.488)--(3.302,2.486)--(3.353,2.485)--(3.404,2.484)%
	--(3.455,2.482)--(3.506,2.481)--(3.557,2.480)--(3.608,2.479)--(3.659,2.478)--(3.710,2.477)%
	--(3.761,2.476)--(3.811,2.475)--(3.862,2.474)--(3.913,2.473)--(3.964,2.473)--(4.015,2.472)%
	--(4.066,2.471)--(4.117,2.470)--(4.168,2.470)--(4.219,2.469)--(4.270,2.468)--(4.320,2.468)%
	--(4.371,2.467)--(4.422,2.466)--(4.473,2.466)--(4.524,2.465)--(4.575,2.465)--(4.626,2.464)%
	--(4.677,2.464)--(4.728,2.463)--(4.779,2.463)--(4.829,2.462)--(4.880,2.462)--(4.931,2.461)%
	--(4.982,2.461)--(5.033,2.461)--(5.084,2.460)--(5.135,2.460)--(5.186,2.459)--(5.237,2.459)%
	--(5.288,2.459)--(5.338,2.458)--(5.389,2.458)--(5.440,2.458)--(5.491,2.457)--(5.542,2.457)%
	--(5.593,2.457)--(5.644,2.456)--(5.695,2.456)--(5.746,2.456)--(5.797,2.456)--(5.847,2.455)%
	--(5.898,2.455)--(5.949,2.455)--(6.000,2.455)--(6.051,2.454);
	\draw[gp path,dash dot dot] (1.012,2.440)--(1.063,2.440)--(1.114,2.440)--(1.165,2.440)--(1.216,2.440)%
	--(1.266,2.440)--(1.317,2.440)--(1.368,2.440)--(1.419,2.440)--(1.470,2.440)--(1.521,2.440)%
	--(1.572,2.440)--(1.623,2.440)--(1.674,2.440)--(1.725,2.440)--(1.775,2.440)--(1.826,2.440)%
	--(1.877,2.440)--(1.928,2.440)--(1.979,2.440)--(2.030,2.440)--(2.081,2.440)--(2.132,2.440)%
	--(2.183,2.440)--(2.234,2.440)--(2.284,2.440)--(2.335,2.440)--(2.386,2.440)--(2.437,2.440)%
	--(2.488,2.440)--(2.539,2.440)--(2.590,2.440)--(2.641,2.440)--(2.692,2.440)--(2.743,2.440)%
	--(2.793,2.440)--(2.844,2.440)--(2.895,2.440)--(2.946,2.440)--(2.997,2.440)--(3.048,2.440)%
	--(3.099,2.440)--(3.150,2.440)--(3.201,2.440)--(3.252,2.440)--(3.302,2.440)--(3.353,2.440)%
	--(3.404,2.440)--(3.455,2.440)--(3.506,2.440)--(3.557,2.440)--(3.608,2.440)--(3.659,2.440)%
	--(3.710,2.440)--(3.761,2.440)--(3.811,2.440)--(3.862,2.440)--(3.913,2.440)--(3.964,2.440)%
	--(4.015,2.440)--(4.066,2.440)--(4.117,2.440)--(4.168,2.440)--(4.219,2.440)--(4.270,2.440)%
	--(4.320,2.440)--(4.371,2.440)--(4.422,2.440)--(4.473,2.440)--(4.524,2.440)--(4.575,2.440)%
	--(4.626,2.440)--(4.677,2.440)--(4.728,2.440)--(4.779,2.440)--(4.829,2.440)--(4.880,2.440)%
	--(4.931,2.440)--(4.982,2.440)--(5.033,2.440)--(5.084,2.440)--(5.135,2.440)--(5.186,2.440)%
	--(5.237,2.440)--(5.288,2.440)--(5.338,2.440)--(5.389,2.440)--(5.440,2.440)--(5.491,2.440)%
	--(5.542,2.440)--(5.593,2.440)--(5.644,2.440)--(5.695,2.440)--(5.746,2.440)--(5.797,2.440)%
	--(5.847,2.440)--(5.898,2.440)--(5.949,2.440)--(6.000,2.440)--(6.051,2.440);
	\draw[gp path,dash dot dot] (1.927,0.616)--(1.927,2.539);
	\gpcolor{color=gp lt color border}
	\draw[gp path] (1.012,4.263)--(1.012,0.616)--(6.051,0.616)--(6.051,4.263)--cycle;
	\gpdefrectangularnode{gp plot 1}{\pgfpoint{1.012cm}{0.616cm}}{\pgfpoint{6.051cm}{4.263cm}}
\end{tikzpicture}

%% file: d-dim_Dirac_arxiv_v3.bbl
\begin{thebibliography}{25}
\providecommand{\natexlab}[1]{#1}
\providecommand{\url}[1]{\texttt{#1}}
\providecommand{\doi}[1]{doi:\href{https://doi.org/#1}{#1}}
\providecommand{\MR}[1]{\href{https://mathscinet.ams.org/mathscinet-getitem?mr=#1}{#1}}
\providecommand{\Zbl}[1]{Zbl \href{https://zbmath.org/?q=an:#1}{#1}}

\bibitem[Atkinson and Han(2012)]{AH2012}
Kendall Atkinson and Weimin Han.
\newblock \emph{Spherical Harmonics and Approximations on the Unit Sphere: An
  Introduction}.
\newblock Springer Berlin Heidelberg, 2012.
\newblock ISBN 9783642259838.
\newblock \doi{10.1007/978-3-642-25983-8}.
\newblock \MR{2934227}.

\bibitem[Baricz and Ponnusamy(2013)]{BP2012}
\'{A}rp\'{a}d Baricz and Saminathan Ponnusamy.
\newblock On {Tur\'{a}n} type inequalities for modified {Bessel} functions.
\newblock \emph{Proceedings of the American Mathematical Society}, 141\penalty0
  (2):\penalty0 523--532, 2013.
\newblock \doi{10.1090/s0002-9939-2012-11325-5}.
\newblock \MR{2996956}.

\bibitem[Ben-Artzi and Klainerman(1992)]{BK1992}
Matania Ben-Artzi and Sergiu Klainerman.
\newblock Decay and regularity for the {S}chr\"{o}dinger equation.
\newblock \emph{Journal d'Analyse Math\'{e}matique}, 58\penalty0 (1):\penalty0
  25--37, 1992.
\newblock \doi{10.1007/bf02790356}.
\newblock \MR{1226935}.

\bibitem[Ben-Artzi and Umeda(2021)]{BU2021}
Matania Ben-Artzi and Tomio Umeda.
\newblock Spectral theory of first-order systems: from crystals to {D}irac
  operators.
\newblock \emph{Reviews in Mathematical Physics}, 33\penalty0 (05):\penalty0
  Article No.~2150014, 2021.
\newblock \doi{10.1142/s0129055x21500148}.
\newblock \MR{4274368}.

\bibitem[Ben-Artzi et~al.(2020)Ben-Artzi, Ruzhansky, and Sugimoto]{BRS2020}
Matania Ben-Artzi, Michael Ruzhansky, and Mitsuru Sugimoto.
\newblock Spectral identities and smoothing estimates for evolution operators.
\newblock \emph{Advances in Differential Equations}, 25\penalty0
  (11/12):\penalty0 627--650, 2020.
\newblock \doi{10.57262/ade/1605150115}.
\newblock \MR{4173171}.

\bibitem[Bez and Sugimoto(2017)]{BS2017}
Neal Bez and Mitsuru Sugimoto.
\newblock Optimal constants and extremisers for some smoothing estimates.
\newblock \emph{Journal d'Analyse Math\'{e}matique}, 131\penalty0 (1):\penalty0
  159--187, 2017.
\newblock \doi{10.1007/s11854-017-0005-8}.
\newblock \MR{3631453}.

\bibitem[Bez et~al.(2015)Bez, Saito, and Sugimoto]{BSS2015}
Neal Bez, Hiroki Saito, and Mitsuru Sugimoto.
\newblock Applications of the {F}unk-{H}ecke theorem to smoothing and trace
  estimates.
\newblock \emph{Advances in Mathematics}, 285:\penalty0 1767--1795, 2015.
\newblock \doi{10.1016/j.aim.2015.08.025}.
\newblock \MR{3406541}.

\bibitem[Chihara(2002)]{Chi2002}
Hiroyuki Chihara.
\newblock Smoothing effects of dispersive pseudodifferential equations.
\newblock \emph{Communications in Partial Differential Equations}, 27\penalty0
  (9-10):\penalty0 1953--2005, 2002.
\newblock \doi{10.1081/pde-120016133}.
\newblock \MR{1941663}.

\bibitem[{\relax DLMF}()]{DLMF}
{\relax DLMF}.
\newblock {NIST Digital Library of Mathematical Functions}.
\newblock Release 1.2.7, June 15, 2026.
\newblock URL \url{https://dlmf.nist.gov/}.
\newblock F.~W.~J. Olver, A.~B. {Olde Daalhuis}, D.~W. Lozier, B.~I. Schneider,
  R.~F. Boisvert, C.~W. Clark, B.~R. Miller, B.~V. Saunders, H.~S. Cohl, and
  M.~A. McClain, eds.

\bibitem[Gradshteyn and Ryzhik(2014)]{GR2014}
I.~S. Gradshteyn and I.~M. Ryzhik.
\newblock \emph{{Table of Integrals, Series, and Products}}.
\newblock Elsevier, Boston, 8th edition, 2014.
\newblock ISBN 978-0-12-384933-5.
\newblock \doi{10.1016/c2010-0-64839-5}.
\newblock \MR{3307944}.

\bibitem[Hartman(1977)]{Har1977}
Philip Hartman.
\newblock On the products of solutions of second order disconjugate
  differential equations and the {Whittaker} differential equation.
\newblock \emph{SIAM Journal on Mathematical Analysis}, 8\penalty0
  (3):\penalty0 558--571, 1977.
\newblock \doi{10.1137/0508044}.
\newblock \MR{0435510}.

\bibitem[Ikoma(2022)]{Iko2022}
Makoto Ikoma.
\newblock Optimal constants of smoothing estimates for the 2{D} {D}irac
  equation.
\newblock \emph{Journal of Fourier Analysis and Applications}, 28\penalty0
  (4):\penalty0 Article No.~57, 2022.
\newblock \doi{10.1007/s00041-022-09950-6}.
\newblock \MR{4447305}.

\bibitem[Ikoma and Suzuki(2025)]{IkS2025}
Makoto Ikoma and Soichiro Suzuki.
\newblock Optimal constants of smoothing estimates for the {3D Dirac} equation.
\newblock \emph{Analysis and Mathematical Physics}, 15\penalty0 (4):\penalty0
  Article No.~90, 2025.
\newblock \doi{10.1007/s13324-025-01083-5}.
\newblock \MR{4921436}.

\bibitem[Kalf et~al.(2019)Kalf, Okaji, and Yamada]{KOY2019}
Hubert Kalf, Takashi Okaji, and Osanobu Yamada.
\newblock Explicit uniform bounds on integrals of {B}essel functions and trace
  theorems for {F}ourier transforms.
\newblock \emph{Mathematische Nachrichten}, 292\penalty0 (1):\penalty0
  106--120, 2019.
\newblock \doi{10.1002/mana.201700326}.
\newblock \MR{3909223}.

\bibitem[Kato(1966)]{Kat1966}
Tosio Kato.
\newblock Wave operators and similarity for some non-selfadjoint operators.
\newblock \emph{Mathematische Annalen}, 162\penalty0 (2):\penalty0 258--279,
  1966.
\newblock \doi{10.1007/bf01360915}.
\newblock \MR{0190801}.

\bibitem[Kato and Yajima(1989)]{KY1989}
Tosio Kato and Kenji Yajima.
\newblock Some examples of smooth operators and the associated smoothing
  effect.
\newblock \emph{Reviews in Mathematical Physics}, 1\penalty0 (4):\penalty0
  481--496, 1989.
\newblock \doi{10.1142/s0129055x89000171}.
\newblock \MR{1061120}.

\bibitem[Segura(2021)]{Seg2021}
J.~Segura.
\newblock Monotonicity properties for ratios and products of modified {Bessel}
  functions and sharp trigonometric bounds.
\newblock \emph{Results in Mathematics}, 76\penalty0 (4):\penalty0 Article
  No.~221, 2021.
\newblock \doi{10.1007/s00025-021-01531-1}.
\newblock \MR{4328474}.

\bibitem[Simon(1992)]{Sim1992}
Barry Simon.
\newblock Best constants in some operator smoothness estimates.
\newblock \emph{Journal of Functional Analysis}, 107\penalty0 (1):\penalty0
  66--71, 1992.
\newblock \doi{10.1016/0022-1236(92)90100-w}.
\newblock \MR{1165866}.

\bibitem[Sugimoto(1998)]{Sug1998}
Mitsuru Sugimoto.
\newblock Global smoothing properties of generalized {S}chr\"{o}dinger
  equations.
\newblock \emph{Journal d'Analyse Math\'{e}matique}, 76\penalty0 (1):\penalty0
  191--204, 1998.
\newblock \doi{10.1007/bf02786935}.
\newblock \MR{1676995}.

\bibitem[Szeg\H{o}(1975)]{Sze1975}
G\'{a}bor Szeg\H{o}.
\newblock \emph{Orthogonal polynomials}.
\newblock American Mathematical Society Colloquium Publications, Vol. XXIII.
  American Mathematical Society, Providence, RI, fourth edition, 1975.
\newblock ISBN 978-0-8218-1023-1.
\newblock \doi{10.1090/coll/023}.
\newblock \MR{0372517}.

\bibitem[Vilela(2001)]{Vil2001}
Mari~Cruz Vilela.
\newblock Regularity of solutions to the free {S}chr\"{o}dinger equation with
  radial initial data.
\newblock \emph{Illinois Journal of Mathematics}, 45\penalty0 (2):\penalty0
  361--370, 2001.
\newblock \doi{10.1215/ijm/1258138345}.
\newblock \MR{1878609}.

\bibitem[Walther(1999)]{Wal1999}
Bj\"{o}rn~G. Walther.
\newblock A sharp weighted {$L^2$}-estimate for the solution to the
  time-dependent {S}chr\"{o}dinger equation.
\newblock \emph{Arkiv f\"or Matematik}, 37\penalty0 (2):\penalty0 381--393,
  1999.
\newblock \doi{10.1007/bf02412222}.
\newblock \MR{1714761}.

\bibitem[Walther(2000)]{Wal2000}
Bj\"{o}rn~G. Walther.
\newblock Homogeneous estimates for oscillatory integrals.
\newblock \emph{Acta Mathematica Universitatis Comenianae. New Series},
  69:\penalty0 151--171, 2000.
\newblock URL \url{http://eudml.org/doc/121290}.
\newblock \MR{1819518}.

\bibitem[Walther(2002)]{Wal2002}
Bj\"{o}rn~G. Walther.
\newblock Regularity, decay, and best constants for dispersive equations.
\newblock \emph{Journal of Functional Analysis}, 189\penalty0 (2):\penalty0
  325--335, 2002.
\newblock \doi{10.1006/jfan.2001.3863}.
\newblock \MR{1891852}.

\bibitem[Watanabe(1991)]{Wat1991}
Kazuo Watanabe.
\newblock Smooth perturbations of the self-adjoint operator
  {$|\Delta|^{\alpha/2}$}.
\newblock \emph{Tokyo Journal of Mathematics}, 14\penalty0 (1):\penalty0
  239--250, 1991.
\newblock \doi{10.3836/tjm/1270130504}.
\newblock \MR{1108171}.

\end{thebibliography}
